\documentclass[3p,times]{elsarticle}

\pdfoutput=1

\usepackage{amssymb}

\usepackage[figuresright]{rotating}
\usepackage{amsmath, amscd}
\usepackage{wrapfig}
\usepackage{tikz}

\newtheorem{theorem}{Theorem}
\newtheorem{lemma}[theorem]{Lemma}
\newtheorem{proposition}{Proposition}
\newproof{proof}{Proof}
\newdefinition{remark}{Remark}
\newdefinition{assumption}{Assumption}

\begin{document}

\begin{frontmatter}




\title{Variational integrators for electric circuits\tnoteref{t1}}
\tnotetext[t1]{Preprint submitted to Journal of Computational Physics}


\author[UPB]{Sina Ober-Bl\"obaum}
\author[CIT1]{Molei Tao}
\author[CIT2]{Mulin Cheng}
\author[CIT1,CIT2]{Houman Owhadi}
\author[CIT1,CIT2,cor2]{Jerrold E.~Marsden}
\address[UPB]{Computational Dynamics and Optimal Control, University of Paderborn, Germany}
\address[CIT1]{Control and Dynamical Systems, California Institute of Technology, USA}
\address[CIT2]{Applied and Computational Mathematics, California Institute of Technology, USA}
\cortext[cor1]{ deceased}
\begin{abstract}
In this contribution, we develop a variational integrator for the simulation of (stochastic and multiscale) electric circuits.
When considering the dynamics of an electrical circuit, one is faced with three special situations: 1. The system involves external (control) forcing through external (controlled) voltage sources and resistors. 2. The system is constrained via the Kirchhoff current (KCL) and voltage laws (KVL). 3. The Lagrangian is degenerate.
Based on a geometric setting, an appropriate variational formulation is presented to model the circuit from which the equations of motion are derived. A time-discrete variational formulation provides an iteration scheme for the simulation of the electric circuit. Dependent on the discretization, the intrinsic degeneracy of the system can be canceled for the discrete variational scheme. In this way, a variational integrator is constructed that gains several advantages compared to standard integration tools for circuits; in particular, a comparison to BDF methods (which are usually the method of choice for the simulation of electric circuits) shows that even for simple LCR circuits, a better energy behavior and frequency spectrum preservation can be observed using the developed variational integrator.

\end{abstract}


\end{frontmatter}


\section{Introduction}

Variational integrators have mainly been developed and used for a wide variety of mechanical systems.
However, real-life systems are generally not of purely mechanical character.
In fact, more and more systems become multidisciplinary in the sense, that not only mechanical parts, but also electric and software subsystems are involved, resulting in mechatronic systems.
Since the integration of these systems with a unified simulation tool is desirable, the aim of this work is to extend the applicability of variational integrators to mechatronic systems. In particular, as the first step towards a unified simulation, we develop a variational integrator for the simulation of electric circuits.

\paragraph{Overview}

Variational integrators \cite{MaWe01} are based on a discrete variational formulation of the underlying system, e.g.\ based on a discrete version of Hamilton's principle for conservative mechanical systems.
 The resulting integrators given by the discrete Euler-Lagrange equations are symplectic and momentum-preserving and have an excellent long-time energy behavior.
 Choosing different variational formulations (e.g.~Hamilton, Lagrange-d'Alembert, Hamilton-Pontryagin, etc.), variational integrators have been developed for classical conservative mechanical systems (for an overview see \cite{LMOW04,LMOWe04}), forced \cite{Kane00} and controlled \cite{DMOC} systems, constrained systems (holonomic \cite{leyendecker07-2, DMOCC} and nonholonomic systems \cite{KoMa2010}),  nonsmooth systems \cite{FMOW03}, stochastic systems \cite{BRO08}, and multiscale systems \cite{TaOwMa2010}.
Most of these systems share the assumption, that they are non-degenerate, i.e.~the Legendre transformation of the corresponding Lagrangian is a diffeomorphism.
Applying Hamilton's principle to a regular Lagrangian system, the resulting Euler-Lagrange equations are ordinary differential equations of second order and equivalent to Hamilton's equations.
 
The Lagrangian formulation for LC circuits is based on
the electric and magnetic energies in the circuit and the interconnection constraints expressed in the Kirchhoff laws. 
There exist a large variety of different approaches for a Lagrangian or Hamiltonian formulation of electric circuits (see e.g.~\cite{Chua74,KwMaBa82,ClSch2003,BL89,Sz79} and references therein). All of theses authors treat the question of which choice of the Lagrangian coordinates and derivatives is the most appropriate one. Several settings have been proposed and analyzed, e.g.~a variational formulation based on capacitor charges and currents, on inductor fluxes and voltages, and a combination of both settings, as well as formulations based on linear combinations of the charges and flux linkages.
Typically, one wants to find a set of generalized coordinates, such that the resulting Lagrangian is non-degenerate. However, within such a formulation, the variables are not easily interpretable in terms of original terms of a circuit.

A recently-considered alternative formulation is based on a redundant set of coordinates resulting in a Lagrangian system for which the Lagrangian is degenerate.
For a degenerate Lagrangian system, i.e.~the Legendre transform is not invertible, the Euler-Lagrange equations involve additional hidden algebraic constraints. Then, the equations do not have a unique solution, and additional constraints are required for unique solvability of the system. For the circuit case, these are provided by the Kirchhoff Current Law (KCL). 
From a geometric point of view, the KCL provides a constraint distribution that induces a \emph{Dirac structure} for the degenerate system. The associated system is then denoted by an \emph{implicit Lagrangian system}.
In \cite{YoMa2006a} and \cite{YoMa2006c}, it was shown that nonholonomic mechanical systems and LC circuits as degenerate Lagrangian systems can be formulated in the context of induced Dirac structures and associated implicit Lagrangian systems. The variational structure of an implicit Lagrange system is given in the context of the Hamiltonian-Pontryagin-d'Alembert principle, as shown in \cite{YoMa2006b}. 
The resulting Euler-Lagrange equations are called the \emph{implicit Euler-Lagrange equations} \cite{YoMa2006a,YoMa2006b,MvdS02}, which are 
semi-explicit differential-algebraic equations that consist of a system of first order differential equations and an additional algebraic equation that constrains the image of the Legendre transformation (called the \emph{set of primary constraints}). Thus, the modeling of electric circuits involves both primary
constraints as well as constraints coming from Kirchhoff's laws.
In \cite{JYM10}, an extension towards the interconnection of implicit Lagrange systems for electric circuits is demonstrated.
For completeness, we have to mention that the corresponding notion of implicit Hamiltonian systems and implicit Hamiltonian equations was developed earlier by \cite{Bl00,MvSB95,BC97}.
An intrinsic Hamiltonian formulation of dynamics of LC circuits as well as interconnections of Dirac structures have been developed, e.g.~in \cite{MvSB95} and \cite{CvS07}, respectively.

There are only a few works dealing with the variational simulation of degenerate systems, e.g.~in \cite{RoMa02}, variational integrators with application to point vertices as a special case of degenerate Lagrangian system are developed.
Although there exists a variety of different variational formulations for electric circuits, variational integrators for their simulation have not been concretely investigated and applied thus far.
In \cite{LeOh2010}, a framework for the description of the discrete analogues of implicit Lagrangian and Hamiltonian systems is proposed. This framework is the foundation for the development of an integration scheme. However, no concrete simulation scenarios have yet been performed.
Furthermore, the discrete formulation of the variational principle is slightly different from the approach presented in this work, thus resulting in a different scheme.

\paragraph{Contribution}
In this work, we present a unified variational framework for the modeling and simulation of electric circuits.
The focus of our analysis is on the case of ideal linear circuit elements, consisting of inductors, capacitors, resistors and voltage sources. However, this is not a restriction of this approach, and the variational integrators can also be developed for nonlinear circuits, which is left for future work.
A geometric formulation of the different possible state spaces for a circuit model is introduced.
This geometric view point forms the basis for a variational formulation. Rather than dealing with Dirac structures, we work directly with the corresponding variational principle, where we follow the approach introduced in \cite{YoMa2006b}.
When considering the dynamics of an electric circuit, one is faced with three specific situations that lead to a special treatment within the variational formulation and thus the construction of appropriate variational integrators: 1.~The system involves external (control) forcing through external (controlled) voltage sources. 2.~The system is constrained via the Kirchhoff current (KCL) and voltage laws (KVL). 3.~The Lagrangian is degenerate leading to primary constraints.
For the treatment of forced systems, the Lagrange-d'Alembert principle is the principle of choice.
Involving constraints, one has to consider constrained variations resulting in a constrained principle. The degeneracy requires the use of the Pontryagin version; thus, the principle of choice is the constrained {\it Lagrange-d'Alembert-Pontryagin principle} \cite{YoMa2006b}. 
Two variational formulations are considered: First, a constrained variational formulation is introduced for which the KCL constraints are explicitly given as algebraic constraints, whereas the KVL are given by the resulting Euler-Lagrange equations. Second, an equivalent reduced constrained variational principle is developed, for which the KCL constraints are eliminated due to a representation of the Lagrangian on a reduced space. In this setting, the charges and flux linkages are the differential variables, whereas the currents play the role of algebraic variables. The number of inductors in the circuit and the circuit topology determine the degree of degeneracy of the system.
For the reduced version, we show for which cases the degeneracy of the system is canceled via the KCL constraints. Based on the variational formulation, a variational integrator for electric circuits can be constructed. 
For the case of a degenerate system, the applicability of the variational integrator is dependent on the choice of discretization. Based on the type and order of the discretization, the degeneracy of the continuous system is canceled for the resulting discrete scheme.
Three different integrators and their applicability to different electric circuits are investigated. 
The generality of a unified geometric (and discrete) variational formulation is advantageous for the analysis -- for very complex circuits in particular. Using the geometric approach, the main structure-preserving properties of the (discrete) Lagrangian system can be derived. In particular, good energy behavior and preservation of the spectrum of high frequencies of the solutions can be observed. Furthermore, preserved momentum maps due to symmetries of the Lagrangian system can be derived.
Going one step further, we extend the approach to a stochastic and multiscale setting.
Due to the variational framework, the resulting stochastic integrator will well capture the statistics of the solution (see for instance \cite{BRO10}), and the resulting multiscale integrator will still be variational \cite{TaOwMa2010}.

\paragraph{Outline}

In Section~\ref{sec:elecirc}, we first review the basic notation for electric circuits followed by a graph representation to describe the circuit topology. In addition, we introduce a geometric formulation that gives an interpretation of the different state spaces of a circuit model.
Based on the geometric view point, the two (reduced and unreduced) variational formulations are derived in Section~\ref{sec:varcirc}. The equivalence of both formulations as well as conditions for obtaining a non-degenerate reduced system are proven. 
In Section~\ref{sec:disvar}, the construction of different variational integrators for electric circuits is described and conditions for their applicability are derived.
The main structure-preserving properties of the Lagrangian system and the variational integrator are summarized in Section~\ref{sec:structure}.
In Section~\ref{sec:noise}, the approach is extended for the treatment of noisy circuits.
In Section~\ref{sex:example}, the efficiency of the developed variational integrators is demonstrated by means of numerical examples. A comparison with standard circuit modeling and circuit integrators is given. In particular, the applicability of the multiscale method FLAVOR \cite{TaOwMa2010} is demonstrated for a circuit with different time scales.

\section{Electric circuits}\label{sec:elecirc}

\subsection{Basic notations}
Considering an electric circuit, we introduce the following notations (following \cite{Nils2005}): A \emph{node} is a point in the circuit where two or more elements meet. A \emph{path} is a trace of adjacent elements, with no elements included more than once. A \emph{branch} is a path that connects two nodes. A \emph{loop} is a path that begins and ends at the same node. A \emph{mesh} (also called \emph{fundamental loop}) is a loop that does not enclose any other loops. A \emph{planar circuit} is a circuit that can be drawn on a plane without crossing branches.

Let  $q(t), v(t), u(t) \in \mathbb{R}^n$ be the time-dependent charges, the currents and voltages of the circuit elements with $t\in[0,T]$, where $q_J(t), v_J(t), u_J(t) \in \mathbb{R}^{n_J},\; J\in {\{L,C,R,V\}}$ are the corresponding quantities through the $n_L$ inductors, the $n_C$ capacitors, the $n_R$ resistors, and the $n_V$ voltage sources.
In addition, we give each of those devices an assumed current flow direction. In Table \ref{tab:chars}, the characteristic equations for basic elements are listed.
\begin{table}[tbp]
\center
\begin{tabular}{lll}
\hline
device & linear & nonlinear\\
\hline\hline
resistor & $v_R = G u_R$ & $v_R = g(u_R,t)$\\
\hline
capacitor & $v_C = C \frac{d}{dt}u_C$ & $v_C = \frac{d}{dt}q_C(u_C,t)$\\
\hline
inductor & $u_L = L \frac{d}{dt}v_L$ & $u_L = \frac{d}{dt}p_L(v_L,t)$\\
\hline\hline
&&\\
\hline
device &independent & controlled\\
\hline\hline
voltage source & $u_V = v(t)$ & $u_V = v(u_{ctrl},v_{ctrl},t)$\\
\hline
current source & $v_I = i(t)$ & $v_I = v(u_{ctrl},v_{ctrl},t)$\\
\hline\hline\\
\end{tabular}
\caption{Characteristic equations for basic circuit elements.}
\label{tab:chars}
\end{table}
For the analysis in this work, we focus on ideal linear circuit elements, resulting in the following constitutive laws:
\[
u_L(t) = L \dot{v}_L(t),\quad v_C(t) = C \dot{u}_C(t),\quad u_R(t)= R v_R(t),
\]
with inductance $L$, capacitance $C$ and resistance $R=G^{-1}$ with conductance $G$ and where in general it holds $\dot{q}(t)=v(t)$. The flux linkage for each element is denoted by $p(t) \in \mathbb{R}^n$ and for an inductor, it is defined as the time integral of the voltage across the inductor.
Note that in the case of an inductor (resp.~a capacitor), the associated charge $q_L$ (resp.~flux linkage $p_C$) is an artificial variable. Similarly, for the resistors and the voltage sources, the associated charges $q_R, q_V$ and flux linkages $p_R, p_V$ are artificial variables.

Ideal inductors and capacitors are purely reactive, i.e.~they dissipate no energy.
Thus, the magnetic energy stored in one inductor with inductance $L$ is
\[
E_\text{mag}= \frac{1}{2} L v_L^2
\]
where the amount of energy storage in one capacitor with capacitance $C$ is
\[
E_\text{el} = \int_{q=0}^{q_C} u_C \,dq = \int_{q=0}^{q_C} \frac{q}{C} \,dq = \frac{1}{2}\frac{1}{C}q_C^2.
\]

\subsection{Graph representation}
Consider now a circuit as a connected, directed graph with $n$ edges and $m+1$ nodes.
On the $i$th edge, there are: a capacitor with capacitance $C_i$, an inductor with inductance $L_i$, a voltage source $\epsilon_i$ and a resistor with resistance $R_i$, one or several of which can be zeros (cf.~Figure~\ref{celements}).
\begin{figure}[htb]
 \centering
 \includegraphics[width=0.5\textwidth]{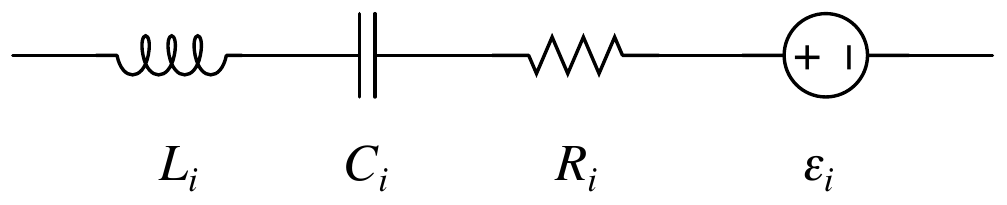} 
\caption{A typical branch of a circuit. On this edge, there are: an inductor $L_i$, a capacitor $C_i$, a resistor $R_i$, and a voltage source $\epsilon_i$, one or several of which can be zeros.}
\label{celements}
\end{figure}
Thus, branches in the circuit correspond to edges in the graph.
In the special case that each edge in the graph represents only one circuit element, the number of edges in the graph equals the number of circuit elements, and the number of nodes of the circuit and the graph are the same. 
For simplicity, we use the notions from circuit theory, i.e.~talking about branches and meshes in the graph.

For the analysis with circuits, one is faced with the following two basic laws:
\begin{enumerate}
\item The Kirchhoff Current Law (KCL) states that the sum of currents leading to and leaving from any node is equal to zero.
\item The Kirchhoff Voltage Law (KVL) states that the sum of voltages along each mesh (or fundamental loop) of the network is equal to zero.
\end{enumerate}
Let $K \in\mathbb{R}^{n,m}$ be the {\em Kirchhoff Constraint matrix} of a given circuit represented via a graph defined by
\begin{equation}
K_{ij} = \left\{ \begin{array}{ll} -1& \text{branch $i$ connected inward to node $j$}\\  +1& \text{branch $i$ connected outward to node $j$}\\  0& \text{otherwise.} \end{array}\right.
\end{equation}
In the special case where the two ends of an edge are connected to the same node, we set $K_{ij} =0$.
Since the ground node is excluded, the Kirchhoff Constraint matrix has only $m$ rather than $m+1$ columns.
Allowing only one circuit element for one branch, either inductor, capacitor, resistor, or voltage source, $K$ can be expressed as
\[  K = \left(\begin{array}{c} K_L \\ K_C \\K_R \\ K_V\end{array}\right), \]
where $K_{J}\in\mathbb{R}^{n_J,m}, J\in\{ L,C,R,V\}$ is the Constraint Matrix for the set of $n_L$ inductors, $n_C$ capacitors, $n_R$ resistors, and $n_V$ voltage sources, respectively with $n_L+n_C+n_R+n_V = n$.
The Kirchhoff Constraint Matrix provides the Kirchhoff current constraints as $K^T v  =0$.
For connected, planar graphs, the number of meshes $l$ is determined via $l=n-m$, where $n$ is the number of branches and $m+1$ the number of nodes. This is a direct consequence from Euler's formula \cite{GY04}.
We can thus define the \emph{Fundamental Loop matrix} $K_2\in\mathbb{R}^{n,n-m}$ by
\begin{equation}\label{eq:meshmatrix}
K_{2,ij} = \left\{ \begin{array}{ll} -1& \text{branch $i$ is a backward branch in mesh $j$}\\  +1& \text{branch $i$ is a forward branch in mesh $j$}\\  0& \text{branch $i$ does not belong to mesh $j$,} \end{array}\right.
\end{equation}
where again $K_2$ can be expressed as
\[  K_2 = \left(\begin{array}{c} K_{2,L} \\ K_{2,C} \\K_{2,R} \\ K_{2,V}\end{array}\right) \]
with $K_{2,J}\in\mathbb{R}^{n_J,n-m}, J\in\{ L,C,R,V\}$ is the Loop Matrix for the set of $n_L$ inductors, $n_C$ capacitors, $n_R$ resistors, and $n_V$ voltage sources, respectively.
The Fundamental Loop Matrix provides the Kirchhoff voltage constraints as $K_2^T u = 0$. An alternative expression of the Kirchhoff voltage constraints is given by $K\hat{u}=u$, where $\hat{u}$ are the node voltages of the circuit. By $u\in ker(K_2^T)$ and $u\in im(K)$ it follows directly $ker(K_2^T) = im(K)$ and thus $im(K_2)\perp im(K)$.

\subsection{Geometric setting}

Using a geometric approach for analyzing circuits, we define the configuration manifold to be the \emph{charge space} $Q\subseteq \mathbb{R}^n$ of circuit branches with points on the manifold denoted by $q \in Q$.
For a particular charge configuration $q$, the tangent bundle $TQ$ is the \emph{current space} with currents $v\in T_qQ \subseteq \mathbb{R}^n$ passing through the branches. The corresponding cotangent bundle $T^*Q$ is the \emph{flux linkage space} with the flux linkages $p\in T_q^*Q \subseteq \mathbb{R}^n $. Note that due to the analogon of the quantities, configuration, velocity and momentum in mechanical systems we stick with the notation $(q,v,p)$ for charge, current and flux linkage. The branch voltages $u$ are the analogon of forces for the mechanical system and are thus assumed to be covectors in the cotangent space $T_q^*Q$.

Let $\Delta_Q \subset TQ$ be a constraint distribution, which is locally given by
\begin{equation}\label{eq:distribution}
\Delta_Q(q) = \{ v\in T_qQ \,| \, \langle w^a,v \rangle =0,\, a=1,\ldots,m\} \subset T_qQ
\end{equation}
with the natural pairing $\langle \cdot, \cdot \rangle: T_q^*Q \times T_qQ\rightarrow \mathbb{R}$ of cotangent and tangent vectors.
$w^a$ are $m$ independent one-forms that form the basis for the annihilator $\Delta^0_Q(q)\subset  T^*Q$ which is locally given by
\begin{equation}\label{eq:annihilator}
\Delta^0_Q(q) = \{ w\in T^*_qQ \,| \, \langle w, v\rangle =0 \; \forall v \in \Delta_Q(q)\} \subset T^*_qQ.
\end{equation}

Using the matrix $K^T$ as local coordinate representation for the one-forms $w^a$, the distribution~\eqref{eq:distribution} forms the \emph{constraint KCL space} given by the submanifold
\[\Delta_Q(q) = \{ v\in T_qQ \,| \, K^T v =0\} \subset T_qQ,\]
that is spanned by $ker(K^T)$. Its annihilator $\Delta^0_Q(q)$ can thus locally expressed by the image $im(K)$ of $K$. Chooing this coordinate representation and with $ker(K_2^T) = im(K)$, the annihilator~\eqref{eq:annihilator} decribes the \emph{constraint KVL space} by
\[\Delta^0_Q(q) = \{ u\in T^*_qQ \,| \, K_2^T u =0 \} \subset T^*_qQ\]

Note, that the choices of $K$ and $K_2$ are in general not unique. The only design criterium for $K_2$ is the condition $im(K)\perp im(K_2)$. Alternative to \eqref{eq:meshmatrix}, a matrix $K_2$ can be constructed using a QR-decomposition of $K$. Thus, this approach is not restricted to cases, where the mesh topology is obvious as for planar graphs.
However, in the following we work with the Fundamental Loop matrix as candidate for the matrix $K_2$ due to the physical interpretation.

From a geometric point of view we can distinguish between three different spaces: Let $\mathfrak{B}$ denote the space of branches, $\mathfrak{M}$ the space of meshes and $\mathfrak{N}$ the space of nodes, where we exclude the one node defined as ground. $(q,v,p,u)$ denote the branch charges, currents, flux linkages, and voltages, and $(\tilde q,\tilde v,\tilde p,\tilde u)$ and $(\hat q,\hat v,\hat p,\hat u)$ the corresponding quantities in mesh and node space, respectively. From KCL and KVL we know that the node currents (and charges) as well as the mesh voltages are zero. For $\mathfrak{M}$ and $\mathfrak{N}$, we define the corresponding configuration, tangent and cotangent spaces $M\subseteq\mathbb{R}^{n-m},T_{\tilde{q}}M\subseteq\mathbb{R}^{n-m},T_{\tilde{q}}^*M\subseteq\mathbb{R}^{n-m}$ and $N \subseteq\mathbb{R}^m,T_{\hat{q}}N\subseteq\mathbb{R}^m,T_{\hat{q}}^*N\subseteq\mathbb{R}^m$.
Then, branch, loop and node space are defined to be the Pontryagin bundle consisting of the direct sum of tangent and cotangent space, i.e.~$\mathfrak{B} = \Delta_Q \oplus \Delta_Q^0$, $\mathfrak{M}= TM \oplus T^*M$ and $\mathfrak{N}=TN \oplus T^*N$.

The following diagram gives the relation between the defined spaces in terms of the Kirchhoff Constraint matrix $K$ and the Fundamental Loop matrix $K_2$
{\large
\begin{equation}\label{eq:spaces}
\begin{CD}
T^*N @>K>> \Delta^0_Q @>K_2^T>> T^*M \\[-3pt]
\mathfrak{N} && \mathfrak{B} && \mathfrak{M} \\[-10pt]
\text{{\normalsize nodes}}  && \text{{\normalsize branches}}&& \text{{\normalsize meshes}}\\[-3pt]
TN @<<K^T< \Delta_Q @<<K_2< TM \\
&& &&
 \end{CD}
\end{equation}
}
with the linear maps $K^T: \Delta_Q \rightarrow TN$ and $K_2: TM \rightarrow \Delta_Q$, and their adjoints $K: T^*N \rightarrow \Delta_Q^0$ and $K_2^T: \Delta_Q^0 \rightarrow T^*M$.
As already stated above, the branch currents consistent with KCL are determined by $ker(K^T)$, where the branch voltages consistent with KVL are given by $ker(K_2^T)$.
On the other hand, from diagram \eqref{eq:spaces} we can directly follow, that the set of branch currents fulfilling the KCL can alternatively expressed as $v=K_2 \tilde{v}$, whereas the set of branch voltages that fulfill the KVL are in the image of $K$ as $u = K\hat{u}$.
These are the standard relations between
branch currents $v(t)$ and mesh currents $\tilde{v}(t)$ and
 branch voltages $u(t)$ and node voltages $\hat{u}(t)$, respectively, given by KCL and KVL.

Note, that diagram~\eqref{eq:spaces} represents the general relations between the tangent bundles $TQ, TN$, and $TM$, and the corresponding cotangent bundles. The matrices $K$ and $K_2$ are local coordinate choices for the embeddings of the different spaces. These are not unique, however using different coordinates representations, the submanifolds $TN, T^*N, TM$, and $T^*M$ loose their physical meaning. 

Following the lines of \cite{YoMa2006b} the tangent space at $q$ can be splitted such that $T_qQ=\mathcal{H}_q \oplus \mathcal{V}_q$, where $\mathcal{H}_q = \Delta_q(q)$ is the \emph{horizontal space} and $\mathcal{V}_q$ the \emph{vertical space} at $q$. The matrix $K^T$ is a local matrix representation of the \emph{Ehresmann connection} $A_q: T_q Q \rightarrow \mathcal{V}_q$.

\begin{remark}
A branch can consist of more than one circuit element in a row. In this case, the branch voltage is assumed to be the sum of the voltages of all elements in this branch.
\end{remark}

\section{Variational formulation for electric circuits}\label{sec:varcirc}

In the following, we derive the equations of motion for the circuit system, making use of variational principles known in mechanics. Due to the constrained system (via KCL and KVL), we present two different variational formulations that distinguish the way the constraints are involved.

\subsection{Constrained variational formulation}

We can define a Lagrangian $\mathcal{L}: TQ \rightarrow \mathbb{R}$ of the circuit system consisting of the difference between magnetic and electric energy as
\begin{equation}\label{Lagrangian}
\mathcal{L}(q,v) = \frac{1}{2} v^T L v - \frac{1}{2} q^T C q
\end{equation}
with $L = \text{diag}(L_1,\ldots,L_n)$ and $C= \text{diag}\left(\frac{1}{C_1},\ldots,\frac{1}{C_n}\right)$.
In the case where no inductor (resp.~no capacitor) is on branch $i$, the corresponding entry $L_i$ (resp.~$\frac{1}{C_i}$) in the matrix $L$ (resp.~$C$) is zero.
In the presence of mutual inductors rather than self inductors, the matrix $L$ is not diagonal anymore, but always positive semi-definite. If not explicitly mentioned the following theory and construction is also valid for mutual inductors. 
The Legendre transform $\mathbb{F}\mathcal{L}: TQ \rightarrow T^*Q$ is defined by
\begin{equation}
\mathbb{F}\mathcal{L}(q,v) = (q,\partial \mathcal{L}/ \partial v) = (q,Lv).
\end{equation}
Note that the Lagrangian can be degenerate if the Legendre transform is not invertible, i.e.~$L$ is singular.
The \emph{constraint flux linkage subspace}\footnote{also denoted by the set of primary constraints} is defined by the Legendre transform as
\[
P = \mathbb{F}\mathcal{L}(\Delta_Q) \subset T^*Q,
\]
where $\Delta_Q \subset TQ$ is the distribution.
The Lagrangian force of the system consists of a damping force that results from the resistors and an external force being the voltage sources
\begin{equation}\label{eq_fl}
f_L(q,v,t) = -\text{diag}(R)v + \text{diag}(\mathcal{E}) u
\end{equation}
with $R=(R_1,\ldots,R_n)^T$ and $\mathcal{E}=(\epsilon_1,\ldots,\epsilon_n)^T$. If no resistor is on branch $i$, the corresponding entry $R_i$ in the vector $R$ is zero. For the entries of the vector $\mathcal{E}$, it holds $\epsilon_i=0$ if no voltage source is on branch $i$ and $\epsilon_i=1$ otherwise. Here we assume that the time evolution of the voltage sources is given as time dependent function $u_s(t)$. Thus, in the following, we replace $\text{diag}(\mathcal{E}) u$ by $u_s(t)$ for a given function $u_s:[0,T]\rightarrow \mathbb{R}^n$.

To derive the equations of motion for the circuit system, we make use of the Lagrange-d'Alembert-Pontryagin principle, i.e.~we are searching for curves $q(t)$, $v(t)$ and $p(t)$ fulfilling
\begin{equation}
\delta \int_0^T \mathcal{L}({q}(t),{v}(t)) + \left\langle {p}(t),\dot{{q}}(t) - {v}(t) \right\rangle \,dt + \int_0^T f_L({q}(t),{v}(t),t) \cdot \delta q(t) \,dt=0
\end{equation}
with fixed initial and final variations $\delta q(0)=\delta q(T)=0$ and constrained variations $\delta q \in \Delta_Q(q)$.

Taking variations gives us
\begin{equation}
\int_0^T \left [ \left\langle  \frac{\partial \mathcal{L}}{\partial q}+ f_L, \delta q \right\rangle    - \langle \dot{{p}}, \delta {q} \rangle + \left\langle \delta p ,\dot{{q}} - {v}  \right\rangle + \left\langle \left( \frac{\partial \mathcal{L}}{\partial v} \right)- {p},  \delta{v} \right\rangle \right] \, dt =0
\end{equation}
for arbitrary variations $\delta{v}$ and $\delta {p}$, $K^Tv = 0$ and constrained variations $\delta q \in \Delta_Q(q)$. This leads to the constrained Euler-Lagrange equations
\begin{subequations}\label{eq:EL_full}
\begin{align}
 \frac{\partial \mathcal{L}}{\partial q} -\dot{p} + f_L & \in \Delta^0_Q(q) \label{EL_K_1}\\
\dot{{q}} & =  {v}\label{EL_K_2}\\
 \frac{\partial \mathcal{L}}{\partial v} -{p}  & =0\label{EL_K_3}\\
 K^T v & = 0. \label{EL_K_4}
\end{align}
\end{subequations}
For the Lagrangian (\ref{Lagrangian}) and the forces \eqref{eq_fl}, the constrained Euler-Lagrange equations are
\begin{subequations}\label{eq:EL_circuit_full}
\begin{align}
\dot{ p} & = -C q  - \text{diag}(R) {v} +u_s + K \lambda \label{eq:EL_circuit_full_p} \\
\dot{{q}} & = {v}\\
{p} & =  L{v} \label{eq:EL_circuit_full_3}\\
K^T v & = 0,
\end{align}
\end{subequations}
where $\lambda$ represent the node voltages $\hat{u}\in T^*N$. Thus, the first line corresponds to the KVL equations of the form $K\hat{u} = u$, and the last line are the KCL equations.
System \eqref{eq:EL_circuit_full} is a differential-algebraic system with differential variables $q$ and $p$ and algebraic variables $v$ and $\lambda$. The involvement of the function $u_s(t)$ makes the system a non-autonomous system. Equation \eqref{eq:EL_circuit_full_3} (also denoted by primary constraints) reflects the degeneracy of the Lagrangian system: since $\mathbb{F}\mathcal{L}$ is not invertible (i.e.~$L$ is singular), we can not eliminate the algebraic variable $v$ to obtain a purely Hamiltonian formulation.
However, in the next step, we eliminate the algebraic variable $\lambda$ by the use of a reduced constrained variational principle.

\subsection{Reduced constrained variational formulation}\label{proj_method}

With the following reduced principle, we derive a slightly different form of the resulting differential-algebraic system. This reduced formulation is advantageous from different perspectives:
First, the reduced formulation is less redundant, such that the Lagrange multipliers are eliminated and the state space dimension is reduced. Second, for specific circuits, the degeneracy of the Lagrangian is canceled. Third, the reduced state space still has a physical and geometric interpretation: The reduced Lagrangian is defined on the mesh space $TM\subseteq\mathbb{R}^{2(n-m)}$ rather than on the branch space $TQ\subseteq\mathbb{R}^{2n}$.

For the reduction, instead of treating the KCL as extra constraint in the form $K^T v = 0$, we directly involve the KCL form $K_2 \tilde{v}=v$ with $\tilde{v}\in T_qM\subseteq\mathbb{R}^{n-m}$ for the definition of the new Lagrangian system.
Since $K$ is constant, the constraints are integrable, i.e.~the configurations $q$ are constrained to be in the submanifold
\[C = \{ q\in Q \,| \, K^T q =0\}\]
for consistent initial values $q_0\in C$.
This simply means that topological relationships that apply for currents will also hold for charges to within a constant vector.
Then, it holds $T_qC = \Delta_Q(q)$ and the branch charges $q$ can be expressed by the mesh charges $\tilde{q}\in M \subseteq\mathbb{R}^{n-m}$ as $q=K_2\tilde{q}$.
We define the constrained Lagrangian $\mathcal{L}^M:TM \rightarrow \mathbb{R}$ via pullback as
$\mathcal{L}^M := K_2^* \mathcal{L} : TM \rightarrow \mathbb{R}$ with
\begin{equation}\label{eq:redL}
\mathcal{L}^M(\tilde{q},\tilde{v}) = \mathcal{L}(K_2 \tilde{q}, K_2\tilde{v})=   \frac{1}{2} \tilde v^TK_2^T L K_2\tilde v - \frac{1}{2} \tilde q^TK_2^T C K_2 \tilde q
\end{equation}
with the Legendre transform $\mathbb{F}\mathcal{L}^M : TM \rightarrow T^*M$
\[
\mathbb{F}\mathcal{L}^M(\tilde q, \tilde v) = (\tilde q, \partial \mathcal{L}^M/\partial \tilde{v}) = (\tilde q, K_2^TLK_2 \tilde v ).
\]
Depending on the inductor matrix $L$ and the graph topology, the matrix $K_2^T LK_2$ can still be singular, i.e.~the Lagrangian system can still be degenerate.
The cotangent bundle $T^*M$ is given by
\begin{align*}
T^*M &= \{ (\tilde q, \tilde p) \in \mathbb{R}^{n-m,n-m} \, | \, (\tilde q,\tilde p) = \mathbb{F}\mathcal{L}^M(\tilde q,\tilde v)  \,\text{with}\,(\tilde q,\tilde v)\in TM \}\\
&  =  \{ (\tilde q, \tilde p) \in \mathbb{R}^{n-m,n-m} \, | \, (\tilde q,\tilde p) = (\tilde q, K^T_2 p)  \,\text{with}\,p \in P \}.
\end{align*}
Thus, the constrained force $f_L^M$ in $T^*M$ is defined as
\begin{equation}\label{eq:consforce}
f_L^M(\tilde q, \tilde v,t) = K_2^T f_L(K_2 \tilde q,K_2 \tilde v,t) = -K_2^T \text{diag}(R) K_2 \tilde{v} + K_2^T u_s(t).
\end{equation}
With $\tilde{p}\in T_{\tilde q}^*M \subset \mathbb{R}^{n-m}$ given as $\tilde{p} = K_2^Tp$ we obtain the following reduced Lagrange-d'Alembert-Pontryagin principle
\begin{equation}
\delta \int_0^T \mathcal{L}^M(\tilde{q}(t),\tilde{v}(t)) + \left\langle \tilde{p}(t),\dot{\tilde{q}}(t) - \tilde{v}(t) \right\rangle \,dt + \int_0^T f^M_L(\tilde{q}(t),\tilde{v}(t),t) \cdot \delta \tilde q(t) \,dt=0
\end{equation}
with fixed initial and final variations $\delta \tilde{q}(0)=\delta \tilde{q}(T)=0$.
Taking variations gives us
\begin{equation}
\int_0^T \left[ \left\langle  \frac{\partial \mathcal{L}^M}{\partial \tilde q}+ f^M_L, \delta \tilde q \right\rangle    - \langle \dot{\tilde{p}}, \delta \tilde{q} \rangle + \left\langle \delta \tilde{p} , \dot{\tilde{q}} - \tilde{v} \right\rangle + \left\langle \left( \frac{\partial \mathcal{L}^M}{\partial \tilde v} \right)- \tilde{p},  \delta\tilde{v} \right\rangle \right]\, dt=0
\end{equation}
for arbitrary variations $\delta \tilde{v}$ and $\delta \tilde{p}$ and $\delta \tilde q$. This results in the reduced Euler-Lagrange equations
\begin{subequations}\label{eq:ELred_general}
\begin{align}
 \frac{\partial \mathcal{L}^M}{\partial \tilde q} -\dot{\tilde p} + f^M_L & =0\label{ELc_K_1}\\
\dot{\tilde{q}} & =  \tilde{v}\label{ELc_K_2}\\
 \frac{\partial \mathcal{L}^M}{\partial \tilde v} -\tilde{p}  & =0.\label{ELc_K_3}
\end{align}
\end{subequations}
For the Lagrangian \eqref{eq:redL} and the forces \eqref{eq:consforce}, the constrained Euler-Lagrange equations are
\begin{subequations}\label{EL_Kcirc}
\begin{align}
\dot{\tilde p} & = K_2^T \left(-CK_2 \tilde q  - \text{diag}(R) K_2 \tilde{v} + u_s\right) \\
\dot{\tilde{q}} & =  \tilde{v}\\
\tilde{p} & = K_2^T LK_2\tilde{v} \label{EL_Kcirc_3}.
\end{align}
\end{subequations}
Here, the first equation is now the KVL in the form $K_2^T u = 0$, in which the KCL in the form $K_2 \tilde{v} = v$ is also involved. System \eqref{EL_Kcirc} is a differential-algebraic system with differential variables $\tilde{q}$ and $\tilde{p}$ and algebraic variables $\tilde{v}$. The algebraic equation \eqref{EL_Kcirc_3} is the Legendre transformation of the system. If this is invertible (i.e.~the matrix $K_2^TLK_2$ is regular), the algebraic variable $v$ can be eliminated. In this case, the Euler-Lagrange equations \eqref{EL_Kcirc} represent a non-degenerate Lagrangian system.

In the following proposition we show for which cases the reduced Lagrangian system is non-degenerate for LC circuits, i.e.~for which cases the KVL cancels the degeneracy. The statements for RCL and RCLV circuits can be derived in an analogous way (see Remark \ref{rem:prop}b)).

\begin{proposition}\label{prop:degeneracy}
For LC circuits (including only self inductors), the system is non-degenerate if the number of capacitors equals the number of independent constraints involving the currents through the capacitives' branches.
\end{proposition}

\begin{proof}
We have to show that $ker(K_2^TLK_2)=\{0\}$. Let $n_C$ be the number of capacitors, $m$ the number of Kirchhoff Constraints such that $K_C^T\in\mathbb{R}^{m,n_C}$. Let $l_C\le m$ be the number of independent constraints involving the currents through the capacitives branches. With $n_C = l_C \le m$ we have $rank(K_C^T)=n_C$, thus $ker(K_C^T)=\{0\}$.
On the other hand, it holds
\[ ker(L) = \{ v \in T_qQ \, | \, v_L =0 \}\]
and
\[\mathcal{R}(K_2) = ker(K^T) = \{ v \in T_qQ \, | \, K^T v =0 \} = \{ v \in T_qQ \, | \, (K_L^T \, K^T_C) \left(\begin{array}{c} v_L \\ v_C \end{array}\right) =0 \}\]
With $ker(K_C^T)=\{0\}$ this results in
\begin{equation}\label{eq:RK_2}
\mathcal{R}(K_2) \cap ker(L) = \{ v \in T_qQ \, | \, K_C^T v_C =0 \}=\{0\}.
\end{equation}
and thus $ker(LK_2)= \{ 0 \}$.
Since $L$ is a diagonal matrix, we can split $K_2^TLK_2$ into $K_2^T \sqrt{L}^T \sqrt{L} K_2$, where $\sqrt{L}$ corresponds to the diagonal matrix with diagonal elements $\sqrt{L_i}$, $L_i>0, \,i=1,\ldots,n$.
Since $K_2$ has full column rank, we know with \eqref{eq:RK_2} that $\sqrt{L}K_2$ also has full column rank. It follows for $y\in \mathbb{R}^{n-m}$ and $y\in ker(K_2^TLK_2)$
\[
K_2^TLK_2 y = 0 \Rightarrow y^T K_2^TLK_2 y = 0 \Leftrightarrow y^T K_2^T \sqrt{L}^T \sqrt{L} K_2 y = 0 \Leftrightarrow \| \sqrt{L} K_2 y \|_2 = 0
\]
and thus $y=0$ since $ker(\sqrt{L}K_2)= \{ 0 \}$.
We therefore have $ker(K_2^TLK_2) = ker(\sqrt{L}K_2) =\{0\}$ and the matrix $K_2^TLK_2$ is invertible.
\end{proof}

\begin{remark}\label{rem:prop}
\begin{itemize}
\item[a)]Intuitively spoken, the degeneracy of the original Lagrangian is due to the lack of magnetic energy terms for the capacitors. With each independent constraint on the capacitor currents, one degree of freedom of the system can be removed. Hence, as many capacitors constraints are required to remove the capacitor current (cf.~\cite{ClSch2003}).
\item[b)] In addition, for a RLC (resp.~RCLV) non-degenerate circuit, the number of resistors (resp.~and voltage sources) has to equal the number of independent constraints involving the currents through the resistor (resp.~and voltage source) branches.
\end{itemize}
\end{remark}

\begin{theorem}[Equivalence]\label{th:equivalence}
The system \eqref{eq:EL_full} and the reduced system \eqref{eq:ELred_general} are equivalent in the following sense:
\begin{itemize}
\item[(i)] Let $(\tilde{q},\tilde{p},\tilde{v})$ be a solution of the reduced system \eqref{eq:ELred_general} and let $q=K_2\tilde{q}, v=K_2\tilde{v}$ and $(q,p)=\mathbb{F} \mathcal{L}(q,v)$. Then $(q,v,p)$  is a solution to system \eqref{eq:EL_full} and it holds $\tilde{p} = K_2^T {p}$.
\item[(ii)] Let $(q,v,p)$ be a solution to system \eqref{eq:EL_full} and $\tilde{q}= K_2^+ q, \tilde{v}= K_2^+ v$ and let $\tilde{p} = K_2^T {p}$ with the well-defined pseudo-inverse $K_2^+$ of $K_2$ (with $K_2^+ K_2 = I$). Then $(\tilde{q},\tilde{p},\tilde{v})$ is a solution of the reduced system \eqref{eq:ELred_general}.
\end{itemize}
\end{theorem}
\begin{proof}
\begin{itemize}
\item[(i)]
Assume $(\tilde{q},\tilde{p},\tilde{v})$ is a solution of \eqref{eq:ELred_general}. 
From the assumption $p=\mathbb{F} L(q,v)$ it follows $p-\frac{\partial}{\partial {v}} \mathcal{L}(q,v)=0$. With $\mathcal{L}^M = K_2^* \mathcal{L}$ it holds
\[
\frac{\partial \mathcal{L}^M}{\partial \tilde{q}}(\tilde{q},\tilde{v}) = \frac{\partial \mathcal{L}}{\partial \tilde{q}} (K_2\tilde{q},K_2\tilde{v}) = \left(\frac{\partial q}{\partial \tilde{q}}\right)^T \frac{\partial  \mathcal{L} }{\partial {q}}(K_2\tilde{q},K_2\tilde{v})=  K_2 ^T \frac{\partial \mathcal{L}}{\partial {q}} (q,v).
\]
Similarly, it holds $\displaystyle \frac{\partial \mathcal{L}^M}{\partial \tilde{v}}(\tilde{q},\tilde{v}) = K_2 ^T \frac{\partial  \mathcal{L}}{\partial {v}}(q,v)$ and thus it follows 
\[
\tilde{p}=\displaystyle \frac{\partial \mathcal{L}^M}{\partial \tilde{v}}(\tilde{q},\tilde{v})=K_2 ^T \frac{\partial \mathcal{L}}{\partial {v}} (q,v) = K_2^T p.
\]
Together with \eqref{eq:consforce}, this gives
\[
K_2^T \dot{p} = \dot{\tilde{p}} = \frac{\partial \mathcal{L}^M}{\partial \tilde{q}} + f^M_L = K_2^T\left(\frac{\partial \mathcal{L}}{\partial {q}} + f_L\right) \Rightarrow K_2^T \left(\frac{\partial \mathcal{L}}{\partial {q}} -\dot{p} +f_L  \right) =0 \Rightarrow \frac{\partial \mathcal{L}}{\partial {q}} -\dot{p} +f_L \in ker(K_2^T).
\]
With $ker(K_2^T) = im(K)$, it follows 
\[\frac{\partial \mathcal{L}}{\partial {q}} -\dot{p} +f_L  \in im(K) = \Delta_Q^0(q)\]
as can be seen from diagram \eqref{eq:spaces}. 
Furthermore, we
have
\[
\dot{q} = K_2 \dot{\tilde{q}} = K_2 \tilde{v} = v.
\]
and since it holds $v=K_2 \tilde{v}$, it follows from diagram \eqref{eq:spaces} $K^T v =0$. Both expressions are equivalent formulations of the KCL. 
\item[(ii)]
Now assume $({q},{p},{v})$ is a solution of \eqref{eq:EL_full}.  With $ker(K_2^T)=im(K)$ and \eqref{eq:consforce}, it follows immediately 
\[
\dot{\tilde{p}} = K_2^T \dot{p} = K_2^T\left(\frac{\partial \mathcal{L}}{\partial {q}}  +f_L\right) = \frac{\partial \mathcal{L}^M}{\partial \tilde{q}} + f^M_L.
\]
Furthermore, from $\dot{q}=v$ we get $K_2^+ \dot{q} =K_2^+ v$ which gives $\dot{\tilde{q}}=\tilde{v}$.
Finally, we have $\displaystyle \tilde{p}=K_2^T p = K_2^T \frac{\partial \mathcal{L}}{\partial v} = \frac{\partial \mathcal{L}^M}{\partial \tilde{v}}$.
\end{itemize}
\end{proof}

\begin{remark}
We require the assumption $(q,p)=\mathbb{F}L(q,v)$ (the fulfillment of the Legendre transformation) in Theorem~\ref{th:equivalence}(i)
for the fulfillment of the relation \eqref{EL_K_3}. 
A unique derivation of $p$ directly from $\tilde{p}$ is in general not possible from $\tilde{p}=K_2^T p$ as it is for $q$ and $v$: Although there is a canonical projection $K_2^T: T^*Q \rightarrow T^*M$,
there is no corresponding canonical embedding of $T^*M$ into $T^*Q$ (see also \cite{MaWe01}).
Assuming $\tilde{p}=K_2^T p$ instead of $(q,p)=\mathbb{F}L(q,v)$ in (i), we only get the relation $K_2^T(p-\partial \mathcal{L} / \partial v)=0$, and \eqref{EL_K_3} may not be fulfilled. 
\end{remark}

\section{Discrete variational principle for electric circuits}\label{sec:disvar}

In this section, we derive a discrete variational principle that leads to a variational integrator for the circuit system. 
Since the solution of the reduced system \eqref{eq:ELred_general} can be easily transformed to a solution of the full system \eqref{eq:EL_full} (Theorem~\ref{th:equivalence}), we restrict the discrete derivation to the reduced case. 
For the case of a degenerate reduced system, the choice of discretization is important to obtain a variational integrator that manages to bypass the difficulty of intrinsic degeneracy and thus is applicable for a simulation.
In this section, three different discretizations are introduced that result in three different discrete variational schemes for which the solvability conditions are derived.
 
For the discrete variational derivation, we introduce a discrete time grid $\Delta t = \{ t_k = kh\, | \,k=0,\ldots,N \}$, $Nh = T$, where $N$ is a positive integer and $h$ the step size.
We replace the charge $\tilde{q}:[0,T]\rightarrow M$, the current $\tilde{v}:[0,T]\rightarrow T_{\tilde{q}}M$ and the flux linkage $\tilde{p}:[0,T]\rightarrow T_{\tilde{q}}^*M$ by their discrete versions $\tilde{q}_d: \{t_k\}_{k=0}^N \rightarrow M$, $\tilde{v}_d: \{t_k\}_{k=0}^N \rightarrow T_{\tilde{q}}M$ and $\tilde{p}_d: \{t_k\}_{k=0}^N \rightarrow T^*_{\tilde{q}}M$, where we view $\tilde{q}_k=\tilde{q}_d(kh)$, $\tilde{v}_k=\tilde{v}_d(kh)$ and $\tilde{p}_k=\tilde{p}_d(kh)$ as an approximation to $\tilde{q}(kh)$, $\tilde{v}(kh)$ and $\tilde{p}(kh)$, respectively.

\subsection{Forward Euler}
We replace the reduced Lagrange-d'Alembert-Pontryagin principle with a discrete version
\begin{equation}\label{eq:lapp_discrete}
\delta \left\{  h \sum\limits_{k=0}^{N-1} \left( \mathcal{L}^M( \tilde{q}_k, \tilde{v}_k) + \left\langle \tilde{p}_k, \frac{\tilde{q}_{k+1}-\tilde{q}_k}{h}- \tilde{v}_k \right\rangle \right) \right\} + h\sum\limits_{k=0}^{N-1} f_L^M(\tilde{q}_k,\tilde{v}_k,t_k) \delta \tilde{q}_k=0,
\end{equation}
where in \eqref{eq:lapp_discrete} the time derivative $\dot{q}(t)$ is approximated by the forward difference operator and the force evaluated at the left point. 

For discrete variations $\delta \tilde{q}_k$ that vanish in the initial and final points as $\delta \tilde{q}_0= \delta \tilde{q}_N=0$ and discrete variations $\delta \tilde{v}_k$ and $\delta \tilde{p}_k$ this gives
\begin{equation}
\begin{array}{c}
\left\langle  {\displaystyle \frac{\partial \mathcal{L}^M}{\partial \tilde v}(\tilde{q}_{0}, \tilde{v}_{0})} -  \tilde{p}_{0} , \delta \tilde{v}_0 \right\rangle + {\displaystyle \sum\limits_{k=1}^{N-1} \left [\left\langle  \frac{\partial \mathcal{L}^M}{\partial \tilde{q}}(\tilde{q}_{k},\tilde{v}_{k}) - \frac{1}{h} (\tilde{p}_{k}-\tilde{p}_{k-1}) + f_L^M(\tilde{q}_k,\tilde{v}_{k},t_k), \delta \tilde{q}_k \right\rangle \right.}\\
\left. \left\langle \delta \tilde{p}_{k-1},{\displaystyle \frac{\tilde{q}_{k}-\tilde{q}_{k-1}}{h}}-\tilde{v}_{k-1}\right\rangle + \left\langle {\displaystyle \frac{\partial \mathcal{L}^M}{\partial \tilde v}(\tilde{q}_{k}, \tilde{v}_{k})}-  \tilde{p}_{k}, \delta \tilde{v}_k  \right\rangle \right]  + \left\langle \delta \tilde{p}_{N-1}, {\displaystyle \frac{\tilde{q}_{N}-\tilde{q}_{N-1}}{h}- \tilde{v}_{N-1} }\right\rangle = 0.
\end{array}
\end{equation}
This leads to the discrete reduced constrained Euler-Lagrange equations
\begin{subequations}\label{eq:backwardLag}
 \begin{align}
 \frac{\partial \mathcal{L}^M}{\partial \tilde v}(\tilde{q}_{0}, \tilde{v}_{0})=  \tilde{p}_{0}&\\
\left.\begin{array}{rl}
{\displaystyle \frac{\partial \mathcal{L}^M}{\partial \tilde{q}}(\tilde{q}_{k},\tilde{v}_{k}) - \frac{1}{h} (\tilde{p}_{k}-\tilde{p}_{k-1}) + f_L^M(\tilde{q}_k,\tilde{v}_{k},t_k)}&=0\\
{\displaystyle \frac{\tilde{q}_{k}-\tilde{q}_{k-1}}{h}}&= \tilde{v}_{k-1}\\
{\displaystyle \frac{\partial \mathcal{L}^M}{\partial \tilde v}(\tilde{q}_{k}, \tilde{v}_{k})}&=  \tilde{p}_{k}
\end{array} \right\} &k=1,\ldots,N-1\\
\frac{\tilde{q}_{N}-\tilde{q}_{N-1}}{h}= \tilde{v}_{N-1}.&
\end{align}
\end{subequations}
For the Lagrangian defined in \eqref{eq:redL} and the Lagrangian forces defined in \eqref{eq:consforce}, this results in
\begin{subequations}
 \begin{align}
  \tilde{p}_0 = K_2^T L K_2\tilde{v}_0&\label{eq:DEL_fd_1}\\
 \left.\begin{array}{rl}\label{eq:DEL_fd}
{\displaystyle \frac{\tilde{p}_k-\tilde{p}_{k-1}}{h} }& = K_2^T \left(-C K_2 \tilde{q}_k- \text{diag}(R)K_2\tilde{v}_k + u_s(t_k)\right) \\[8pt]
{\displaystyle \frac{\tilde{q}_{k}-\tilde{q}_{k-1}}{h}}&=  \tilde{v}_{k-1}\\[8pt]
K_2^T L K_2\tilde{v}_k &= \tilde{p}_k
\end{array} \right\} &k=1,\ldots,N-1\\
\frac{\tilde{q}_{N}-\tilde{q}_{N-1}}{h}= \tilde{v}_{N-1}.\label{eq:DEL_fd_3}&
\end{align}
\end{subequations}
This gives the following update rule:
For given $(\tilde{q}_0,\tilde{v}_0)$, use \eqref{eq:DEL_fd_1} to compute $\tilde{p}_0$. Then, use the iteration scheme
\begin{equation}\label{eq:it_fd}
\left(\begin{matrix}
I & 0 & 0 \\[2pt] 0 & K_2^T L K_2 & -I\\[2pt] h K_2^TCK_2 & h K_2^T \text{diag}(R)K_2 & I
\end{matrix}\right) \left(\begin{matrix} \tilde{q}_k\\[2pt]\tilde{v}_k\\[2pt]\tilde{p}_k\end{matrix}\right) = \left(\begin{matrix}
I & h I & 0 \\[2pt] 0 & 0 & 0\\[2pt] 0 & 0 & I \end{matrix}\right) \left(\begin{matrix} \tilde{q}_{k-1}\\[2pt]\tilde{v}_{k-1}\\[2pt]\tilde{p}_{k-1}\end{matrix}\right) +\left(\begin{matrix} 0\\[2pt]0\\[2pt]hK_2^T\end{matrix}\right) u_s(t_k)\quad \text{for}\; k=1,\ldots,N
\end{equation}
to compute $\tilde{q}_1,\ldots,\tilde{q}_{N}$, $\tilde{v}_1,\ldots,\tilde{v}_{N}$ and $\tilde{p}_1,\ldots,\tilde{p}_{N}$.
\begin{proposition}\label{prop:forw_solve}
System \eqref{eq:it_fd} is uniquely solvable if the matrix $K_2^T (L+h \text{\emph{diag}}(R))K_2$ is regular.
\end{proposition}

\begin{proof}
System \eqref{eq:it_fd} is uniquely solvable if the iteration matrix $A = \left(\begin{matrix}
I & 0 & 0 \\[2pt] 0 & K_2^T L K_2 & -I\\[2pt] h K_2^TCK_2 & h K_2^T \text{diag}(R)K_2 & I
\end{matrix}\right)$ has zero nullspace. For $Az = 0$ for $z=(\tilde{q},\tilde{v},\tilde{p})$, it holds (i) $\tilde{q}=0$, (ii) $\tilde{p} = K_2^TLK_2 \tilde{v}$, (iii) $h K_2^TCK_2 \tilde{q} +  h K_2^T \text{diag}(R)K_2 \tilde{v} + \tilde{p} = 0$. Substituting (i) and (ii) in (iii) gives $K_2^T( L+ h \text{diag}(R) )K_2 \tilde{v}  = 0$.
Thus $z=0$ is the unique solution of $Az=0$ iff $K_2^T( L+ h \text{diag}(R) )K_2$ has zero nullspace.
\end{proof}

\subsection{Backward Euler}
Approximating the time derivative $\dot{q}(t)$ by the backward difference operator rather than by the forward difference operator as
\begin{equation}\label{eq:lapp_discrete_bd}
\delta \left\{  h \sum\limits_{k=1}^{N} \left( \mathcal{L}^M( \tilde{q}_{k}, \tilde{v}_{k}) + \left\langle \tilde{p}_k, \frac{\tilde{q}_{k}-\tilde{q}_{k-1}}{h}- \tilde{v}_k \right\rangle \right) \right\} + h\sum\limits_{k=1}^{N} f_L^M(\tilde{q}_k,\tilde{v}_k,t_k) \delta \tilde{q}_k=0
\end{equation}
with discrete variations $\delta \tilde{q}_k$, that vanish in the initial and final points as $\delta \tilde{q}_0= \delta \tilde{q}_N=0$ and discrete variations $\delta \tilde{v}_k$ and $\delta \tilde{p}_k$ yields
\begin{equation}
\begin{array}{c}
\left\langle  \delta \tilde{p}_ 1 ,{\displaystyle \frac{\tilde{q}_{1}-\tilde{q}_{0}}{h}- \tilde{v}_{1}} \right\rangle + \left\langle {\displaystyle \frac{\partial \mathcal{L}^M}{\partial \tilde v}(\tilde{q}_{1}, \tilde{v}_{1})-\tilde{p}_{1}} ,\delta \tilde{v}_1\right\rangle + {\displaystyle \sum\limits_{k=2}^N\left[ \left\langle \delta\tilde{p}_k ,  \frac{\tilde{q}_{k}-\tilde{q}_{k-1}}{h}- \tilde{v}_{k} \right\rangle    \right.}\\
+\left. \left\langle {\displaystyle \frac{\partial \mathcal{L}^M}{\partial \tilde{q}}(\tilde{q}_{k-1},\tilde{v}_{k-1}) - \frac{1}{h} (\tilde{p}_{k}-\tilde{p}_{k-1}) + f_L^M(\tilde{q}_{k-1},\tilde{v}_{k-1},t_{k-1})}, \delta\tilde{q}_{k-1}  \right\rangle + \left\langle  {\displaystyle \frac{\partial \mathcal{L}^M}{\partial \tilde v}(\tilde{q}_{k}, \tilde{v}_{k})}-  \tilde{p}_{k}, \delta \tilde{v}_k \right\rangle \right] = 0.
\end{array}
\end{equation}
This gives a slight, but in this case significant, modification for the Euler-Lagrange equations as
\begin{subequations}\label{eq:forwardLag}
 \begin{align}
 \frac{\tilde{q}_{1}-\tilde{q}_{0}}{h}= \tilde{v}_{1}&\label{eq:DEL_bd_3}\\
 \frac{\partial \mathcal{L}^M}{\partial \tilde v}(\tilde{q}_{1}, \tilde{v}_{1})=  \tilde{p}_{1}& \label{eq:DEL_bd_1}\\
\left.\begin{array}{rl}
{\displaystyle \frac{\partial \mathcal{L}^M}{\partial \tilde{q}}(\tilde{q}_{k-1},\tilde{v}_{k-1}) - \frac{1}{h} (\tilde{p}_{k}-\tilde{p}_{k-1}) + f_L^M(\tilde{q}_{k-1},\tilde{v}_{k-1},t_{k-1})}&=0\\
{\displaystyle \frac{\tilde{q}_{k}-\tilde{q}_{k-1}}{h}}&= \tilde{v}_{k}\\
{\displaystyle \frac{\partial \mathcal{L}^M}{\partial \tilde v}(\tilde{q}_{k}, \tilde{v}_{k})}&=  \tilde{p}_{k}
\end{array} \right\} &k=2,\ldots,N \label{eq:DEL_bd_2}
\end{align}
\end{subequations}
Note that in contrast to the variational scheme \eqref{eq:backwardLag} consisting of an explicit update for the charges $q$ and an implicit update for the fluxes $p$, we now get an implicit scheme for $q$ and an explicit scheme for $p$.
In particular, for the Lagrangian \eqref{eq:redL} and the forces \eqref{eq:consforce}, we obtain the following update rule:
For given $(\tilde{q}_0,\tilde{v}_0)$ compute $\tilde{p}_0$  via $\tilde{p}_0=K_2^TLK_2 \tilde{v}_0$. Then, use the iteration scheme
\begin{equation}\label{eq:it_bd}
\left(\begin{matrix}
I & -hI & 0 \\[2pt] 0 & K_2^T L K_2 & -I\\[2pt] 0 & 0 & I
\end{matrix}\right) \left(\begin{matrix} \tilde{q}_k\\[2pt] \tilde{v}_k\\[2pt] \tilde{p}_k\end{matrix}\right) = \left(\begin{matrix}
I & 0 & 0 \\[2pt] 0 & 0 & 0\\[2pt] -h K_2^TCK_2 & -h K_2^T \text{diag}(R)K_2 & I \end{matrix}\right) \left(\begin{matrix}  \tilde{q}_{k-1}\\[2pt] \tilde{v}_{k-1}\\[2pt] \tilde{p}_{k-1}\end{matrix}\right) + \left(\begin{matrix} 0\\[2pt]0\\[2pt]hK_2^T\end{matrix}\right) u_s(t_{k-1})\quad \text{for}\; k=1,\ldots,N
\end{equation}
to compute $\tilde{q}_1,\ldots,\tilde{q}_{N}$, $\tilde{v}_1,\ldots,\tilde{v}_{N}$ and $\tilde{p}_1,\ldots,\tilde{p}_{N}$.
\begin{proposition}\label{prop:back_solve}
System \eqref{eq:it_bd} is uniquely solvable if the matrix $K_2^T L K_2$ is regular.
\end{proposition}

\begin{proof}
System \eqref{eq:it_bd} is uniquely solvable if the iteration matrix $A =\left(\begin{matrix}
I & -hI & 0 \\[2pt] 0 & K_2^T L K_2 & -I\\[2pt] 0 & 0 & I
\end{matrix}\right) $ has zero nullspace. For $Az = 0$ with $z=(\tilde{q},\tilde{v},\tilde{p})$, it holds (i) $\tilde{q}=h\tilde{v}$, (ii) $\tilde{p} = K_2^TLK_2 \tilde{v}$, (iii) $\tilde{p} = 0$. Thus $z=0$ is the unique solution of $Az=0$ iff $K_2^TLK_2$ has zero nullspace.
\end{proof}
Proposition~\ref{prop:back_solve} says that whenever the KCL cancels the degeneracy of the system, the backward Euler scheme is applicable, whereas the forward Euler scheme is applicable to a wider class of circuit systems (cf.~Proposition~\ref{prop:forw_solve}) for $h$ sufficiently large. The resulting variational Euler schemes \eqref{eq:backwardLag} and \eqref{eq:forwardLag} consisting of a combination of implicit and explicit updates are first order variational integrators. The construction of higher order implicit schemes (e.g.~variational partitioned Runge-Kutta (VPRK) methods along the lines of \cite{Bou07}) allows the simulation of arbitrary circuits. As an example, we present in the following a variational integrator based on the implicit midpoint rule. 

\subsection{Implicit Midpoint Rule}\label{sec:midpoint}

We introduce internal stages $\tilde{Q}_k,\tilde{P}_k,\tilde{V}_k,\, k=1,\ldots,N-1$ that are given on a second time grid $\Delta \tau = \{ \tau_k = (k+\frac{1}{2})h\, | \,k=0,\ldots,N-1 \}$ and define the internal stage vectors $\tilde{Q}_d: \{\tau_k\}_{k=0}^{N-1} \rightarrow M$, $\tilde{V}_d: \{\tau_k\}_{k=0}^{N-1} \rightarrow T_{\tilde{q}}M$ and $\tilde{P}_d: \{\tau_k\}_{k=0}^{N-1} \rightarrow T^*_{\tilde{q}}M$ to be $\tilde{V}_k= \tilde{v}(t_k+\frac{1}{2} h), \tilde{Q}_k = \tilde{q}_k + \frac{1}{2}h \tilde{V}_k, \tilde{P}_k = \frac{\partial \mathcal{L}^M}{\partial \tilde v} (\tilde{Q}_k,\tilde{V}_k)$. The approximations at the nodes are then determined by the internal stages via $\tilde{q}_{k+1} = \tilde{q}_k + h \tilde{V}_k$ and $\tilde p_{k+1}= \tilde p_k + h \frac{\partial \mathcal{L}^M}{\partial \tilde{q}}(\tilde{Q}_k,\tilde{V}_k)$.

Taking variations $\delta \tilde{q}_k, \delta\tilde{Q}_k,\delta \tilde{p}_k, \delta\tilde{P}_k,\delta \tilde{
V}_k$ for the following discrete Lagrange-d'Alembert-Pontryagin principle with $\delta q_N=0$ but free $\delta \tilde{q}_0$ and initial value $\tilde{q}^0$
\begin{equation}\label{eq:lapp_discrete_mp}
\begin{array}{c}
{\displaystyle \delta \left\{  h \sum\limits_{k=0}^{N-1} \left( \mathcal{L}^M( \tilde{Q}_{k}, \tilde{V}_{k}) + \left\langle \tilde{P}_k, \frac{\tilde{Q}_{k}-\tilde{q}_{k}}{h}-\frac{1}{2} \tilde{V}_k \right\rangle   + \left\langle \tilde{p}_{k+1}, \frac{\tilde{q}_{k+1}-\tilde{q}_{k}}{h}- \tilde{V}_k \right\rangle \right) + \langle \tilde p_0, \tilde{q}_0-\tilde{q}^0 \rangle \right\}} \\
{\displaystyle + h\sum\limits_{k=0}^{N-1} f_L^M(\tilde{Q}_k,\tilde{V}_k,\tau_k) \delta \tilde{Q}_k=0}
\end{array}
\end{equation}
gives
\begin{equation}
\begin{array}{c}
 \displaystyle{ \sum\limits_{k=0}^{N-1}\left[\left\langle  \frac{\partial \mathcal{L}^M}{\partial \tilde{q}} (\tilde{Q}_k,\tilde{V}_k) + \frac{\tilde{P}_k}{h} +  f_L^M(\tilde{Q}_k,\tilde{V}_k,\tau_k), \delta \tilde{Q}_k  \right\rangle + \left\langle \frac{\partial \mathcal{L}^M}{\partial \tilde{v}} (\tilde{Q}_k,\tilde{V}_k) - \frac{1}{2}\tilde{P}_k - \tilde p_{k+1}, \delta\tilde{V}_k \right\rangle \right.}\\
\left.+ \left\langle  \delta \tilde{P}_k, {\displaystyle \frac{\tilde{Q}_{k}-\tilde{q}_k}{h} - \frac{1}{2} \tilde{V}_k } \right\rangle + \left\langle \delta\tilde{p}_{k+1}, {\displaystyle \frac{\tilde{q}_{k+1}-\tilde{q}_k}{h} -  \tilde{V}_k }\right\rangle + \left\langle {\displaystyle \frac{-\tilde{P}_k - \tilde{p}_{k+1}+\tilde{p}_k}{h}  } ,\delta\tilde{q}_k  \right\rangle\right] + \left\langle \delta\tilde{p}_0 , \tilde q_0-\tilde q^0 \right\rangle = 0.
\end{array}
\end{equation}
The Euler-Lagrange equations are
\begin{subequations}
\begin{align}
\frac{\partial \mathcal{L}^M}{\partial \tilde{q}} (\tilde{Q}_k,\tilde{V}_k) + \frac{\tilde{P}_k}{h} +  f_L^M(\tilde{Q}_k,\tilde{V}_k,\tau_k)& = 0\\
\frac{\partial \mathcal{L}^M}{\partial \tilde{v}} (\tilde{Q}_k,\tilde{V}_k) - \frac{1}{2}\tilde{P}_k - p_{k+1}& =0\\
\frac{\tilde{Q}_{k}-\tilde{q}_k}{h} - \frac{1}{2} \tilde{V}_k & =0 \label{eq:DEL_mp_3}\\
\frac{\tilde{q}_{k+1}-\tilde{q}_k}{h} -  \tilde{V}_k & =0\label{eq:DEL_mp_4}\\
-\tilde{P}_k - \tilde{p}_{k+1}+\tilde{p}_k & =0,\quad k=0,\ldots,N-1\label{eq:DEL_mp_5}\\
\tilde q_0-\tilde q^0 & =0.
\end{align}
\end{subequations}
Eliminating $\tilde{P}_k$ by equation \eqref{eq:DEL_mp_5} together with $\tilde{V}_k = \tilde{v}_{k+\frac{1}{2}}$, $\tilde{Q}_k = \frac{\tilde{q}_k+\tilde{q}_{k+1}}{2}$ (which follows from \eqref{eq:DEL_mp_3} and \eqref{eq:DEL_mp_4}) and $\tau_k= \frac{t_k+t_{k+1}}{2}= t_{k+\frac{1}{2}}$ leads to the iteration scheme
\begin{subequations}\label{eq:DEL_mid}
\begin{align}
\tilde{p}_{k+1}& = \tilde{p}_k + h  \frac{\partial \mathcal{L}^M}{\partial \tilde{q}} \left(\frac{\tilde{q}_k+\tilde{q}_{k+1}}{2},\tilde{v}_{k+\frac{1}{2}}\right) +  h f_L^M\left(\frac{\tilde{q}_k+\tilde{q}_{k+1}}{2},\tilde{v}_{k+\frac{1}{2}}, t_{k+\frac{1}{2}}\right)\\
\tilde{q}_{k+1}& = \tilde{q}_k + h \tilde{v}_{k+\frac{1}{2}}\\
\frac{\tilde{p}_k + \tilde{p}_{k+1} }{2}& = \frac{\partial \mathcal{L}^M}{\partial \tilde{v}} \left(\frac{\tilde{q}_k+\tilde{q}_{k+1}}{2},\tilde{v}_{k+\frac{1}{2}}\right),\quad k=0,\ldots,N-1.
\end{align}
\end{subequations}
\begin{remark}
The integrator \eqref{eq:DEL_mid} is equivalent to a Runge-Kutta scheme with coefficients $a=\frac{1}{2}, b = 1, c=\frac{1}{2}$ (implicit midpoint rule integrator) applied to the corresponding Hamiltonian system.
\end{remark}
For the circuit case with Lagrangian \eqref{eq:redL} and forces \eqref{eq:consforce}, we start with given $(\tilde{q}_0,\tilde{p}_0)$ to solve iteratively for\linebreak $(\tilde{q}_{k+1},\tilde{v}_{k+\frac{1}{2}},\tilde{p}_{k+1}),\, k=0,\ldots,N-1$ for given $u_s(t)$ using the scheme
\begin{equation}\label{eq:it_mp}
\left(\begin{matrix}
I & -hI & 0 \\[2pt] 0 & K_2^T L K_2 & -\frac{1}{2}I\\[2pt] \frac{1}{2}h K_2^TCK_2 & h K_2^T \text{diag}(R)K_2 & I
\end{matrix}\right) \left(\begin{matrix} \tilde{q}_{k+1}\\[2pt]\tilde{v}_{k+\frac{1}{2}}\\[2pt]\tilde{p}_{k+1}\end{matrix}\right) = \left(\begin{matrix}
I & 0 & 0 \\[2pt] 0 & 0 & \frac{1}{2} I\\[2pt] -\frac{1}{2}h K_2^TCK_2 & 0 & I \end{matrix}\right) \left(\begin{matrix} \tilde{q}_{k}\\[2pt]\tilde{v}_{k-\frac{1}{2}}\\[2pt]\tilde{p}_{k}\end{matrix}\right) + \left(\begin{matrix} 0\\[2pt]0\\[2pt]hK_2^T\end{matrix}\right) u_s\left( t_{k+\frac{1}{2}}\right)
\end{equation}
for $k=0,\ldots,N-1$.
\begin{remark}
The discrete current $\tilde{v}_{k+\frac{1}{2}}$, playing the role of the algebraic variable in the continuous setting,  is only approximated between two discrete time nodes $t_k$ and $t_{k+1}$. Also, note that $\tilde{v}_{k-\frac{1}{2}}$ is not explicitly used for the computation of $(\tilde{q}_{k+1},\tilde{v}_{k+\frac{1}{2}},\tilde{p}_{k+1})$ (which corresponds to a zero column in the matrix of the right hand side of \eqref{eq:it_mp}). This means, that the computation of the magnitudes at time point $t_{k+1}$ depends only on the discrete magnitudes within the time interval $[t_k,t_{k+1}]$, which is characteristic for a one-step scheme. In particular, $\tilde{v}_{-\frac{1}{2}} (k=0)$ is a pseudo-variable that is not used.
\end{remark}
\begin{proposition}\label{prop:mid_solve}
System \eqref{eq:it_mp} is uniquely solvable if the matrix $K_2^T (2L + h\text{\emph{diag}}(R) + \frac{1}{2}h^2C ) K_2$ is regular.
\end{proposition}

\begin{proof}
System \eqref{eq:it_mp} is uniquely solvable if the iteration matrix $A = \left(\begin{matrix}
I & -hI & 0 \\[2pt] 0 & K_2^T L K_2 & -\frac{1}{2}I\\[2pt] \frac{1}{2}h K_2^TCK_2 & h K_2^T \text{diag}(R)K_2 & I
\end{matrix}\right)$ has zero nullspace. For $Az = 0$ with $z=(\tilde{q},\tilde{v},\tilde{p})$, it holds (i) $\tilde{q}=h\tilde{v}$, (ii) $\tilde{p} = 2 K_2^TLK_2 \tilde{v}$, (iii) $\frac{1}{2}h K_2^TCK_2 \tilde{q} +  h K_2^T \text{diag}(R)K_2 \tilde{v} + \tilde{p} = 0$. Substituting (i) and (ii) in (iii) gives $K_2^T(2 L+ h \text{diag}(R) + \frac{1}{2}h^2 C)K_2 \tilde{v}  = 0$
Thus, $z=0$ is the unique solution of $Az=0$ if $K_2^T(2 L+ h \text{diag}(R) + \frac{1}{2}h^2 C)K_2$ has zero nullspace.
\end{proof}
Note that for linear circuits, the condition given in Proposition~\ref{prop:mid_solve} is fulfilled for $h$ sufficiently large, if the continuous system \eqref{EL_Kcirc} has a unique solution.

\begin{remark}[Condition numbers]
From Proposition~\ref{prop:degeneracy}, we see that the continuous reduced system \eqref{eq:ELred_general} may be still degenerate due to the intrinsic degeneracy of the circuit topology and configuration. On the discrete side, both the forward Euler and the midpoint integrator show some regularization property: by perturbing $K_2^TLK_2$ in magnitude proportional to the time step size $h$, both integrators can render the degenerate continuous reduced system \eqref{eq:ELred_general} into regular discrete systems \eqref{eq:it_fd} and \eqref{eq:it_mp}, respectively. 
However, this regularization comes at the price of possible large condition numbers as explained in the following.
The iteration matrices $A$ of the different schemes can be written as $A=A_0 + hE$ with $A_0 = \left( \begin{array}{ccc}  I &0& 0 \\ 0& K_2^TLK_2& -I \\ 0& 0& I \end{array}\right)$ and $E$ given by the respective iteration scheme. If the reduced system is regular (i.e.~$K_2^TLK_2$ is non singular),  $A_0$ is non singular with positive constant condition number $\kappa(A_0)$ and $\kappa(A)$ approaches a positive constant when the step size $h$ goes to zero (e.g.~one can compute the singular values of the perturbed matrix $A_0+hE$ using arguments from perturbation theory). In this case all iteration schemes are well conditioned independent of the step size $h$. However, if the reduced system is degenerate, i.e.~$K_2^TLK_2$ is singular, also $A_0$ is singular and the condition numbers of the forward Euler and the midpoint scheme grow reciprocally to the time step size $h$, i.e.~$\kappa(A)\approx \mathcal{O}(1/h)$ for small $h$. 
When the circuit topology is fixed, the circuit's physical parameters are constants and the time step $h$ is also fixed, preconditioner can be precomputed and applied to the systems \eqref{eq:it_fd} and \eqref{eq:it_mp} to improve their numerical stabilities. Since this work mainly focus on the theoretical aspects of variational integrators, we leave this stability issues for future work and assume thereafter no such issues in the subsequent discussion. 
\end{remark}

\section{Structure-preserving properties}\label{sec:structure}

In this section, we summarize the main structure-preserving properties 
of variational integrators (see e.g.~\cite{MaWe01}) and their interpretation for the case of electric circuits.

\subsection{Symplecticity and preservation of momentum maps induced by symmetries}

\paragraph{Symplecticity}
The flow on $T^*Q$ of the Euler-Lagrange equations preserves the canonical symplectic form of the Hamiltonian system. Variational integrators are symplectic, i.e.~the same property holds for the discrete flow of the discrete Euler-Lagrange equations; the canonical symplectic form is exactly preserved for the discrete solution. 
Using techniques from backward error analysis (see e.g.~\cite{HaLuWa}), it can be shown that symplectic integrators also have good energy properties, i.e.~for long-time integrations, there is no artificial energy growth or decay due to numerical errors. This can also be observed for our circuit examples in Section \ref{sex:example}. In the case of linear LC circuits that involve a quadratic potential, a second order variational integrator, e.g.~the midpoint variational integrator as derived in Section \ref{sec:midpoint}, even exactly preserves the energy (magnetic plus electric energy).

\paragraph{Preservation of momentum maps induced by symmetries}

Noether's theorem states that momentum maps that are induced by symmetries in the system are preserved.
More precisely, let $G$ be a Lie group acting on $Q$ by $\Phi:G\times Q\rightarrow Q$. We write $\Phi_g:= \Phi(g,\cdot)$. The tangent lift of this action $\Phi^{TQ}: G\times TQ\rightarrow TQ$ is given by $\Phi_g^{TQ}(v_q) = T(\Phi_g)\cdot v_q$ with $v_q\in TQ$.
The action is associated with a corresponding momentum map $J:TQ \rightarrow \mathfrak{g}^*$, where $\mathfrak{g}^*$ is the dual of the Lie algebra $\mathfrak{g}$ of $G$. The momentum map is defined by
\[
\langle J(q,v),\xi \rangle = \left\langle \frac{\partial \mathcal{L}}{\partial v}, \xi_Q(q)\right\rangle = \langle p, \xi_Q(q)\rangle 
\quad \forall \xi\in \mathfrak{g},\]
where $\xi_Q$ is the infinitesimal generator of the action on $Q$, i.e.~$\xi_Q(q) :=\left. \frac{d}{dt} \right|_{t=0} \Phi(\exp{(t\xi)} , q)$ and $\exp: \mathfrak{g} \rightarrow G$ is the exponential function.
Having a holonomic system described by a Lagrangian $\mathcal{L}$ and a holonomic constraint $h(q)=0$, then this system has a symmetry if the Lagrangian and the constraint are both invariant under the (lift of the) group action, i.e.~$\mathcal{L}\circ\Phi_g^{TQ} = \mathcal{L}$ and $h\circ \Phi_g(q) =0$ for all $g\in G$. Noether's theorem states, that if the system has a symmetry, the corresponding momentum map is preserved.
In presence of external forces, this statement is still true, if the force is orthogonal to the group action.
The discrete version of Noether's theorem \cite{MaWe01} states, that if the discrete Lagrangian has a symmetry, the corresponding momentum map is still preserved. The variational integrator based on this discrete Lagrangian is thus exactly momentum-preserving. For constrained and forced systems, the preservation still holds with the additional invariance and orthogonality conditions on constraints and forces analogous to the continuous case. 

Considering an electric circuit, we are faced with a constrained distribution given by the KCL and external forces due to resistors and voltage sources. The KCL are formulated on the tangent space; however, since these are linear, they are integrable resulting in KCL on the configuration space. Thus, in the following, we are able to apply the theory of holonomic systems to derive a preserved quantity for an eletrical circuit under some topology assumptions of the underlying graph.

\begin{proposition}[Invariance of Lagrangian]\label{prop:Linv}
The Lagrangian \eqref{Lagrangian} of the unreduced system is invariant under the translation of $q_L$.
\end{proposition}
\begin{proof}
Consider the group $G=\mathbb{R}^{n_L}$ with group element $g \in G$. 
Let  $\Phi:G\times Q \rightarrow Q$ be the action of $G$ defined as $\Phi(g,q)= (q_L+g,q_C)$ for each $g\in G$ with tangent lift $\Phi^{TQ}:G\times TQ\rightarrow TQ, \Phi^{TQ}(g,(q,v))= (q_L+g,q_C,q_R,q_V,v_L,v_C,v_R,v_V)$. 
Then, it holds
\begin{align*}
\mathcal{L} \circ \Phi^{TQ}_g(q,v)& = \frac{1}{2} \left( \begin{array}{c} v_L \\ v_C  \\ v_R \\v_V \end{array} \right)^T L \left( \begin{array}{c} v_L \\ v_C \\ v_R \\v_V \end{array} \right) - \frac{1}{2}\left( \begin{array}{c} q_L+g \\ q_C  \\ q_R \\q_V\end{array} \right)^T C \left( \begin{array}{c} q_L+g_L \\ q_C \\ q_R \\q_V \end{array} \right) \\
&= \frac{1}{2} v^T L v - \frac{1}{2} q^T C q = \mathcal{L}(q,v),
\end{align*}
since $C= \text{diag}\left(\frac{1}{C_1},\ldots,\frac{1}{C_n}\right)$ with the first $n_L$ diagonal elements being zero.
\end{proof}

\begin{assumption}[Topology assumption]\label{ass:top}
For every node $j,\,j=1,\ldots,m$ in the circuit (except ground), the same amount of inductor branches connect inward and outward to node $j$.
\end{assumption}
In particular, Assumption \ref{ass:top} implies, that the sum of each row of $K_L^T$ is zero, i.e.~$\sum_{j=1}^{n_L} (K_L^T)_{ij} =0$ for $i=1,\ldots,m$. 
\begin{proposition}[Invariance of distribution]\label{prop:KCL}
Under Assumption \ref{ass:top}, the KCL on configuration level are invariant under equal translation of $q_L$.
\end{proposition}
\begin{proof}
The group element $g\in G$ describing an equal translation of all components of $q_L$ can be expressed as $g=a \mathbf{1}$ with $a\in \mathbb{R}$ and $\mathbf{1}$ being a vector in $\mathbb{R}^{n_L}$ with each component $1$. It follows
\begin{align*}
K^T\circ \Phi_g(q) &= K^T \left( \begin{array}{c} q_L+g \\ q_C \\q_R \\q_V \end{array} \right) = K^T q + K_L^T g = K^T q + K_L^T \mathbf{1} a = K^T q,
\end{align*}
since the sum of each row of $K_L^T$ is zero.
\end{proof}
 
 \begin{proposition}[Orthogonality of external force]\label{prop:forth}
The external force $f_L$ \eqref{eq_fl} is orthogonal to the action of the group $G=\mathbb{R}^{n_L}$ being translations of $q_L$.
 \end{proposition}
 \begin{proof}
Let $\xi \in \mathfrak{g}=\mathbb{R}^{n_L}$ be an element of the Lie algebra. For the group action $\Phi_g(q) = (q_L+g,q_C,q_R,q_V)$, the infinitesimal generator can be calculated as
\[
\xi_Q(q) = \left. \frac{d}{dt}\right|_{t=0} \Phi_{\exp{t\xi}}(q) = \left. \frac{d}{dt}\right|_{t=0} (q_L+\exp{t\xi},q_C,q_R,q_V) = (\xi,0,0,0)
\]
It thus holds,
\[
\langle f_L,\xi_Q(q)\rangle =\langle -\text{diag}(R)v + \text{diag}(\mathcal{E}) u,\xi_Q(q)\rangle =0,
\]
since $\text{diag}(R)$ and $\text{diag}(\mathcal{E})$ have zero entries in the first $n_L$ lines and columns.
 \end{proof}
 
\begin{theorem}[Preservation of flux]\label{theo:neother}
Under Assumption~\ref{ass:top}, the sum of all inductor fluxes in the electric circuit described by the Lagrangian \eqref{Lagrangian}, the external forces \eqref{eq_fl}, and the KCL is preserved.
\end{theorem}
\begin{proof}
From Proposition \ref{prop:Linv}, \ref{prop:KCL}, and \ref{prop:forth} we know that the Lagrangian and the KCL are invariant under the group action $\Phi_g(q) = (q_L+g,q_C,q_R,q_V)$ with $g\in G=\mathbb{R}^{n_L}$ and the external force $f_L$ is orthogonal to this group action. 
It follows with Noether's theorem, that the induced momentum map is preserved by the flow of the system.
For the momentum map, we calculate
\[
\langle J(q,v),\xi \rangle =\left\langle\frac{\partial \mathcal{L}}{\partial v}, \xi_Q(q) \right\rangle = \frac{\partial \mathcal{L}}{\partial v_i} \xi_Q^i(q) =  \frac{\partial \mathcal{L}}{\partial v_{L_i}} \xi^i.
\]
Thus, the preserved momentum map is $J(q,v) = \sum_{i=1}^{n_L} p_{n_{L_i}}$, i.e.~the sum of the fluxes of all inductors in the circuit.
\end{proof}

\begin{remark}[Proof based on Euler-Lagrange equations]
An alternative proof can be derived based on the Euler-Lagrange equations in the following way.
From \eqref{eq:EL_circuit_full_p}, it holds
\[
\dot{p}_L = K_L \lambda.
\]
For the time derivative of the sum of all inductors, it follows
\begin{align*}
\frac{d}{dt} \sum_{i=1}^{n_L} p_{i} &= \sum_{i=1}^{n_L} \dot{p}_{i} = \sum_{i=1}^{n_L} \sum_{j=1}^{m} (K_L)_{ij} \lambda_j = \sum_{j=1}^{m} \lambda_j  \sum_{i=1}^{n_L}(K_L)_{ij} =0,
\end{align*} 
since with Assumption \ref{ass:top}, it holds $0=  \sum_{i=1}^{n_L} (K_L^T)_{ji} =  \sum_{i=1}^{n_L} (K_L)_{ij}$ for $j=1,\ldots,m$. Thus, $\sum_{i=1}^{n_L} p_{i}$ is preserved.
\end{remark}

\begin{remark}[Momentum map for reduced system]
Also, for the reduced system described by the Lagrangian \eqref{eq:redL}, the same momentum map can be computed by considering the group action $\Phi_{\tilde{g}}(\tilde{q}) = \tilde{q} +\tilde{g}$ with the group element $\tilde{g}\in \tilde{G}\subset\mathbb{R}^{n-m}$ defined as $\tilde{g} = K_2^+ \left(\begin{array}{c} g \\ 0 \end{array}\right)$ with $K_2^+$ being the well-defined pseudo-inverse of $K_2$. 
\end{remark}
\begin{lemma}[Preserved momentum map]
For any linear circuit described by the Lagrangian \eqref{Lagrangian}, the external forces \eqref{eq_fl}, and the KCL, the momentum map defined
by $\eta^T \frac{\partial \mathcal{L}}{\partial v}$ with $\eta \in ker(K_L^T)$ is preserved.
\end{lemma}
\begin{proof}
Using the Euler-Lagrange equations, we see immediately
\[
\frac{d}{dt} \eta^T \frac{\partial \mathcal{L}}{\partial v}= \eta^T \dot{p} = \eta^T K_L \lambda =0,
\] 
since $\eta^T \in ker(K_L^T) \perp im(K_L) \ni K_L \lambda$ and thus, $\eta^T \frac{\partial \mathcal{L}}{\partial v}=const$.
\end{proof}

The discrete Lagrangian system, including constraints and forces introduced in Section~\ref{sec:disvar}, inherits the same symmetry and orthogonality property as the continuous system. Due to the discrete Noether theorem, the resulting variational integrators exactly preserves the sum of inductor fluxes under Assumption~\ref{ass:top} (compare Section~\ref{subsec:tl} for a numerical example).

\subsection{Frequency spectrum}\label{subsec:frequ}
As can be observed in numerical examples (see Section~\ref{sex:example}), the frequency spectrum of the discrete solutions is much better preserved using variational integrators than other integrators.

We want to analytically demonstrate this phenomenon by means of a simple harmonic $1$d oscillator. Assume the curves $(q(t),p(t))$ on $[0,T]$ describe the oscillatory behavior of the system.
Consider the discrete solution $\{(q_k,p_k)\}_{k=0}^N$ defined on the discrete time grid $\{t_k\}_{k=0}^n$ with $t_0=0$, $t_N = T$ and $h= t_{k+1}-t_k$ that is obtained from the one-step update scheme $(q_{k+1}, p_{k+1})^T = A ( q_{k}, p_{k})^T,\, k=0,\ldots,N-1$, where $A\in \mathbb{R}^{2,2}$ depends on the constant time step $h$.
We assume that the discrete solution $\{(q_k,p_k)\}_{k=0}^N$ converges to the solution $(q(t),p(t))$ for decreasing  $h$.
Since this solution is oscillating and due to the convergence of the scheme, at least one eigenvalue $\lambda_1$ of $A$ has to be complex (with nonzero imaginary part) for a small enough time step $h$. Since $A\in\mathbb{R}^{2,2}$ the second eigenvalue $\lambda_2$ has to be complex conjugate to the first one. Thus, the corresponding eigenvectors are linearly independent and
$A$ is diagonalizable as $A = QVQ^{-1}$ with $V = \text{diag}(\lambda_1,\lambda_2)$.
With the coordinate transformation $(  x_{k}, y_{k})^T = Q^{-1} (q_{k}, p_{k})^T$ it holds $(x_{k+1}, y_{k+1})^T = V ( x_{k}, y_{k})^T$, i.e.~$x_{k+1}=\lambda_1 x_k$ and $y_{k+1}= \lambda_2 y_k$.

We demonstrate the preservation of the frequency spectrum for the $1$d oscillator in two steps: (i) We show that for a convergent scheme the update matrix $A$ has two eigenvalues both of norm $1$ if and only if the update scheme is symplectic. (ii) We show that methods defined by matrices with norm $1$ eigenvalues preserve the frequency spectrum defined on different time spans. 
\begin{itemize}
\item[(i)] ``$\Leftarrow$'': Assume the scheme defined by $A$ is symplectic, then $\text{det}(A)=1$ (see e.g.~\cite{MarRat94}). It follows with $\lambda_1$ complex conjugate to $\lambda_2$ ($\lambda_2 = \lambda_1^*$): $1 =\text{det}(Q)\cdot \text{det}(V)\cdot \text{det}(Q^{-1})= \lambda_1\cdot \lambda_2 = | \lambda_1 |^2=  | \lambda_2|^2$ and thus $|\lambda_i|=1$, $i=1,2$. ``$\Rightarrow$'':
Assume $A$ has two complex conjugate eigenvalues $\lambda_1=\lambda_2^*$ with $|\lambda_1| = |\lambda_2| =1$, i.e.~we write $\lambda_1 = e^{i\theta}$ and $\lambda_2 = e^{-i\theta}$ with $\theta\in\mathbb{R}$ and $V = \text{diag}( e^{i\theta }, e^{-i\theta })$. Note that $\theta$ depends on the constant time step $h$ that is used for the discretization. Let $J=\left(\begin{array}{cc} 0&1\\-1& 0\end{array}\right)$ be the canonical symplectic form and introduce the non-canonical symplectic form $\tilde{J}=Q^TJQ$. We show that $V$ preserves $\tilde{J}$, and therefore $A$ preserves $J$, i.e.~$A$ is symplectic. Since $J$ is skew-symmetric with zero diagonal, $\tilde{J}$ is of the form $\left(\begin{array}{cc} 0 & \triangle \\ -\triangle & 0 \end{array}\right)$ with $\triangle \in\mathbb{R}$. It follows
\[ V^T \tilde{J} V = 
    \left(\begin{array}{cc}  e^{i\theta} & 0 \\ 0 & e^{-i\theta } \end{array}\right) \left(\begin{array}{cc} 0 & \triangle \\ -\triangle & 0 \end{array}\right)  \left(\begin{array}{cc}  e^{i\theta} & 0 \\ 0 & e^{-i\theta } \end{array}\right) = \left(\begin{array}{cc} 0 & e^{i\theta }e^{-i\theta } \triangle \\ -e^{-i\theta}e^{i\theta }\triangle & 0 \end{array}\right) = \left(\begin{array}{cc} 0 & \triangle \\ -\triangle & 0 \end{array}\right)= \tilde{J}.
\]

\item[(ii)]
Suppose that the discrete values $x_1,x_2,\ldots,x_N$ determined by the update scheme $A$ are known, and admit the following discrete inverse Fourier transformation

\[x_k=\frac{1}{N}\sum_{n=1}^N \tilde{x}_n \exp\left(\frac{2\pi i}{N} k n\right),\; k=1,\ldots,N.\]
Consider a sequence of discrete points $\{X_k\}_{k=1}^N$ that is shifted by one time step such that $X_k = x_{k+1}= \lambda_1 x_k,\; k=1,\ldots,N$, i.e.~$\{X_k\}_{k=1}^N$ approximates the solution on a later time interval than $\{x_k\}_{k=1}^N$.
This admits the following discrete inverse Fourier transformation

\[X_k=\frac{1}{N}\sum_{n=1}^N \lambda_{1} \tilde{x}_n \exp\left(\frac{2\pi i}{N} k n\right),\; k=1,\ldots,N,\]
i.e., $\tilde{X}_n=\lambda_{1} \tilde{x}_n$.
By the definition of the frequency spectrum, it holds $\tilde{X}_n^* \tilde{X}_n = \tilde{x}_n^* \lambda_{1}^* \lambda_{1} \tilde{x}_n =  \tilde{x}_n^* |\lambda_{1}|^2  \tilde{x}_n=\tilde{x}_n^* \tilde{x}_n$, where the last equality relies on the symplecticity. Shifting the discrete solution arbitrary times, we see, that the spectrum will be preserved using different time intervals for the frequency analysis. This means that, in particular for long-time integration, a frequency analysis on a later time interval yields the same results as on an earlier time interval, which we denote by \emph{preservation of the frequency spectrum}. 
The analysis for $y$ follows analogously, and with the linear transformation $Q$ the same holds for $q$ and $p$.
On the other hand, if $|\lambda_{i,j}|\neq 1,\,i,j=1,2$ (such as for non-symplectic or non-convergent methods), the frequency spectrum will either shrink or grow unbounded.
\end{itemize}

Although the analysis was only performed for the simple case of a $1$d harmonic oscillator (in particular statement (i) is restricted to this case), we believe that 
for higher-dimensional systems, a similar statement as in (ii) can also be shown, which is left for future
work.

\paragraph{Relation to numerical results}
For the numerical computations in Section~\ref{sex:example}, we perform a frequency analysis in two different ways:
Firstly, we calculate the frequency spectrum on different time subintervals of the overall time integration interval $[0,T]$.
This is directly connected to the analytical result of frequency preservation, i.e.~we can observe that the spectrum is independent on the specific time interval using a symplectic integrator; however, using a non-symplectic method, the spectrum 
is damped calculated on a later time interval.
Secondly, we use a fixed time interval $[0,T]$ for the frequency analysis, but use different time steps resulting in different iteration matrices $A$. As we saw for the harmonic oscillator using a symplectic method, the magnitude of the two eigenvalues is independent on the time step $h$ where it might de- or increase for increasing $h$ for a non-symplectic method (e.g.~using the explicit Euler method, the absolute value of the eigenvalues is $1+\mathcal{O}(h^2)$ and the frequency spectrum would grow for larger $h$).

\section{Noisy circuits}\label{sec:noise}

In this section, we extend the constructed variational integrator to the simulation of noisy electric circuits, for which noise is added to each edge of the circuit.
Following the description in \cite{BRO08}, in the stochastic setting, the constrained stochastic variational principle is
\begin{equation}
    \delta \int_0^T \mathcal{L}({q}(t),{v}(t)) + \left\langle {p}(t),\dot{{q}}(t) - {v}(t) \right\rangle \,dt + \int_0^T f_L({q}(t),{v}(t),t) \cdot \delta q(t) \,dt + \int_0^T \delta q(t) \cdot (\Sigma \circ dW_t)=0
\end{equation}
with constrained variations $\delta q \in \Delta_Q(q)$, where $\Sigma$ is a $n\times n$ matrix, usually constant and diagonal, indicating the amplitude of noise at each edge, $W_t$ is a $n$-dimensional Brownian motion, and the last stochastic integral is in the sense of Stratonovich.
This principle leads to the constrained stochastic differential equation
\begin{subequations}\label{eq:EL_full_noise}
\begin{align}
 \frac{\partial \mathcal{L}}{\partial q}  -\dot{p} + f_L  + \Sigma \circ \frac{dW_t}{dt} & \in \Delta^0_Q(q) \label{EL_K_1_n}\\
dq & =  {v} dt\label{EL_K_2_n}\\
 \frac{\partial \mathcal{L}}{\partial v} -{p}  & =0\label{EL_K_3_n}\\
 K^T v & = 0, \label{EL_K_4_n}
\end{align}
\end{subequations}
where by \eqref{EL_K_1_n} we mean that it holds $\displaystyle \int_0^T \left (\frac{\partial \mathcal{L}}{\partial q} dt -d{p} + f_L dt + \Sigma \circ dW_t \right) = \int_0^T \mathcal{X}(q)\, dt$ for a vector field $\mathcal{X}(q) \in  \Delta^0_Q(q)$ for any $T$.
Correspondingly, the reduced stochastic variational principle reads
\begin{equation}
\delta \int_0^T \mathcal{L}^M(\tilde{q}(t),\tilde{v}(t)) + \left\langle \tilde{p}(t),\dot{\tilde{q}}(t) - \tilde{v}(t) \right\rangle \,dt + \int_0^T f^M_L(\tilde{q}(t),\tilde{v}(t),t) \cdot \delta \tilde q(t) \,dt + \int_0^T \delta \tilde{q}(t) \cdot (K_2^T \Sigma \circ dW_t)=0.
\end{equation}
This results in the reduced stochastic Euler-Lagrange equations
\begin{subequations}\label{ELc_K_noise}
\begin{align}
 \frac{\partial \mathcal{L}^M}{\partial \tilde q} dt -d{\tilde p} + f^M_L dt + K_2^T \Sigma \circ dW_t & =0\label{ELc_K_1_noise}\\
d{\tilde{q}} & =  \tilde{v} dt\label{ELc_K_2_noise}\\
 \frac{\partial \mathcal{L}^M}{\partial \tilde v} -\tilde{p}  & =0.\label{ELc_K_3_noise}
\end{align}
\end{subequations}

To derive the discrete equations with noise, the Stratonovich integral is approximated by a discrete version. For simplicity, we present the equations for the forward Euler iteration scheme only. On the interval $[t_k,t_{k+1}]$
the integral $ \int_{t_k}^{t_{k+1}} \delta \tilde{q}(t) \cdot (K_2^T \Sigma \circ dW_t)$ is approximated by the discrete expression $ \delta \tilde{q}_k \cdot (K_2^T \Sigma ) B^k$ with $B^k \sim \mathcal{N}(0,h)$, $k=0,\ldots,N-1$ (see also \cite{BRO08}).
In this way, we obtain the following reduced stochastic discrete variational principle 

\begin{equation}
\delta \left\{  h \sum\limits_{k=0}^{N-1} \left( \mathcal{L}^M( \tilde{q}_k, \tilde{v}_k) + \left\langle \tilde{p}_k, \frac{\tilde{q}_{k+1}-\tilde{q}_k}{h}- \tilde{v}_k \right\rangle \right) \right\} + h\sum\limits_{k=0}^{N-1}  f^M_L(\tilde{q}_k,\tilde{v}_k,t_k) \delta \tilde{q}_k + \sqrt{h}\sum\limits_{k=0}^{N-1} K_2^T \Sigma \xi_k \cdot \delta \tilde{q}_k =0,
\end{equation}
where for each $k=0,\ldots,N-1$, $\xi_k$ is a $n$-dimensional vector with entries being independent standard normal random variables.
The discrete reduced stochastic Euler-Lagrange equations that give the symplectic forward Euler iteration scheme is then given by
\[
 \begin{array}{rl}
{\displaystyle \frac{\partial \mathcal{L}^M}{\partial \tilde{q}}(\tilde{q}_{k},\tilde{v}_{k}) - \frac{1}{h} (\tilde{p}_{k}-\tilde{p}_{k-1}) +  f_L^M(\tilde{q}_k,{v}_{k},t_k)} +  \frac{1}{\sqrt{h}} K_2^T \Sigma \xi_k&=0\\
{\displaystyle \frac{\tilde{q}_{k}-\tilde{q}_{k-1}}{h}}&= \tilde{v}_{k-1}\\
{\displaystyle \frac{\partial \mathcal{L}^M}{\partial v}(\tilde{q}_{k}, \tilde{v}_{k})}&=  \tilde{p}_{k},\quad k=1,\ldots,N.
\end{array} 
\]
Different symplectic variational schemes (e.g.~backward Euler or midpoint scheme) can be derived in the same way as in Section~\ref{sec:disvar} with an appropriate discretization for the Stratonovich integral.
In \cite{BRO08}, it is shown that the stochastic flow of a stochastic mechanical system
on $T^*Q$ preserves the canonical symplectic form almost surely (i.e., with probability one with respect to the noise). Furthermore, an extension of Noether's theorem says that in presence of symmetries of the Lagrangian, the corresponding momentum map is preserved almost surely.

\section{Examples}\label{sex:example}
In the following section, we demonstrate the variational integration scheme by means of simple circuit examples.
The numerical results are compared with solutions resulting from standard modeling and simulation techniques from circuit theory. 
In particular, we compare the variational integrator results based on Lagrangian models with solutions obtained with a Runge-Kutta scheme of fourth order as well with solutions obtained by applying Backward Differentiation Formula (BDF) methods to models derived using the Modified Nodal Analysis (MNA). For all methods, we use a constant step size $h$.

For all examples, we use the convention, that the charge vector $q\in \mathbb{R}^n$ is ordered as $q = (q_L, q_C, q_R, q_V)$, and correspondingly the current, voltage, and linkage flux vectors as well as the Kirchhoff Constraint and the Fundamental Loop matrix.

\subsection{Short introduction to MNA}\label{subsec:comp}

In most circuit simulators, the Modified Nodal
Analysis (MNA) is used to assemble the system of equations. In the following, we
present the standard modified nodal analysis. We follow the description in \cite{Voigtmann06}. A more
detailed description can be found, for example, in \cite{Voigtmann06, GFtM05, Baec07}.
The MNA consists of three steps: $1.$ Apply the Kirchhoff Constraint Law to every node except the ground.
$2.$ Insert the representation for the branch current of resistors, capacitors and current sources.
$3.$ Add the representation for inductors and voltage sources 
explicitly to the system.

The combination of Kirchhoff's laws and the characteristic equations of the different elements yields the system of differential and algebraic equations
\begin{subequations}\label{eq:MNA_nlin}
 \begin{eqnarray}
K^T_C  \dot{q}_C(K_C \hat{u},t) + K^T_R g(K_R \hat{u},t) + K^T_L v_L + K^T_V v_V + K^T_I v_I(K \hat{u},\dot{q}_C(K_C \hat{u},t),v_L,v_V,t) &=&0\\
\dot{p}_L(v_L,t) -K_L \hat{u} &=& 0\\
u_V(K\hat{u},\dot{q}_C(K_C \hat{u},t),v_L,v_V,t)-K_V \hat{u} &=&0
\end{eqnarray}
\end{subequations}
with node voltages $\hat{u}$, branch currents through voltage and flux controlled elements $v_V$ and $v_L$,
voltage dependent charges through capacitors $q_C$ and current dependent fluxes through inductors $p_L$,
voltage dependent conductance $g$, and controlled current and voltage sources $v_I$ and $u_V$.
System \eqref{eq:MNA_nlin} can be rewritten in compact form as
\begin{equation}\label{eq:MNA_DAE}
A [d(x(t),t)]' + b(x(t),t) =0
\end{equation}
with
\[x = \left[\begin{array}{c} \hat{u} \\ v_L \\ v_V \end{array}\right],\quad A =   \left[\begin{array}{cc} K_C^T & 0 \\ 0 & I \\ 0 & 0 \end{array}\right],\quad d(x,t) =  \left[\begin{array}{c}  {q}_C(K_C \hat{u},t) \\ {p}_L(v_L,t) \end{array}\right]\]
and the obvious definition of $b$. The prime $[d(x, t)]' = \frac{d}{dt} [d(x(t), t)]$ denotes differentiation with respect to time. Since the matrix $A \frac{\partial d(x,t)}{\partial x}$ might be singular, equation \eqref{eq:MNA_DAE} is not an ordinary differential equation but of differential algebraic type. For a detailed description of the properties of these equations, we refer to e.g.~\cite{KM06}.

The standard approach to numerically solve the system of equations \eqref{eq:MNA_DAE} is to apply implicit multistep
methods for the time discretization, in particular lower-order BDF schemes or the
trapezoidal rule. For a detailed description of these methods, we refer to e.g.~\cite{Voigtmann06, HW91}.
The advantage of BDF methods is the low computational cost compared to implicit Runge-Kutta methods.
However, they may have bad stability properties (cf.~\cite{Voigtmann06}) and, in particular, they are not symplectic.

\subsection{RLC circuit}

Consider the graph consisting of four boundary edges and two diagonal edges of a square (see Figure \ref{fig:circuit_graph}). On each edge of this graph, we have a pair of capacitor (with capacitance $C_i=1$, $i=1,\ldots,6$) and inductor (with inductance $L_i=1$, $i=1,\ldots,5$) except on one edge.\footnote{The values for inductance and capacitance are just chosen for demonstration purpose of the variational integrator. For real circuits, the values are typically of different order, resulting in a dynamical behavior on a different scale as in our numerical examples.} On this edge, there is only one capacitor which leaves a degenerate Lagrangian.
The corresponding planar graph consists of $n=6$ branches and $m+1=4$ nodes, thus we have $l = 3$ meshes.  

\begin{figure}[htb]
 \centering
\includegraphics[width=0.5\textwidth]{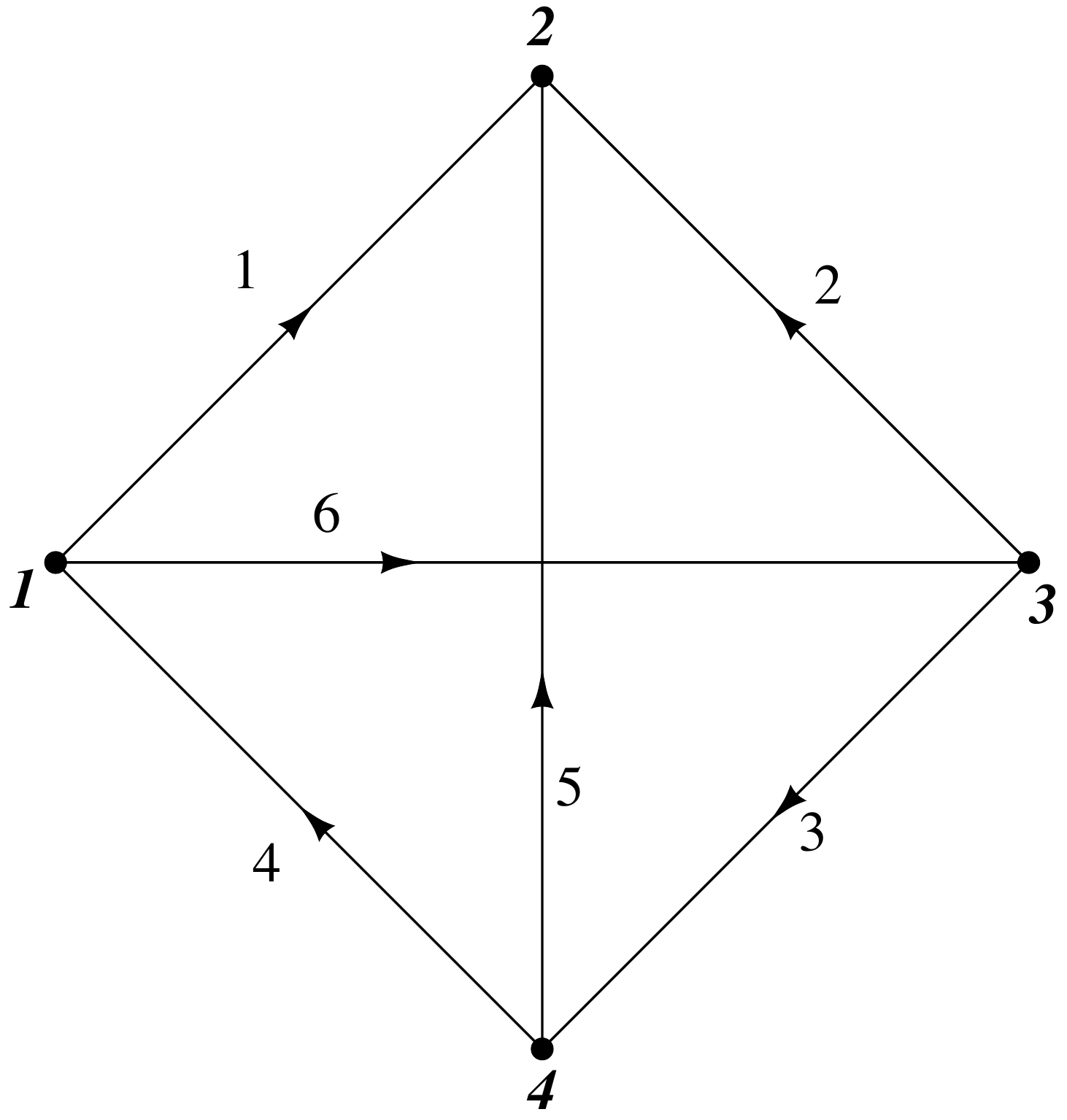}
 \caption{Graph representation of a RLC circuit.}
 \label{fig:circuit_graph}
\end{figure}

The matrix Kirchhoff Constraint matrix $K\in \mathbb{R}^{n,m}$ and the Fundamental Loop matrix $K_2\in \mathbb{R}^{n,n-m}$ are (with the fourth node assumed to be grounded)

\begin{equation}
K = \left(\begin{array}{ccc} 1 & -1 & 0\\ 0 & -1 & 1 \\ 0 & 0 & 1 \\ -1 & 0 & 0\\ 0 & -1 & 0 \\ 1 & 0 & -1  \end{array}\right), \quad\quad 
K_2 = \left(\begin{array}{ccc} 1 & 0 & -1\\ 0 & -1 & 1 \\ 0 & 1 & 0 \\ 1 & 0 & 0\\ -1 & 1 & 0 \\ 0 & 0 & 1  \end{array}\right).
\end{equation}

The matrix $K_2^TLK_2$ is non-singular with $L = \text{diag}(L_1,\ldots,L_5,0)$; thus, the degeneracy of the system is eliminated by the constraints on the system, and all three variational integrators derived in Section~\ref{sec:disvar} can be applied.

\begin{figure}[htb]
 \centering
\begin{tabular}{cc}
\includegraphics[width=0.45\textwidth]{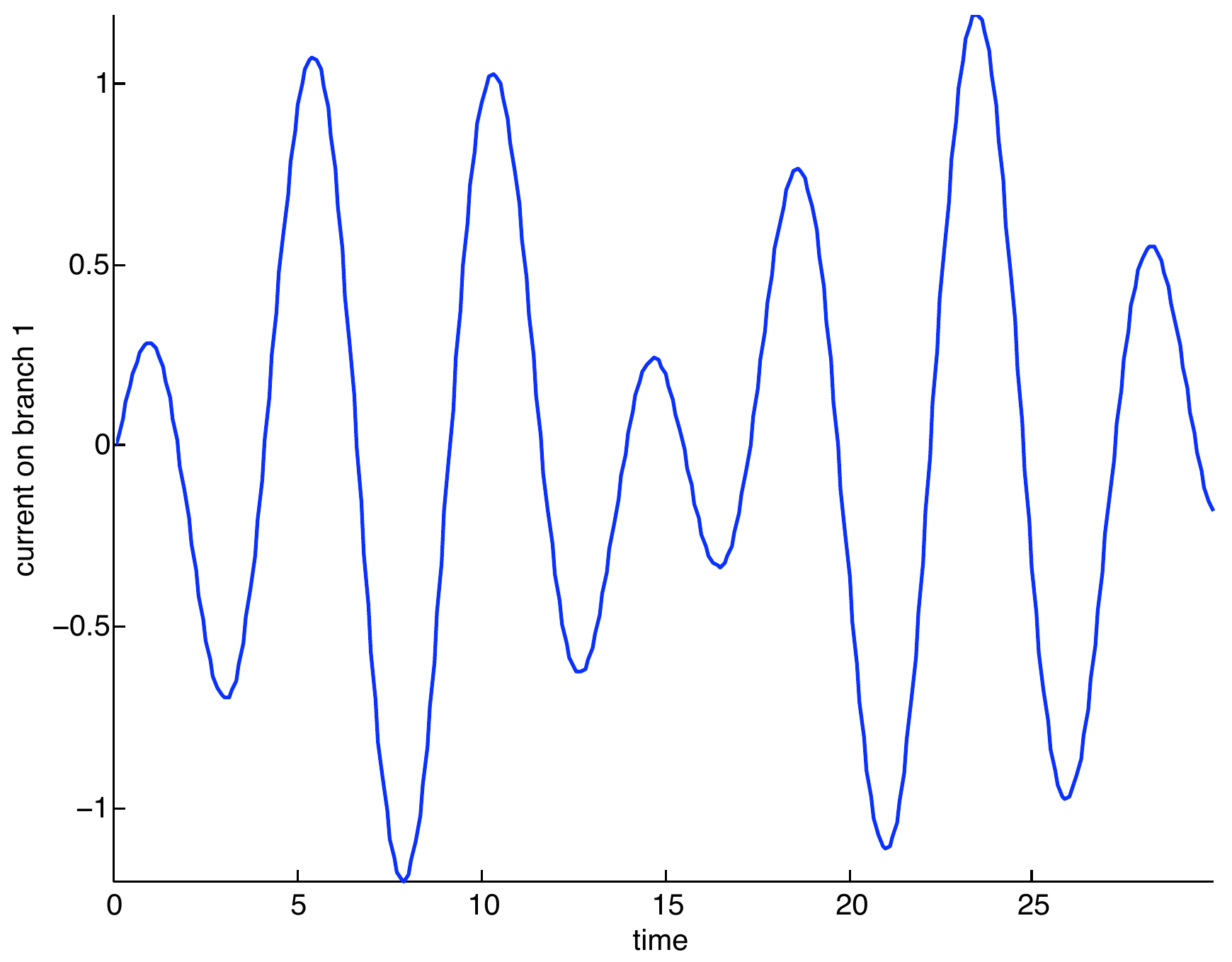} &\includegraphics[width=0.45\textwidth]{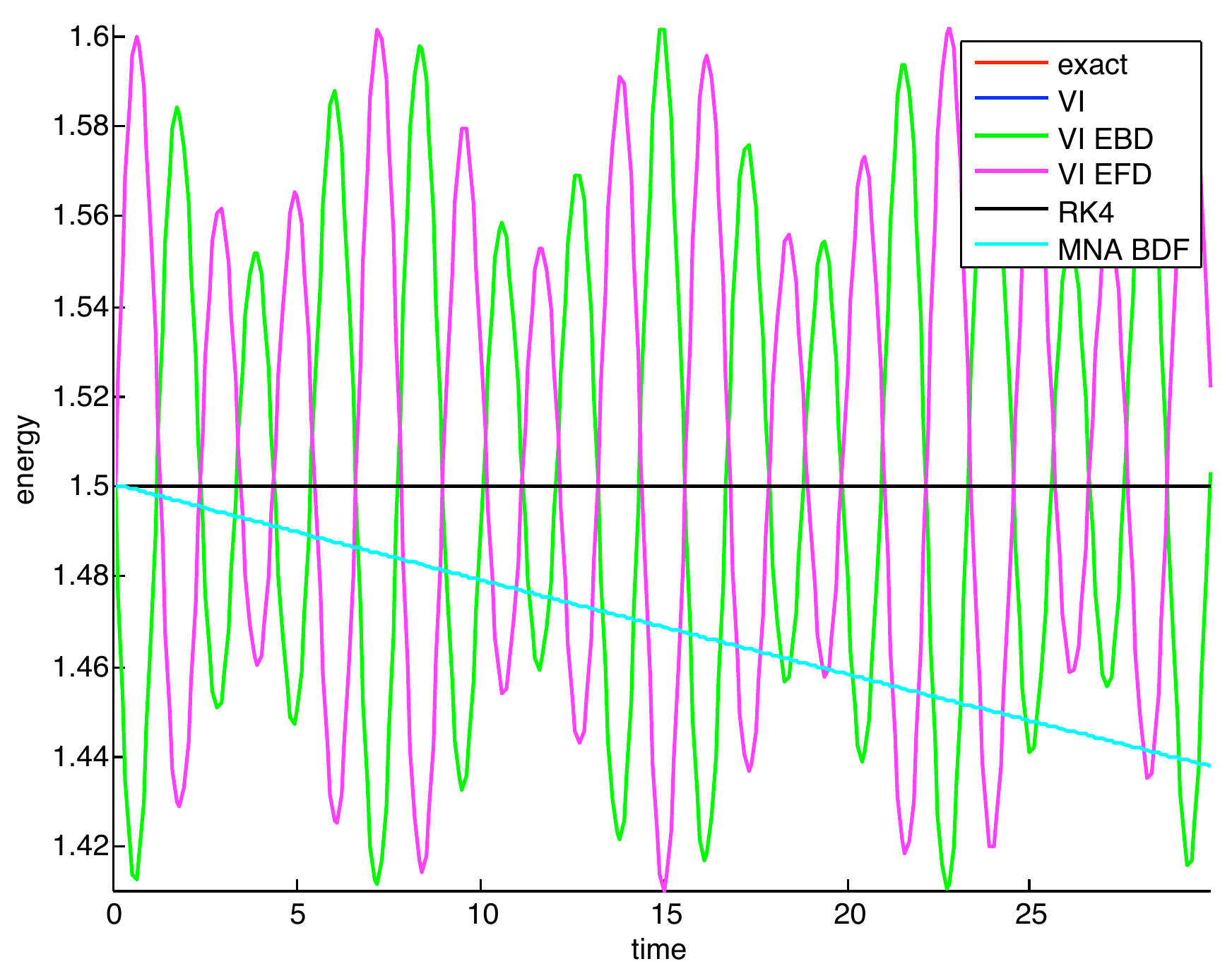} \\
\footnotesize{a)} & \footnotesize{b)} 
\end{tabular}
 \caption{LC circuit (no resistors) with step size $h = 0.1$. a) The oscillating behavior of the current on the first branch is shown. b) Comparison of the exact energy behavior (exact) and the numerical solution using the three different variational integrators, midpoint rule (VI), backward Euler (VI EBD), and forward Euler (VI EFD), a Runge-Kutta method of fourth order (RK), and a BDF method of second order based on MNA (MNA BDF). The energy is (qualitatively) preserved for VI, VI EBD, VI EFD, and RK. The use of BDF leads to an artificial energy decay.}
 \label{fig:graph_currents}
\end{figure}
In Figure \ref{fig:graph_currents} a) the oscillating behavior of the current on the first branch is shown (the currents on the other branches behave in a similar way). For the energy behavior of the LC circuit, we compare the exact solution with solutions obtained with the three different variational integrators, a Runge-Kutta method of fourth order, and a BDF method of second order. Since no resistor or voltage source is involved, this energy should be preserved. For the variational integrator based on the midpoint rule (VI), the energy is exactly preserved since the electric potential is only quadratic. For this relatively short integration time span, we observe that also the solution with the Runge-Kutta scheme (RK4) preserves the energy (the red, blue and black lines in Figure~\ref{fig:graph_currents} b) lie on top of each other). Using the forward (VI EFD, magenta) or backward Euler (VI EBD, green) variational integrator the energy oscillates around its real value, however, no dissipation or artificial growth of the energy occurs in contrast to the solution obtained by a BDF method (MNA BDF): Here, the energy rapidly decreases using a second order BDF method (see the cyan line in Figure~\ref{fig:graph_currents} b)). These results are based on a step size of $h=0.1$. Increasing the step size to $h=0.4$, a phase shifting of the currents computed with variational integrators is observed (which is a typical behavior observed for variational integrators) in contrast to solutions obtained by the Runge-Kutta scheme (see Figure \ref{fig:graph_currents2} a)).\footnote{Note that the Runge-Kutta method is of higher order (fourth order) than the variational integrators, and thus solution curves are of higher accuracy.} However, considering the energy behavior shown in Figure \ref{fig:graph_currents2} b), energy is dissipating for the Runge-Kutta solution, whereas for the variational integrators it (its median, respectively) is preserved. The performance of the BDF method is even worse: The energy is dissipating very fast. 
This can also be observed considering the current in Figure \ref{fig:graph_currents2} a). The amplitude of the current oscillations is damped to almost zero after a certain integration time.
\begin{figure}[htb]
 \centering
\begin{tabular}{cc}
\includegraphics[width=0.45\textwidth]{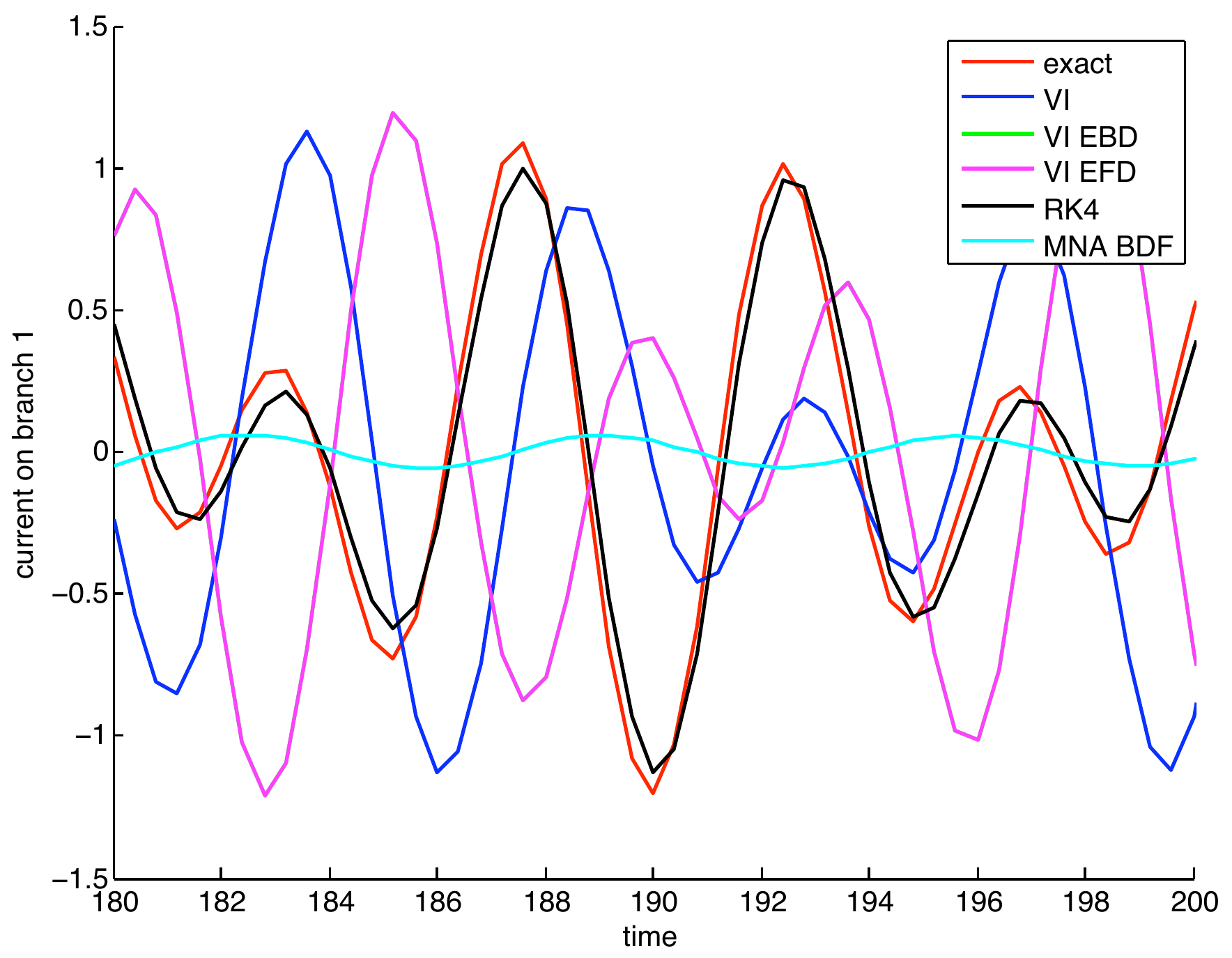} &\includegraphics[width=0.45\textwidth]{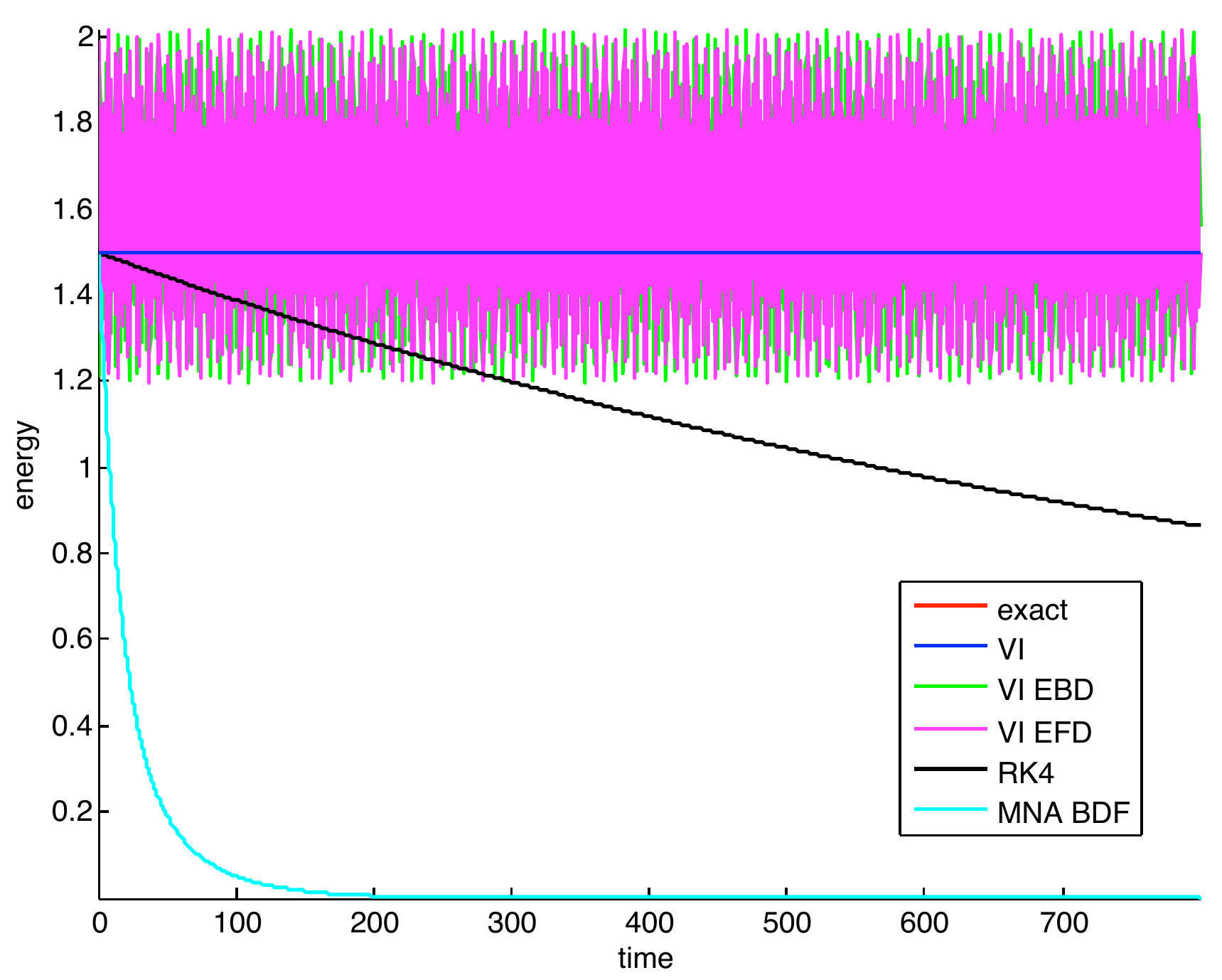} \\
\footnotesize{a)} & \footnotesize{b)} 
\end{tabular}
 \caption{LC circuit (no resistors) with step size $h = 0.4$. Comparison of the exact solution (exact) and the numerical solution using the three different variational integrators, midpoint rule (VI), backward Euler (VI EBD), and forward Euler (VI EFD), a Runge-Kutta method of fourth order (RK4), and a BDF method of second order based on MNA (MNA BDF). a) The use of variational integrators (VI, VI EBD, VI EFD) leads to a phase shifting in the numerical solution of the current. With the BDF method, the oscillations are damped out.  b) The energy is (qualitatively) preserved for VI, VI EBD, and VI EFD. The use of RK4 and BDF leads to an artificial energy decay.}
  \label{fig:graph_currents2}
\end{figure}

We investigate this phenomenon even more by comparing the amplitude of the oscillating branch current and the corresponding spectrum in frequency domain for the different integration methods. The solution of the first branch current oscillates with two frequencies, $\omega_1 = 1$ and $\omega_2 = \sqrt{2}$. 
In Figure~\ref{fig:graph_ampl}, the branch current and the frequency spectrum for the discrete solution computed with two different step sizes ($h=0.1, 0.4$) and different integrators are compared. 
For a bigger step size $h$, the amplitude of the oscillations and the spectrum of the higher frequency are artificially damped for the Runge-Kutta and the BDF method. However, using a variational integrator, the frequency is slightly shifted, but the spectrum is much better preserved. 
For $h=0.4$, we compute the frequency spectrum for three different time intervals, $\left[0,\frac{T}{3}\right]$, $\left[\frac{T}{3},\frac{2T}{3}\right]$, and $\left[\frac{2T}{3},T\right]$. Corresponding to our analytical result regarding frequency preservation (see Section~\ref{subsec:frequ}), we see in Figure~\ref{fig:graph_frequ_short} that the frequency spectrum is the same independent on the integration time using a variational integrator. However, using the Runge-Kutta or BDF method it is damped choosing a time interval after a longer integration time. 
 
In each branch of the circuit, we now add a small resistor with resistance $R_i = 0.001,\,i=1,\ldots,6$.
Again, we compare the oscillating behavior and the energy behavior of the numerical solution obtained by different integrators. Due to the resistors, the energy in the system decays. The rate of energy decay is shown for the exact solution in Figure~\ref{fig:graphR_currents2} b). The variational integrator solution respects this energy decay much better than the Runge-Kutta and BDF scheme.

\begin{figure}[htb]
 \centering
\begin{tabular}{cc}
\includegraphics[width=0.45\textwidth]{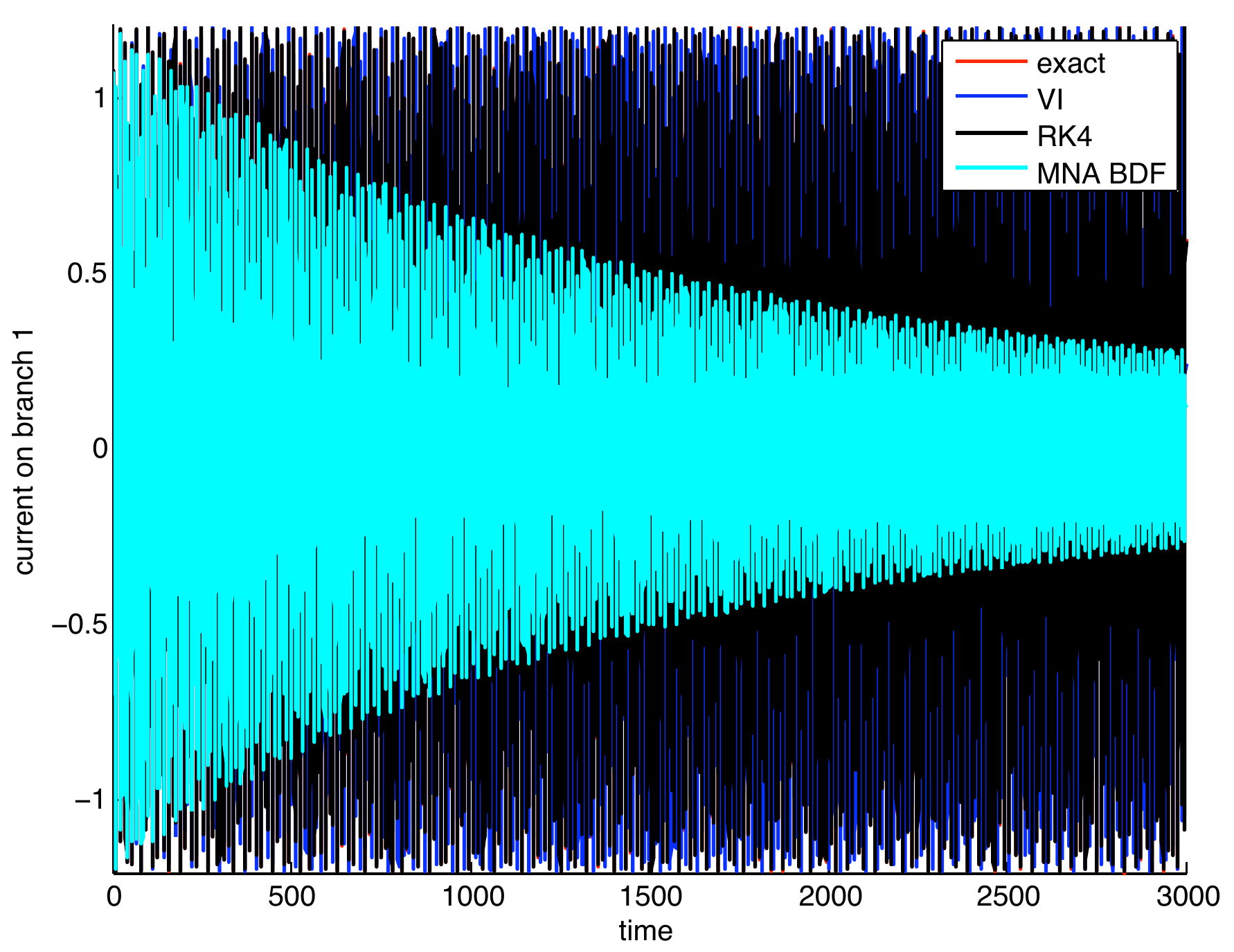} & \includegraphics[width=0.45\textwidth]{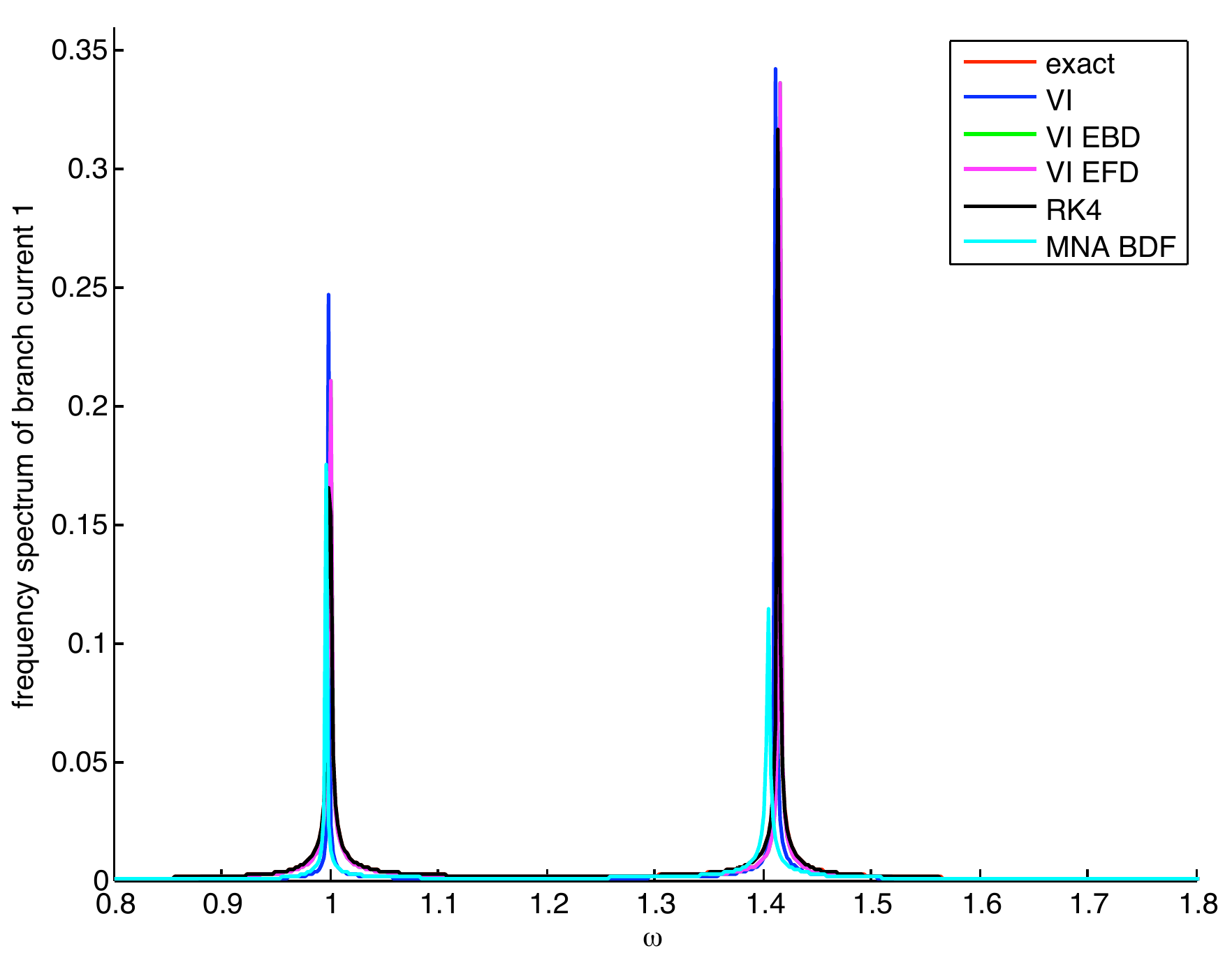} \\
  \footnotesize{$h=0.1$} & \footnotesize{$h=0.1$}\\
 \includegraphics[width=0.45\textwidth]{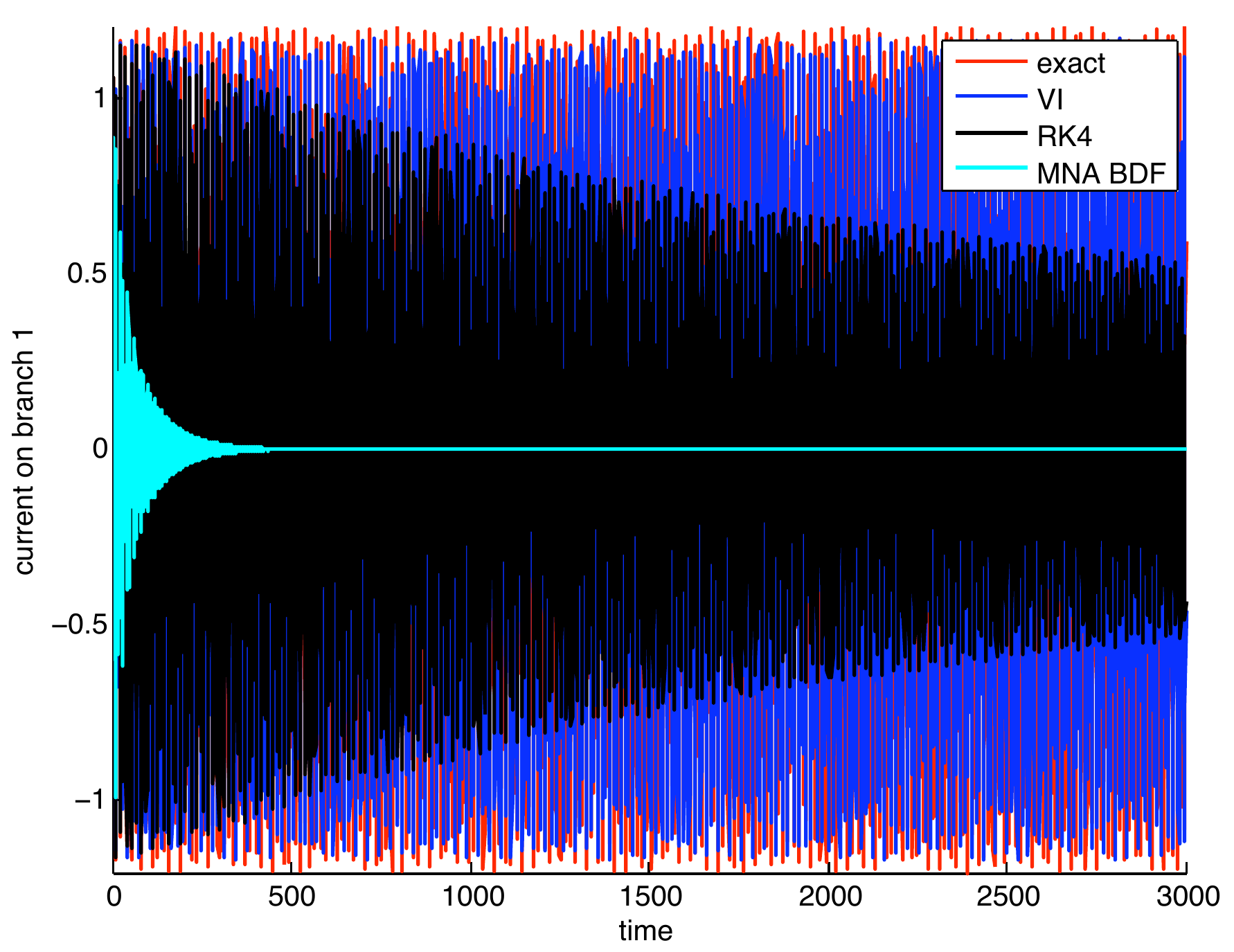} & \includegraphics[width=0.45\textwidth]{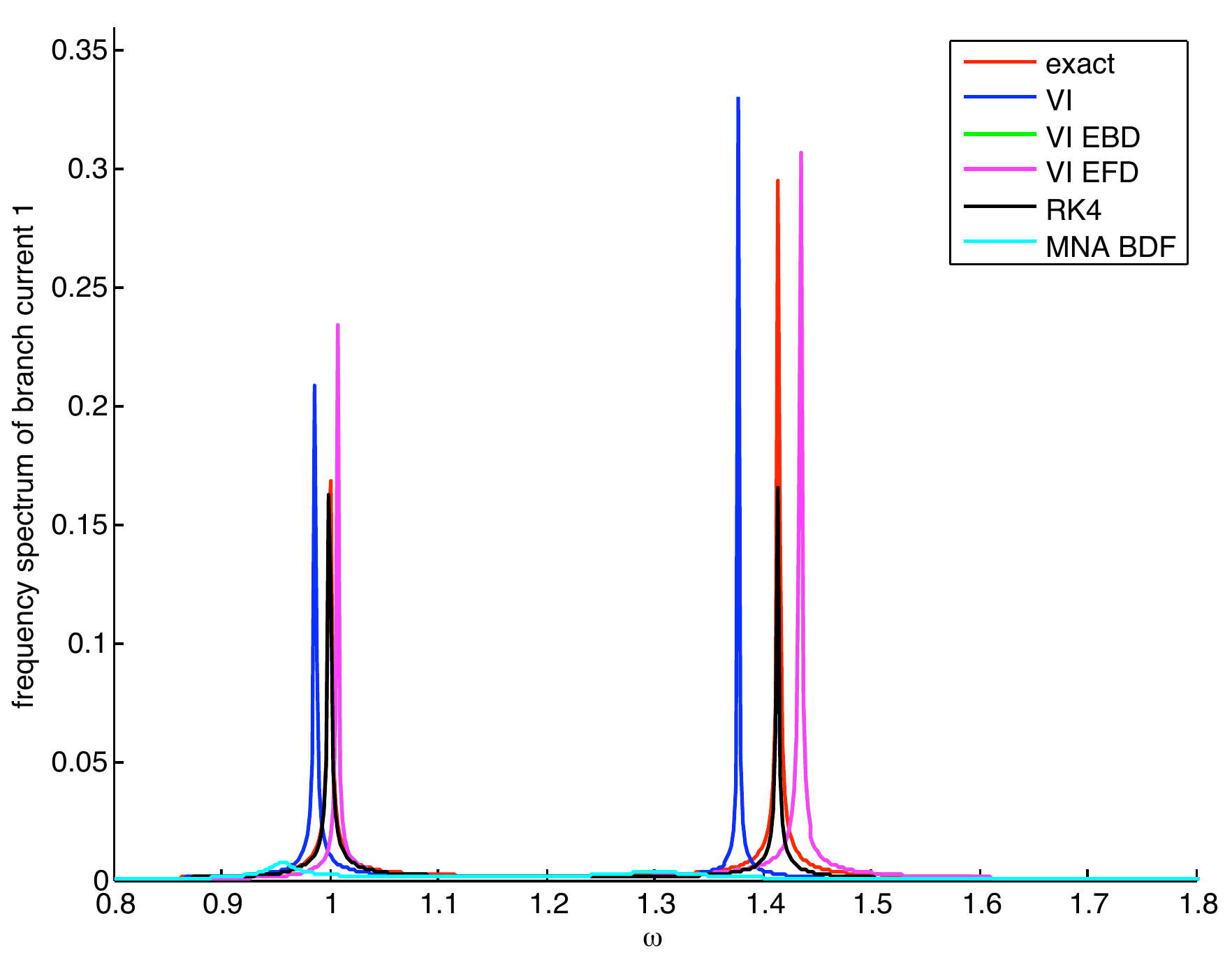} \\
\footnotesize{$h=0.4$} & \footnotesize{$h=0.4$} \\
\end{tabular}
\caption{Current of first branch of LC circuit (left) and corresponding frequency spectrum (right). Comparison of the exact solution (exact) and the numerical solution using the three different variational integrators, midpoint rule (VI), backward Euler (VI EBD), and forward Euler (VI EFD), a Runge-Kutta method of fourth order (RK4), and a BDF method of second order based on MNA (MNA BDF). For increasing step size $h$, the damping of the amplitude of the oscillations and of the higher frequency spectrum increases using RK4 and BDF. For VI, VI EBD, and VI EFD, the frequency is slightly shifted, but the spectrum and the current amplitude is much better preserved.}
\label{fig:graph_ampl}
\end{figure}
\begin{figure}[htb]
 \centering
\includegraphics[width=0.5\textwidth]{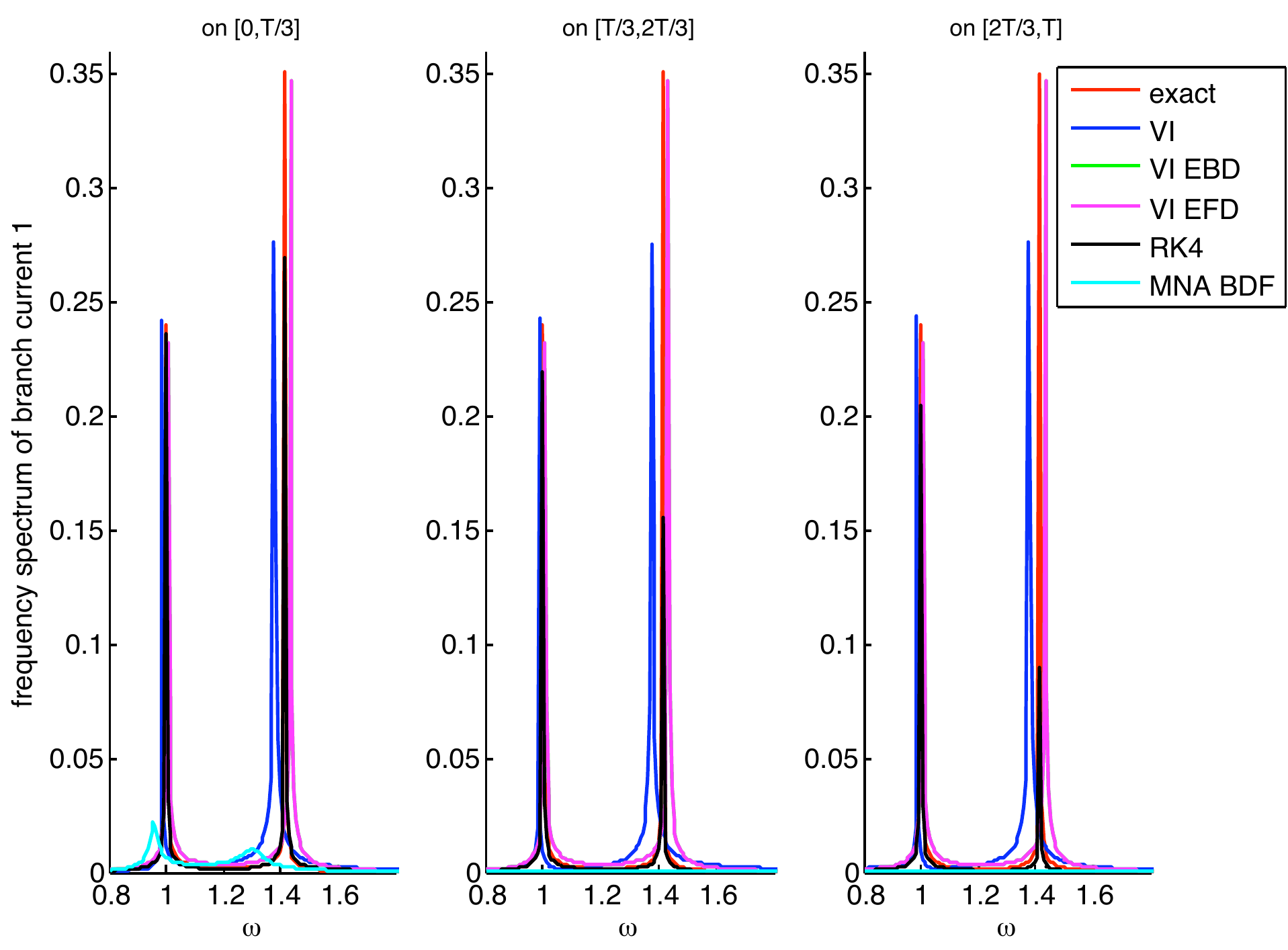}
\caption{Frequency spectrum of first branch current of LC circuit (no resistors) ($h=0.4$) computed on time interval $[0,T/3]$, $[T/3,2T/3]$, $[2T/3,T]$. Comparison of the exact solution (exact) and the numerical solution using the three different variational integrators, midpoint rule (VI), backward Euler (VI EBD), and forward Euler (VI EFD), a Runge-Kutta method of fourth order (RK4), and a BDF method of second order based on MNA (MNA BDF). Using RK4 and BDF the spectrum is damped for higher integration times and preserved using a variational integrator.}
\label{fig:graph_frequ_short}
\end{figure}

\begin{figure}[htb]
 \centering
\begin{tabular}{cc}
\includegraphics[width=0.45\textwidth]{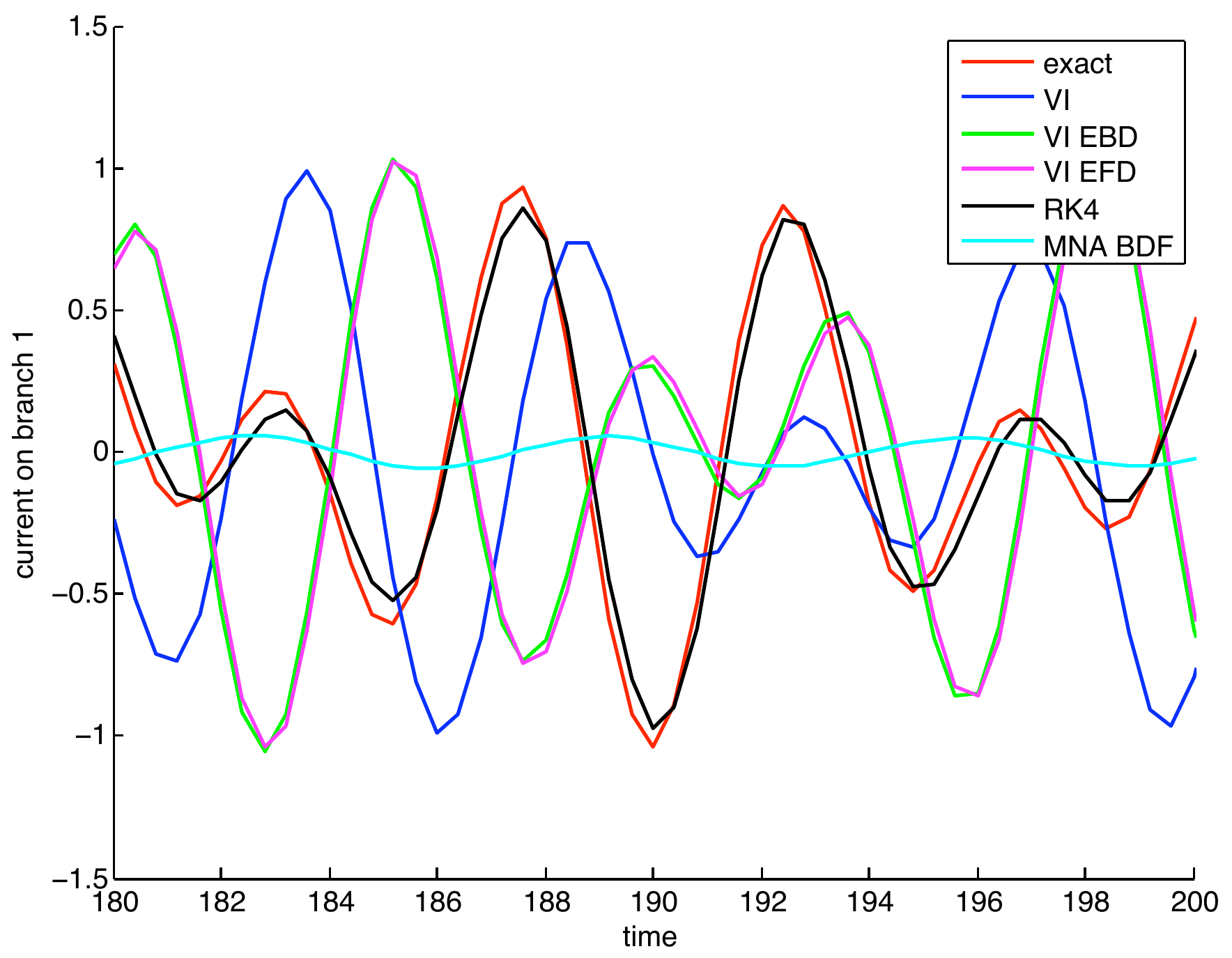} &\includegraphics[width=0.45\textwidth]{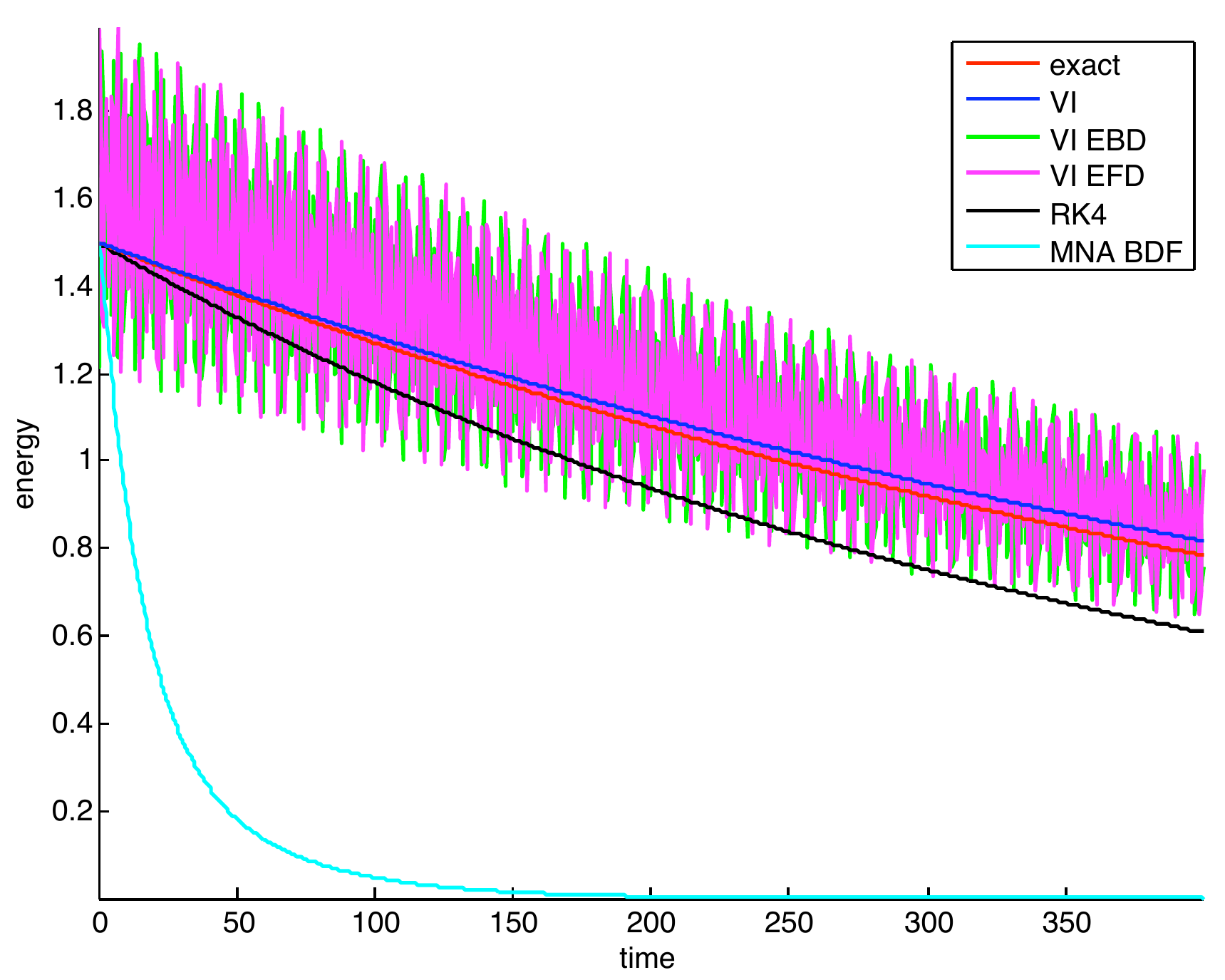} \\
\footnotesize{a)} & \footnotesize{b)} 
\end{tabular}
 \caption{LCR circuit (with resistors) with step size $h = 0.4$. Comparison of the exact solution (exact) and the numerical solution using the three different variational integrators, midpoint rule (VI), backward Euler (VI EBD), and forward Euler (VI EFD), a Runge-Kutta method of fourth order (RK4), and a BDF method of second order based on MNA (MNA BDF). a) The use of variational integrators (VI, VI EBD, VI EFD) leads to a phase shifting in the numerical solution of the current. With the BDF method, the oscillations are artificially damped out.  b) The energy decay is much better preserved for VI, VI EBD, and VI EFD as for RK4 and BDF.}
  \label{fig:graphR_currents2}
\end{figure}

\subsection{Oscillating LC circuit}\label{subsec:osLC}

As second example, we consider the LC circuit given in Figure \ref{fig:circuit1}. It consists of two inductors with inductance $L_1=1$ and $L_2=1$ and two capacitors with capacitance $C_1=1$ and $C_2=10$ and thus has $n=4$ branches and $m+1=3$ nodes. 
The Kirchhoff Constraint matrix $K\in \mathbb{R}^{n,m}$ and the Fundamental Loop matrix $K_2\in \mathbb{R}^{n,n-m}$ are (with the third node assumed to be grounded)

\begin{equation}
K = \left(\begin{array}{cc} 1 & 0\\ 0 & -1 \\ 0 & -1 \\ -1 &1  \end{array}\right), \quad\quad 
K_2 = \left(\begin{array}{cc} 1 & 0\\ 0 &1 \\ 1 & -1 \\ 1 & 0  \end{array}\right).
\end{equation}
\begin{wrapfigure}{r}{0.38\textwidth}
\begin{center}
\includegraphics[width = 0.38\textwidth]{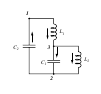}
\end{center}
\caption{Oscillating LC circuit.}
\label{fig:circuit1}
\end{wrapfigure}

With $n_C=2$ and $K_C =  \left(\begin{array}{cc} 0 & -1 \\ -1 &1  \end{array}\right)$ having full rank, we can follow from Proposition~\ref{prop:degeneracy}, that the reduced Lagrangian system is non-degenerate, i.e.~all three variational integrators derived in Section~\ref{sec:disvar} can be applied.

In Figure \ref{fig:circ1_currents}, the oscillating behavior of the branch currents on the inductors (a)--b)) and of the branch charges on the capacitors (c)--d)) is depicted. 
For a long-time simulation, we compare the exact energy behavior of the LC circuit with the energy behavior of the solution obtained with the three different variational integrators (VI, VI EBD, VI EFD), a Runge-Kutta method of fourth order (RK4), and a BDF method of second order based on MNA (BDF MNA). As for the previous example, independent on the step size $h$, the energy is exactly preserved using VI based on midpoint rule, whereas the solutions using the Euler VI oscillate around the real energy value without dissipation or artificial growth of the energy (see Figure \ref{fig:circ1_energy}). Using the BDF method, the energy rapidly decreases for increasing time step $h$. The solution of the Runge-Kutta scheme shows a similar decreasing energy behavior for the step size $h=0.1, 0.2, 0.4$. However, for a step size of $h=0.6$, the energy seems to converge to constant value that is not the exact energy value, but slightly lower.
This might be due to the fact that the amplitude of the first lower frequency is almost preserved and only the amplitude corresponding to the higher frequency is damped out for increasing integration time, as shown in Figure \ref{fig:circ1_ampl} for different time spans and a step size of $h=0.6$. The same property is reflected by the plots in the frequency domain in Figure~\ref{fig:circ1_frequ_short} for $h=0.4$, where for increasing integration time the spectrum of the high frequency ($\omega_2 \approx 1.43$) is damped for the Runge-Kutta and BDF method and preserved using a symplectic method. This phenomenon is also confirmed by the frequency spectrum plot in Figure~\ref{fig:circ1_frequ} for different step sizes. As for the previous example, the spectrum corresponding to the higher frequency is damped out for higher step sizes $h$ for the BDF and the Runge-Kutta method.

\begin{figure}[htb]
 \centering
\begin{tabular}{cc}
\includegraphics[width=0.45\textwidth]{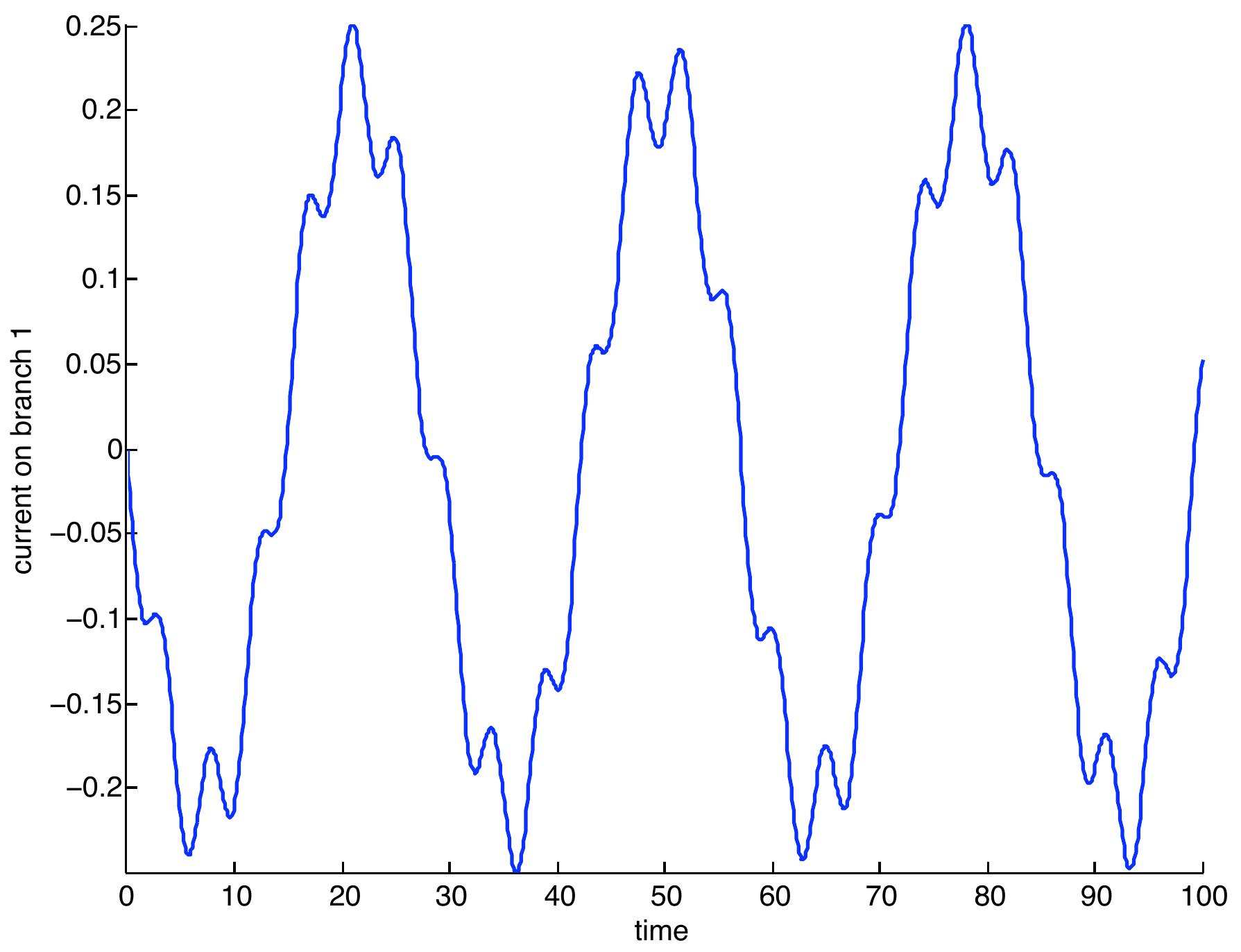} & \includegraphics[width=0.45\textwidth]{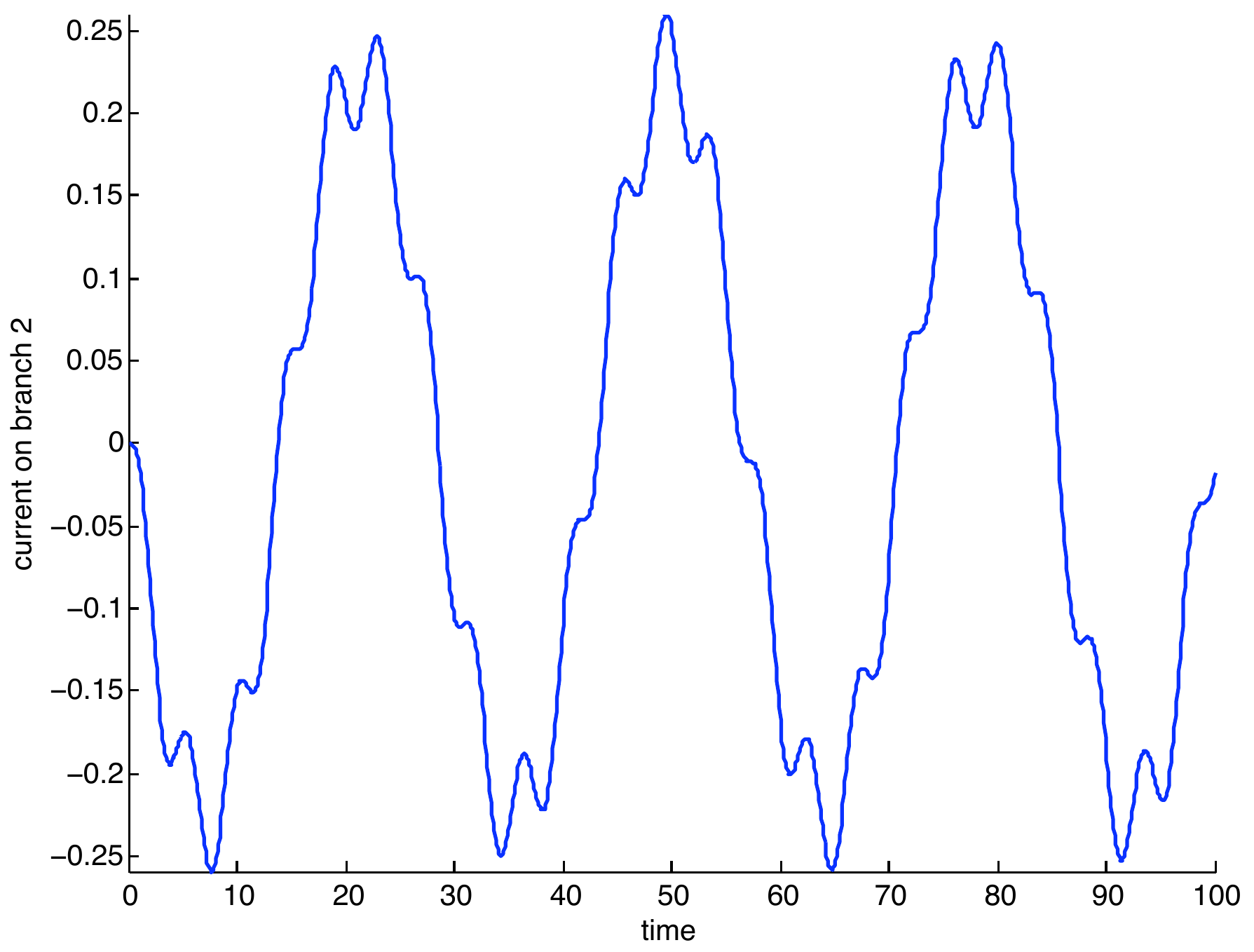} \\
  \footnotesize{a)} & \footnotesize{b)}\\
 \includegraphics[width=0.45\textwidth]{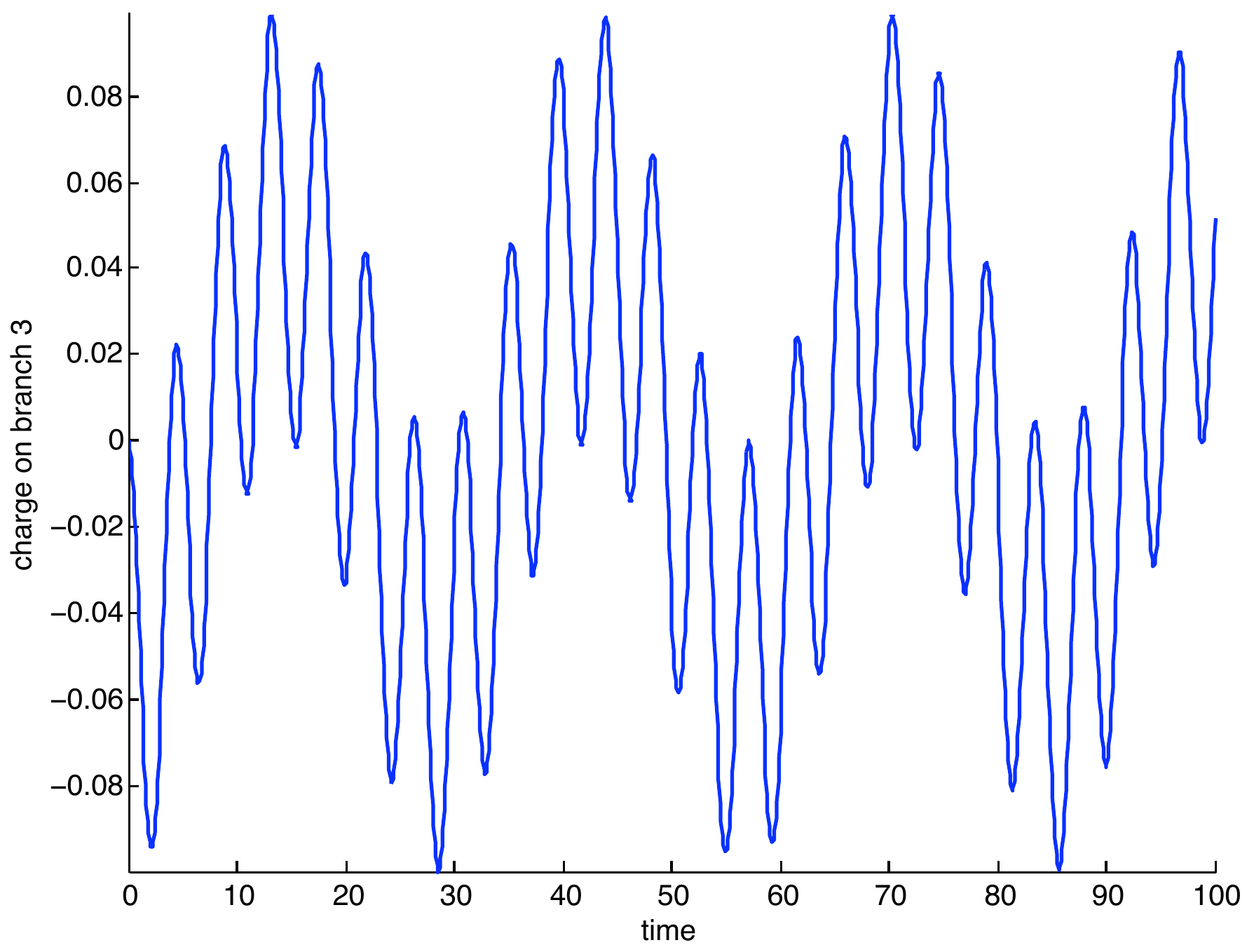} & \includegraphics[width=0.45\textwidth]{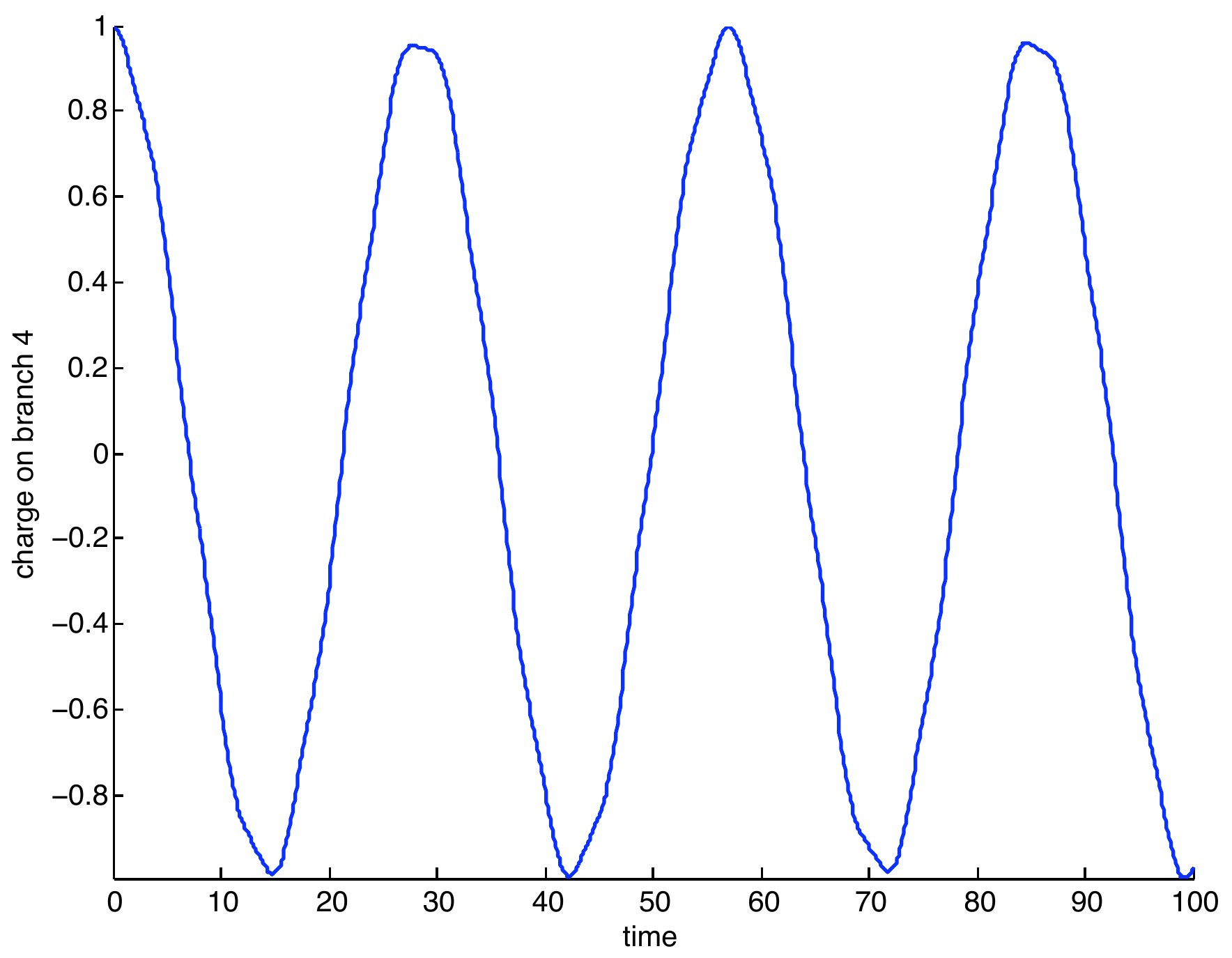} \\
\footnotesize{c)} & \footnotesize{d)} \\
\end{tabular}
\caption{LC circuit. a) Current on inductor $1$. b) Current on inductor $2$. c) Charge on capacitor $1$. d) Charge on capacitor $2$.}
\label{fig:circ1_currents}
\end{figure}

\begin{figure}[htb]
 \centering
\begin{tabular}{cc}
\includegraphics[width=0.45\textwidth]{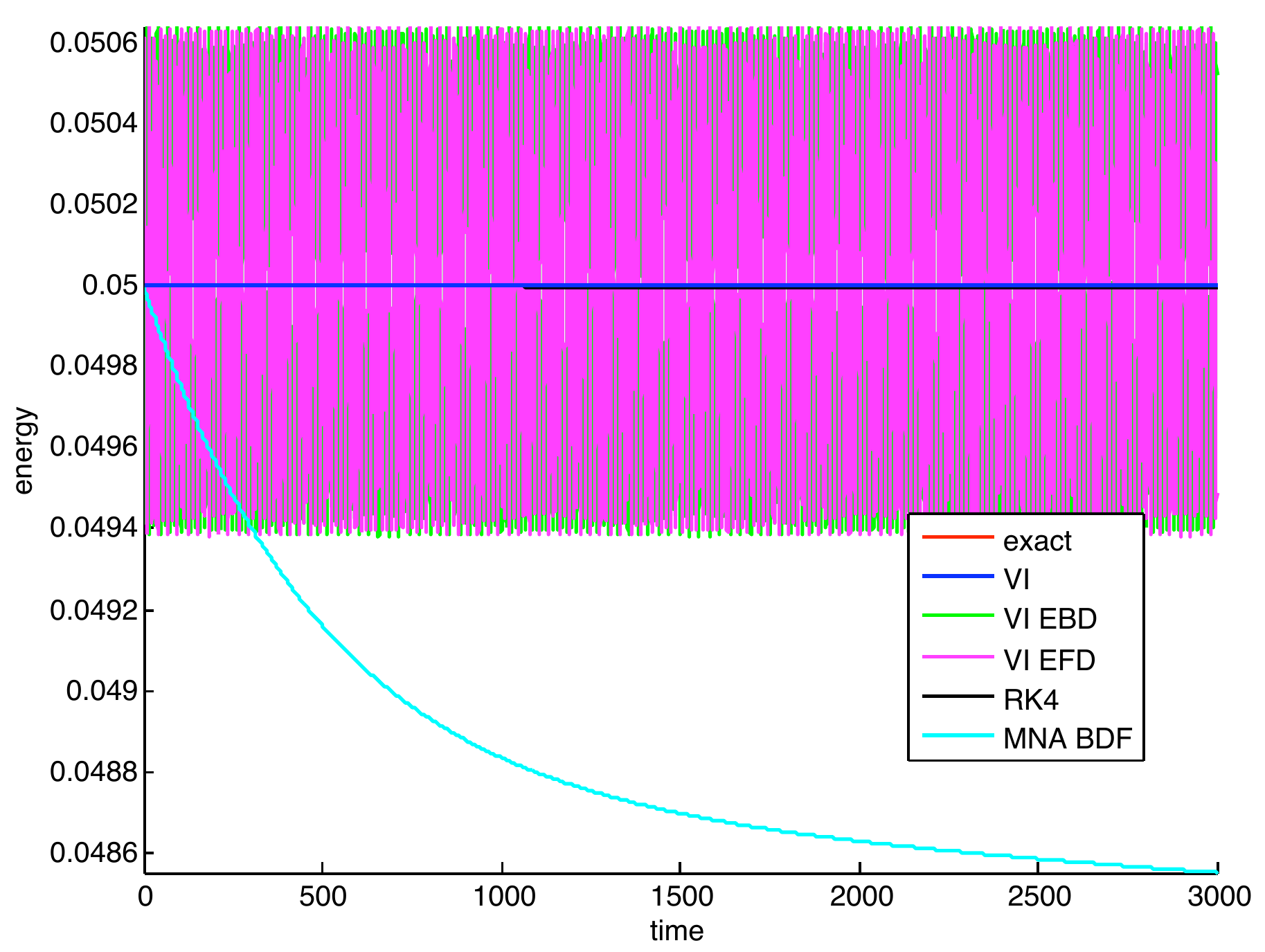} & \includegraphics[width=0.45\textwidth]{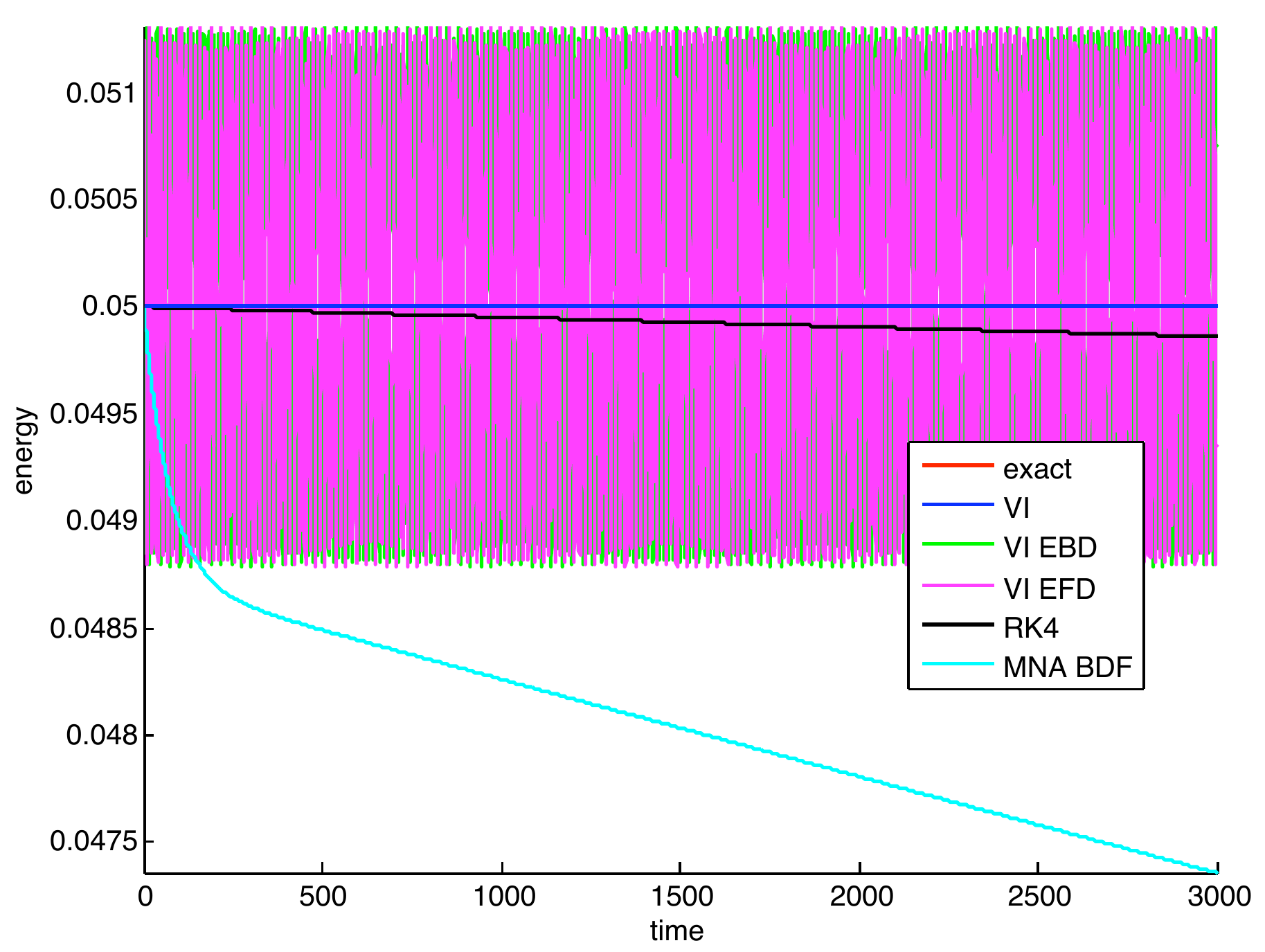} \\
  \footnotesize{$h=0.1$} & \footnotesize{$h=0.2$}\\
 \includegraphics[width=0.45\textwidth]{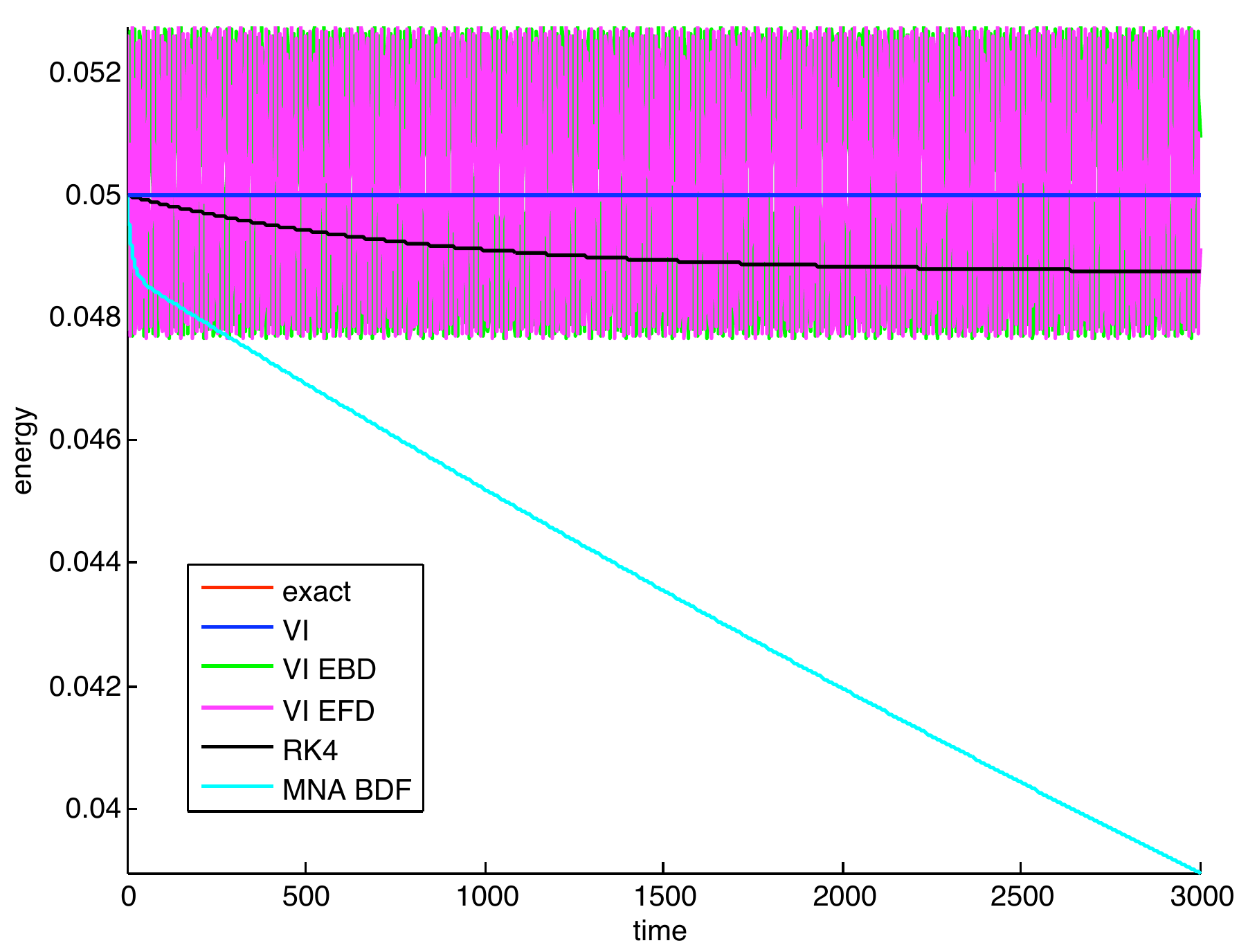} & \includegraphics[width=0.45\textwidth]{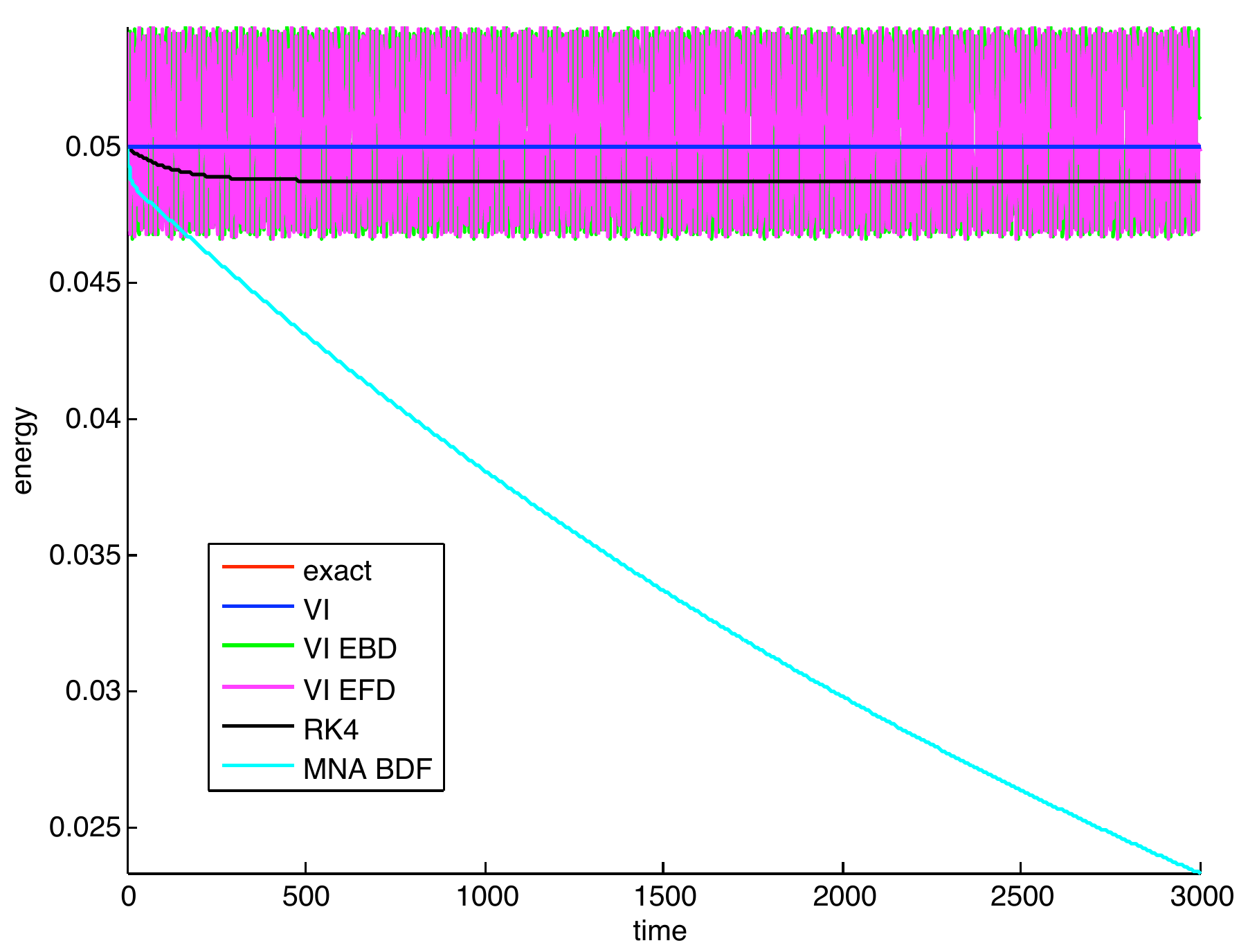} \\
\footnotesize{$h=0.4$} & \footnotesize{$h=0.6$} \\
\end{tabular}
\caption{Energy of oscillating LC circuit: Solutions obtained with the midpoint variational integrator exactly preserve the energy (VI) independent on the step size $h$. The energy computed with the forward (VI EFD) and backward (VI EBD) Euler variational integrators oscillates around the real energy value without dissipation or artificial growth. For solutions obtained with a second order BDF method (BDF MNA), the energy dissipates due to numerical errors. The solution of the Runge-Kutta scheme (RK4) dissipate energy, but seems to converge to lower constant energy value.}
\label{fig:circ1_energy}
\end{figure}

\begin{figure}[htb]
 \centering
\begin{tabular}{cc}
\includegraphics[width=0.45\textwidth]{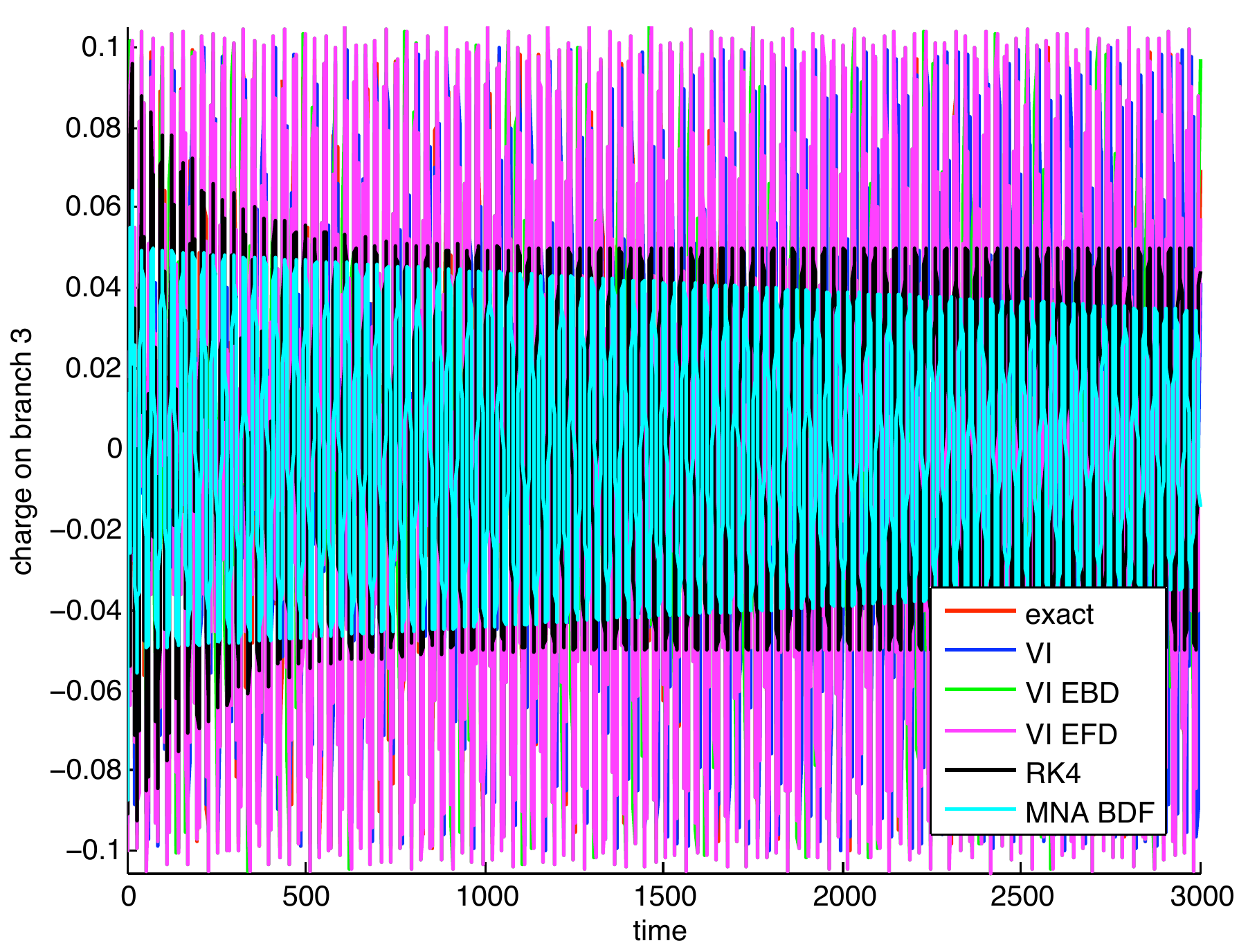} & \includegraphics[width=0.45\textwidth]{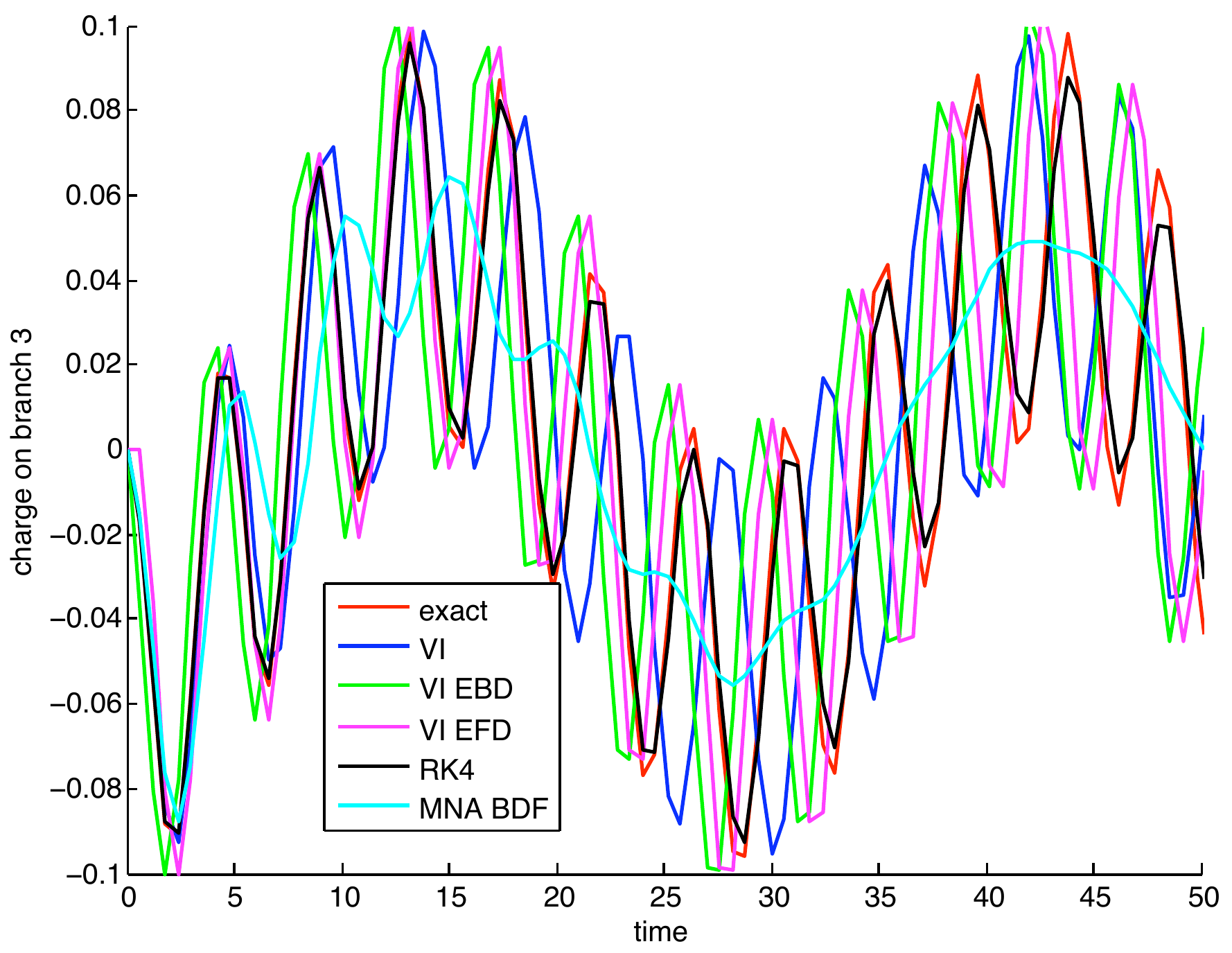} \\
 \includegraphics[width=0.45\textwidth]{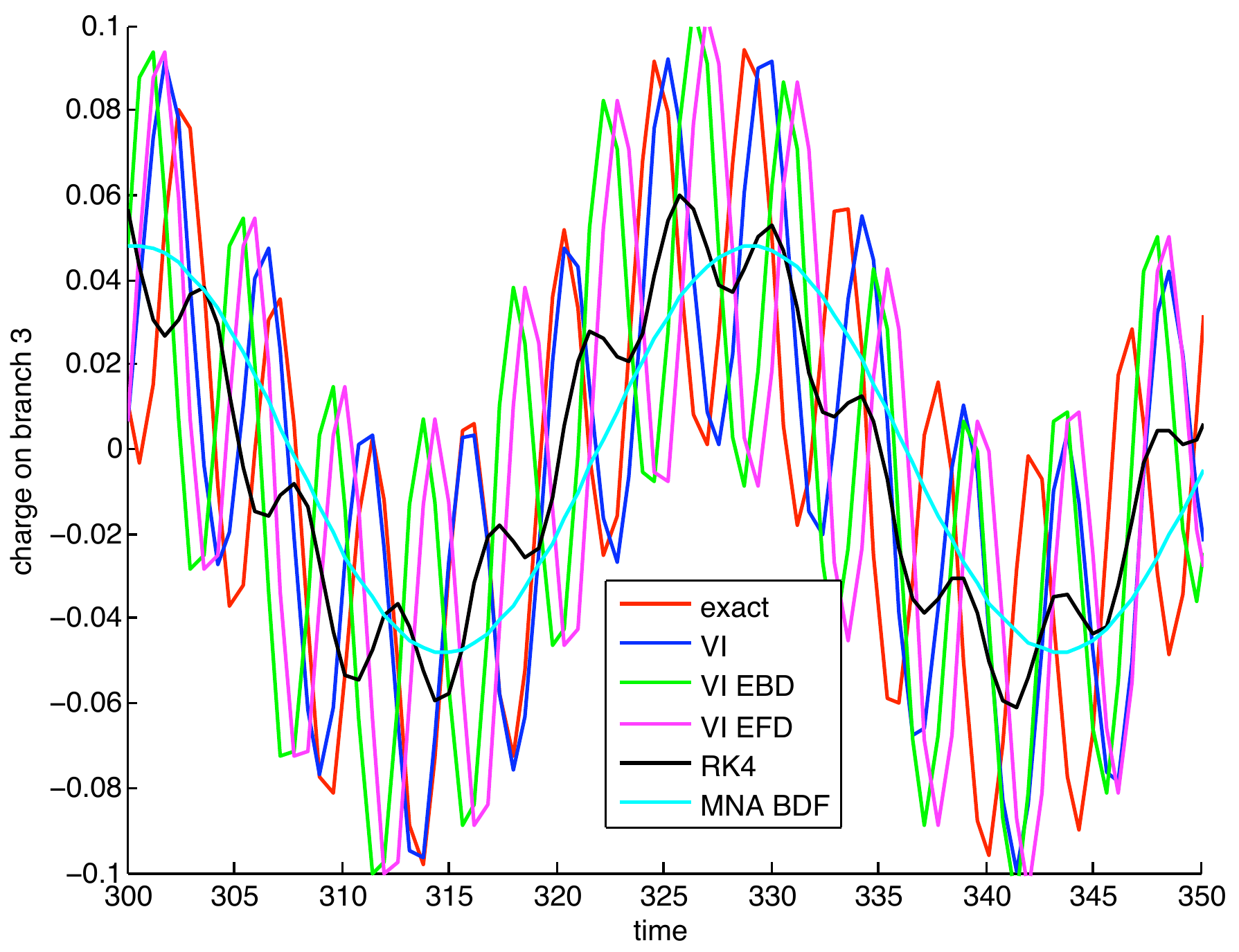} & \includegraphics[width=0.45\textwidth]{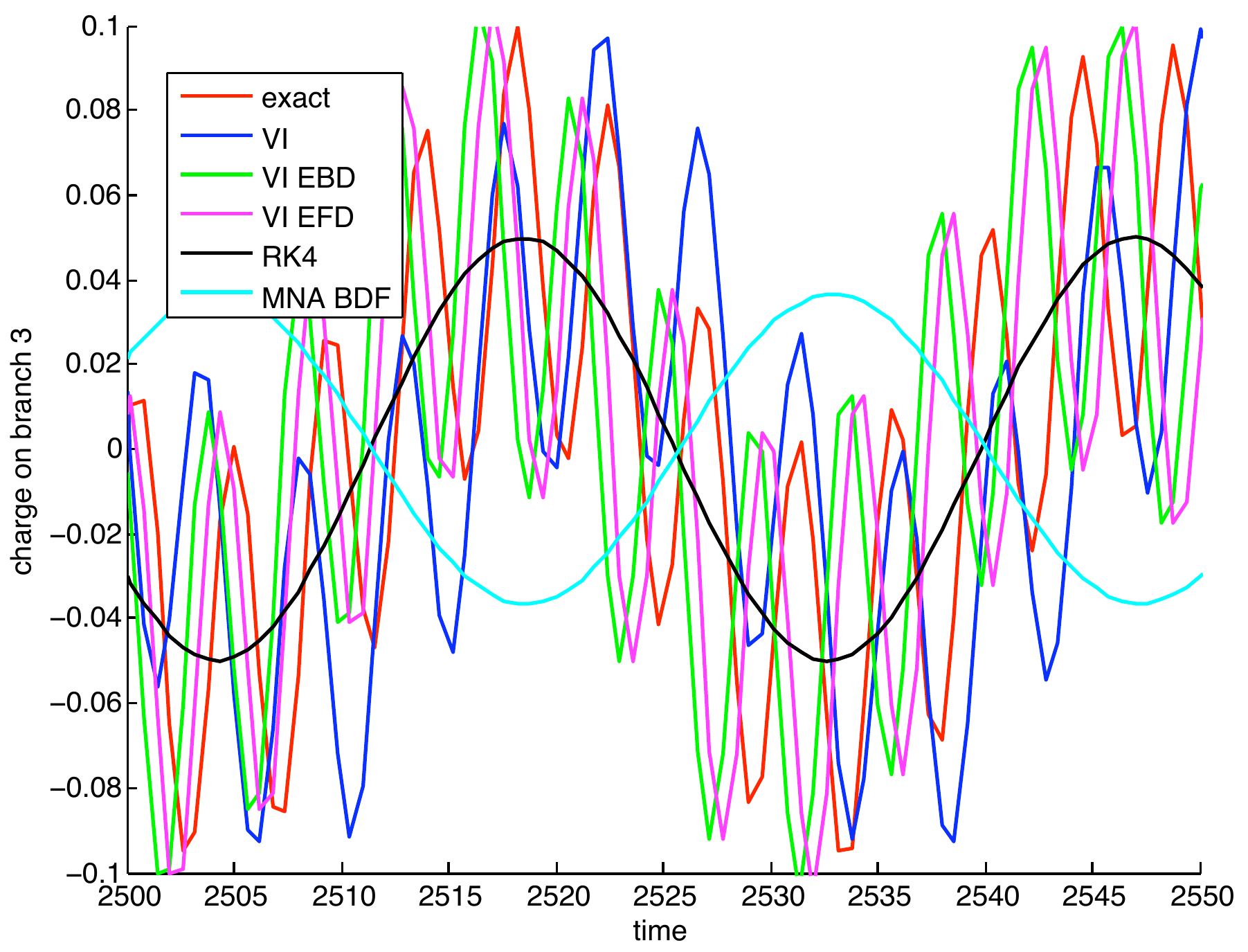} \\
\end{tabular}
\caption{Charge on capacitor $1$ of LC circuit ($h=0.6$). Comparison of the exact solution (exact) and the numerical solution using the three different variational integrators, midpoint rule (VI), backward Euler (VI EBD), and forward Euler (VI EFD), a Runge-Kutta method of fourth order (RK4), and a BDF method of second order based on MNA (MNA BDF). For the variational integrators the amplitudes for low and high frequencies are preserved. 
Using the Runge-Kutta method, the amplitude corresponding to the high frequency is damped out for increasing integration time whereas the amplitude of the lower one seems to be preserved. Using the BDF method, the amplitudes of both frequencies are damped.}
\label{fig:circ1_ampl}
\end{figure}

\begin{figure}[htb]
 \centering
\includegraphics[width=0.5\textwidth]{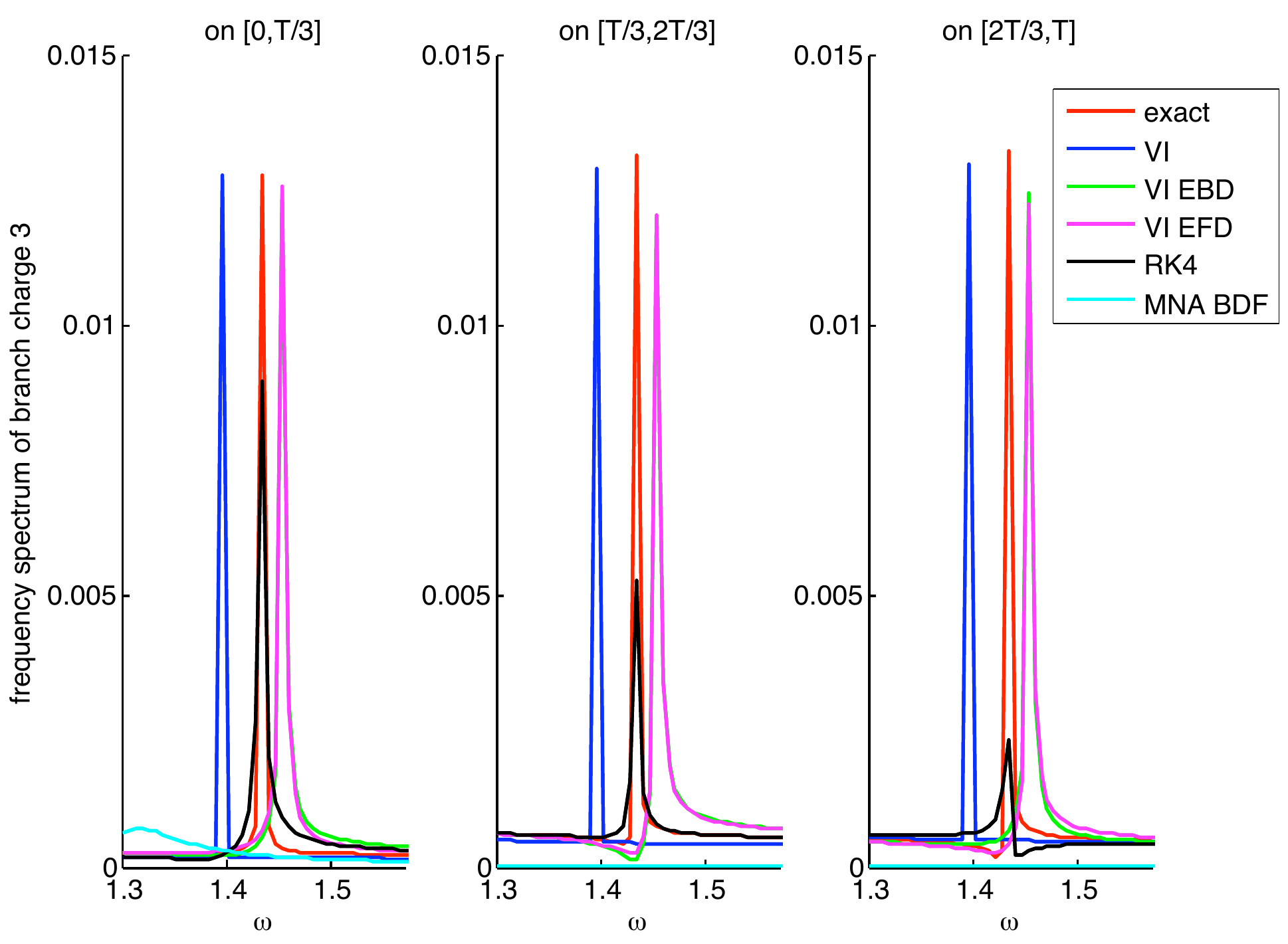} 
\caption{Frequency spectrum of charge on capacitor $1$ of LC circuit ($h=0.4$) computed on time interval $[0,T/3]$, $[T/3,2T/3]$, $[2T/3,T]$. Comparison of the exact solution (exact) and the numerical solution using the three different variational integrators, midpoint rule (VI), backward Euler (VI EBD), and forward Euler (VI EFD), a Runge-Kutta method of fourth order (RK4), and a BDF method of second order based on MNA (MNA BDF). Using RK4 and BDF the spectrum is damped for higher integration times and preserved using a variational integrator.}
\label{fig:circ1_frequ_short}
\end{figure}

\begin{figure}[htb]
 \centering
\begin{tabular}{cc}
\includegraphics[width=0.45\textwidth]{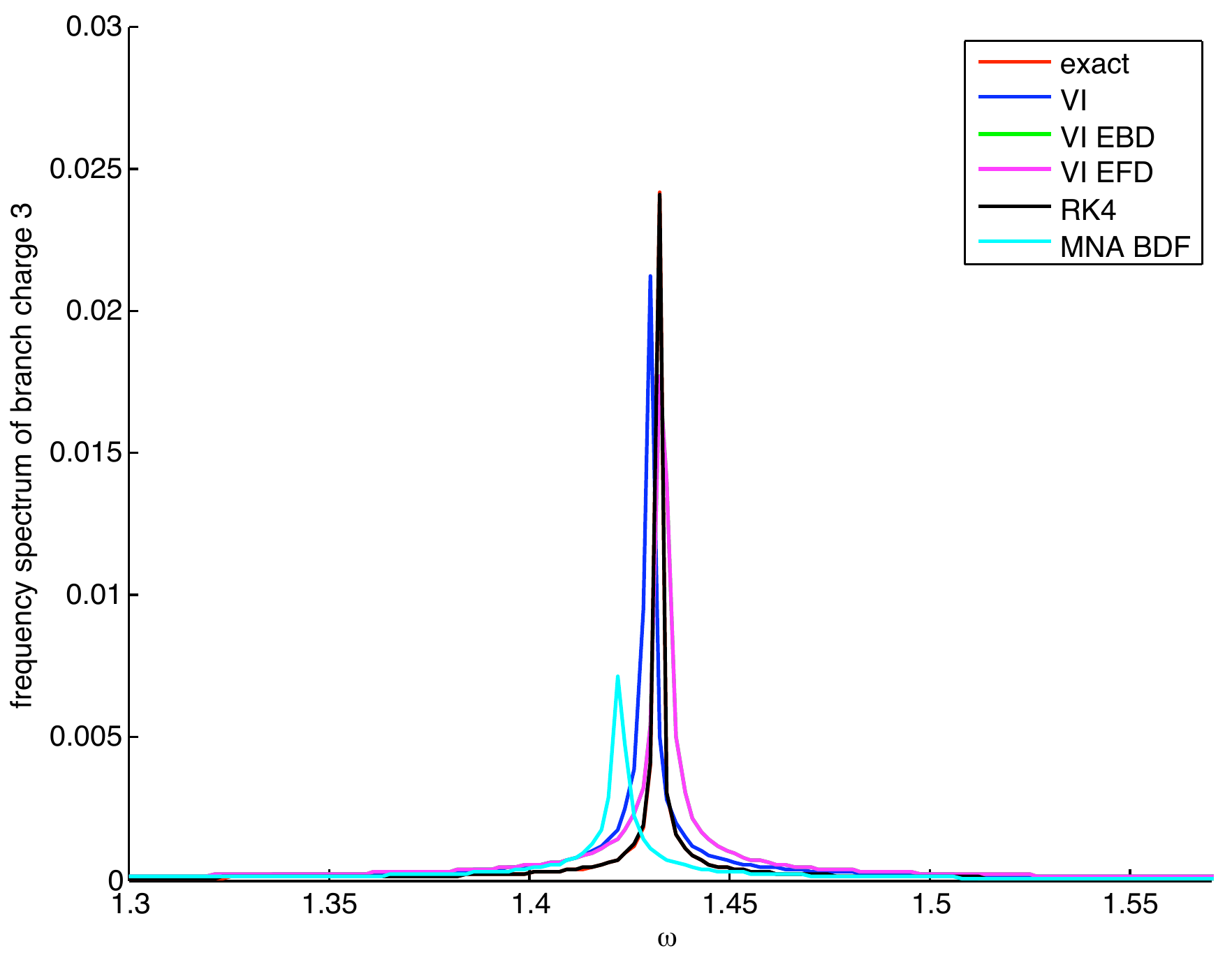} & \includegraphics[width=0.45\textwidth]{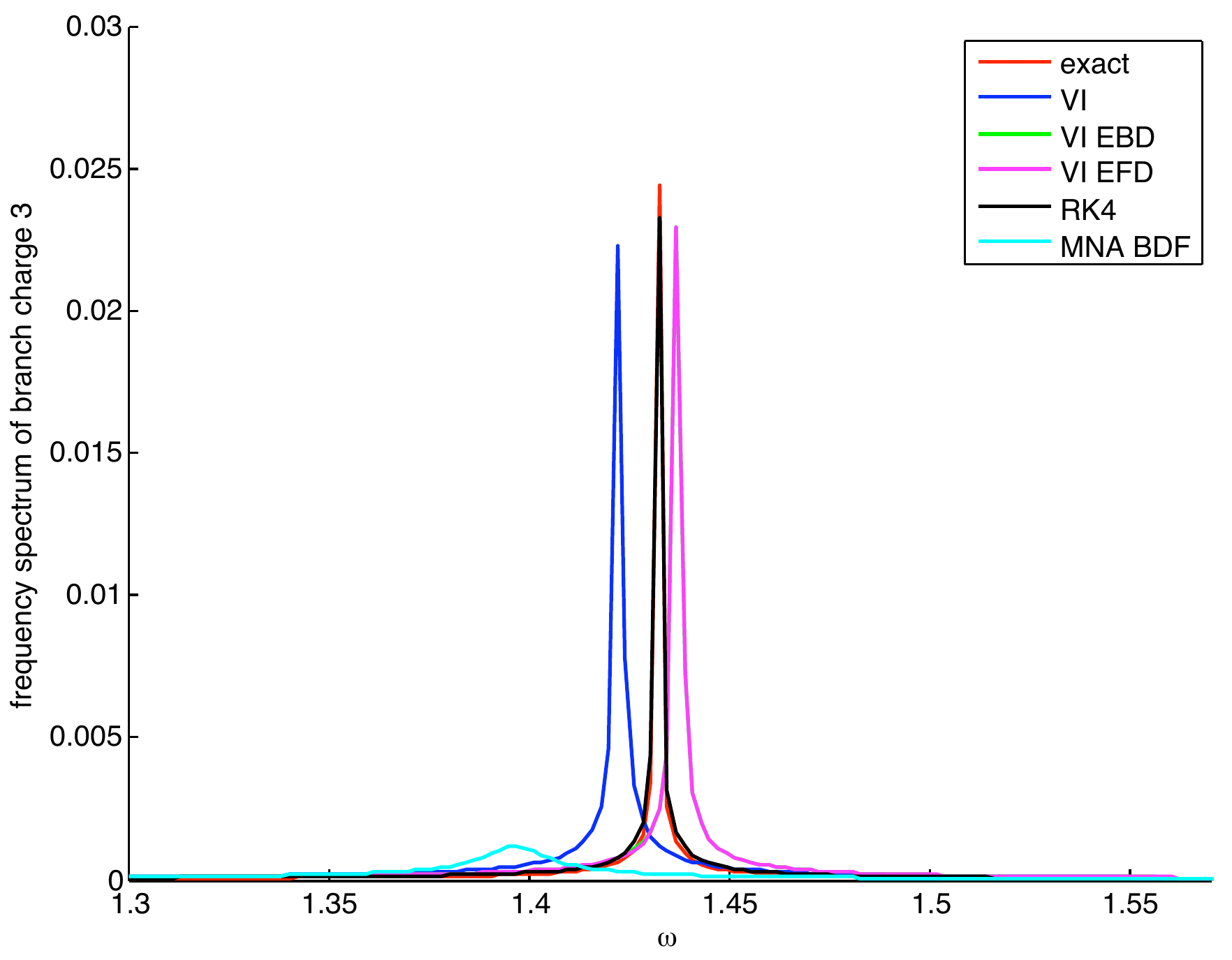} \\
  \footnotesize{$h=0.1$} & \footnotesize{$h=0.2$}\\
 \includegraphics[width=0.45\textwidth]{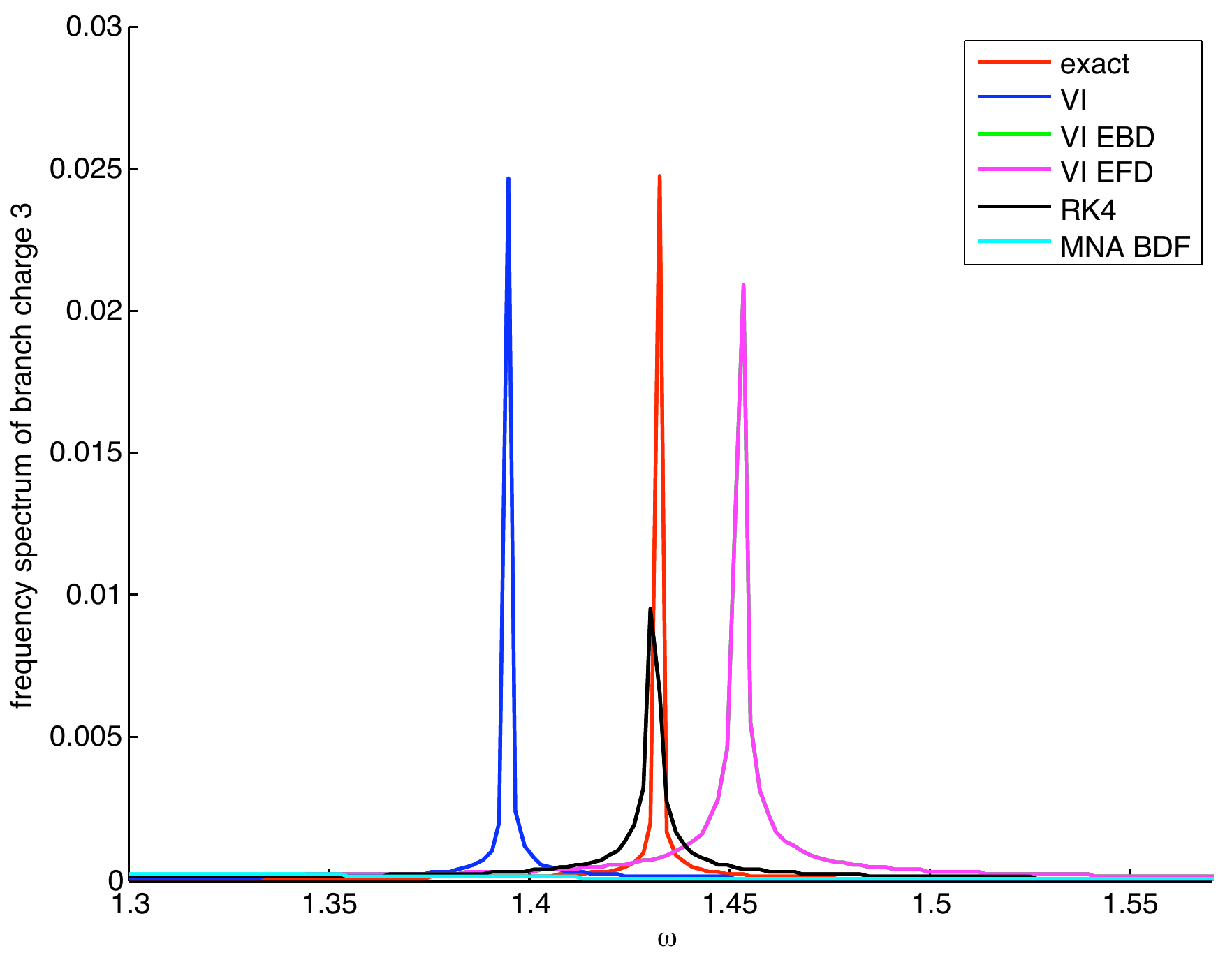} & \includegraphics[width=0.45\textwidth]{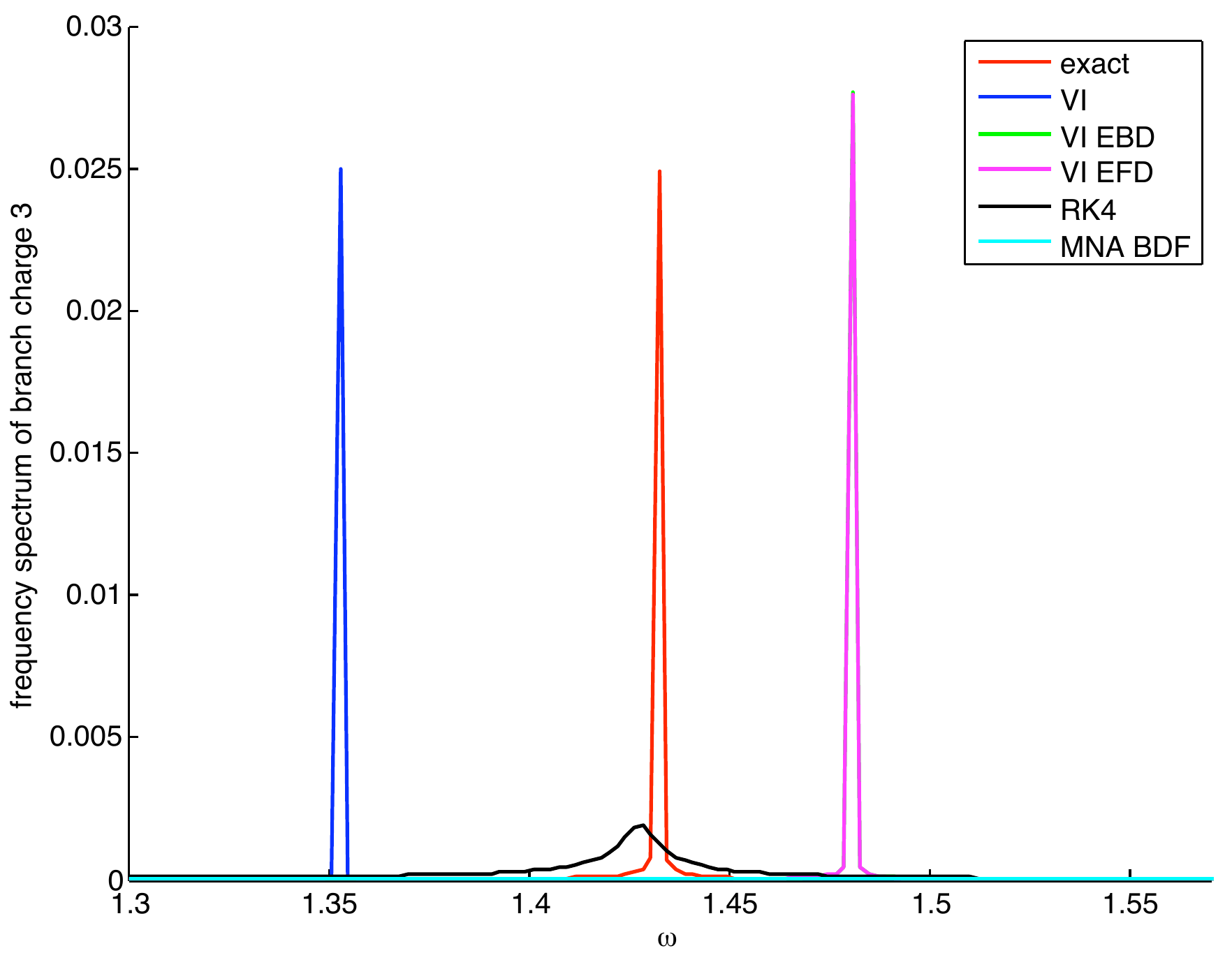} \\
\footnotesize{$h=0.4$} & \footnotesize{$h=0.6$} \\
\end{tabular}
\caption{Frequency spectrum of charge on capacitor $1$ of LC circuit. Comparison of the exact solution (exact) and the numerical solution using the three different variational integrators, midpoint rule (VI), backward Euler (VI EBD), and forward Euler (VI EFD), a Runge-Kutta method of fourth order (RK4), and a BDF method of second order based on MNA (MNA BDF). For increasing step size $h$, the damping of the higher frequency spectrum increases using RK4 and BDF. For VI, VI EBD, and VI EFD, the frequency is slightly shifted, but the spectrum is much better preserved.}
\label{fig:circ1_frequ}
\end{figure}

\subsection{LC transmission line}\label{subsec:tl}
\begin{wrapfigure}{r}{0.5\textwidth}
\vspace{-20pt}
\begin{center}
\includegraphics[width = 0.48\textwidth]{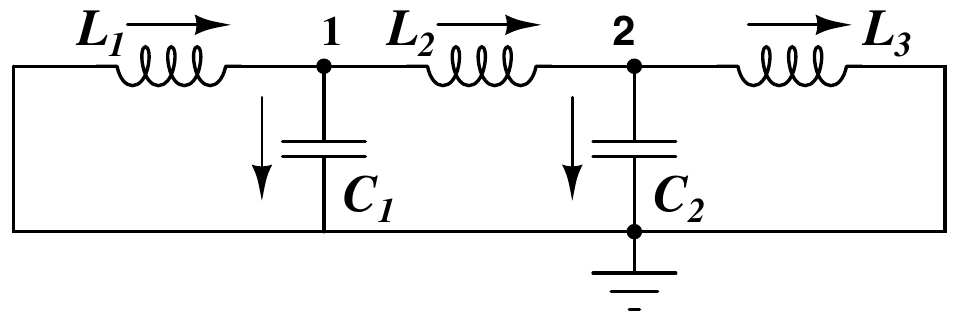}
\end{center}
\caption{LC transmission line.}
\vspace{-10pt}
\label{fig:LCtr}
\end{wrapfigure}
To demonstrate the momentum map preservation properties of the variational integrator, we consider the LC transmission line consisting of a chain of inductors and capacitors as illustrated in Figure \ref{fig:LCtr}.
The matrix Kirchhoff Constraint matrix $K\in \mathbb{R}^{n,m}$ and the Fundamental Loop matrix $K_2\in \mathbb{R}^{n,n-m}$ are (with the third node assumed to be grounded),
\begin{equation*}
K = \left(\begin{array}{cc} -1 & 0\\ 1 & -1 \\ 0 & 1 \\ 1 &0 \\ 0 & 1  \end{array}\right), \quad\quad 
K_2 = \left(\begin{array}{ccc} 1 & 0 & 0\\ 0 & 1 & 0 \\ 0 & 0 & 1 \\ 1 & -1 & 0 \\ 0 & 1 & -1  \end{array}\right).
\end{equation*}
With $n_C=2$ and $K_C =  \left(\begin{array}{cc} 1 & 0 \\ 0 &1  \end{array}\right)$ having full rank, we can follow from Proposition~\ref{prop:degeneracy}, that the reduced Lagrangian system is non-degenerate.
For this circuit, the topology assumption \ref{ass:top} holds, i.e.~all nodes except ground have exactly one branch connected and inward and outward to the node. Thus, from Theorem~\ref{theo:neother} it holds that the sum of the inductor fluxes $p_{L_1}+p_{L_2}+p_{L_3}$ is a conserved quantity that can be also seen from the variational simulation results depicted in Figure~\ref{fig:LCtr_flux}.

\begin{figure}[htb]
 \centering
 \includegraphics[width=0.5\textwidth]{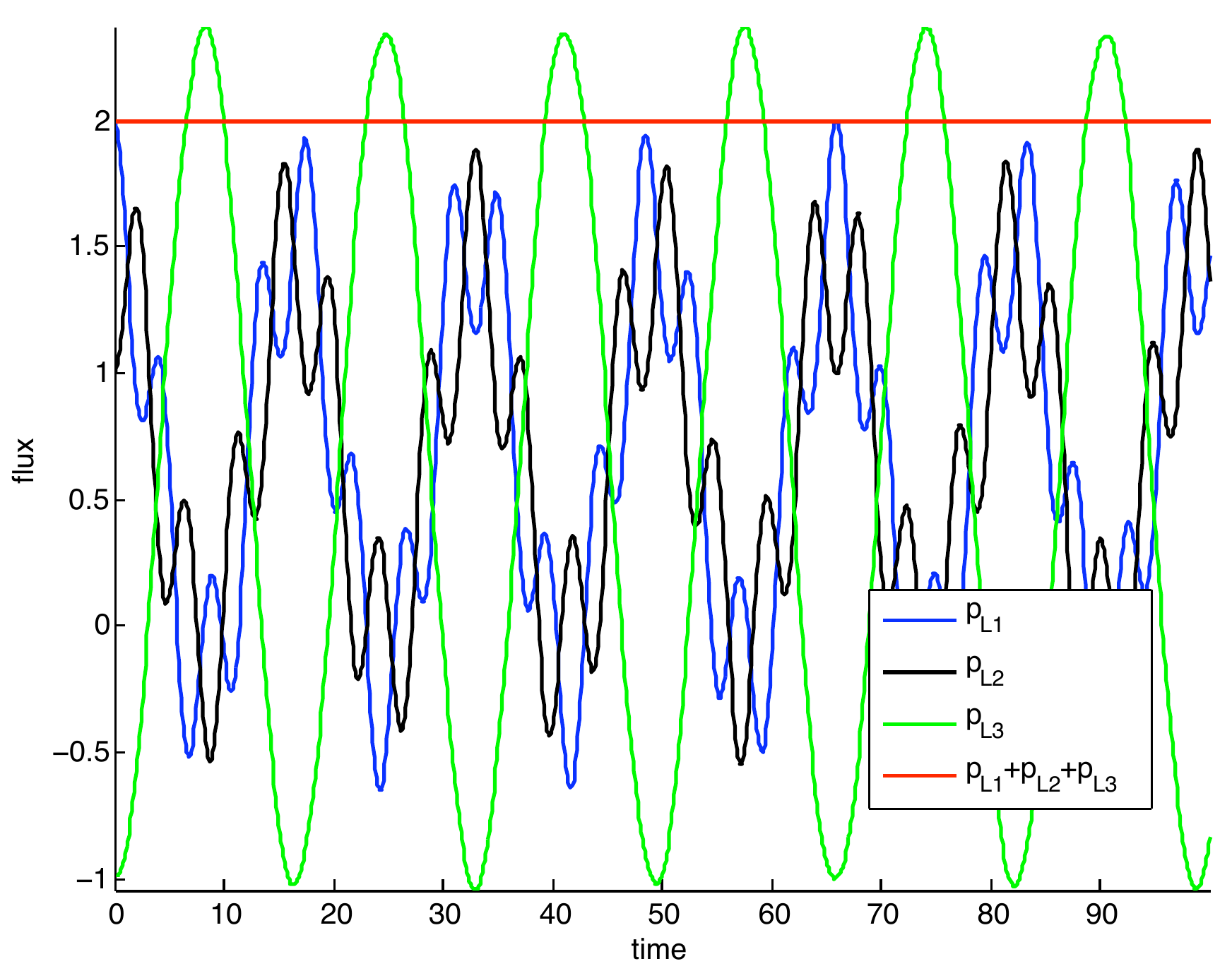} 
\caption{LC transmission line: The sum of the fluxes through all three inductors (blue, black, and green line) is preserved (red line).}
\label{fig:LCtr_flux}
\end{figure}

\subsection{Validation on the stochastic variational integrator}

Consider the stochastic differential equation
\begin{equation}\label{eq:sde}
    dx=Ax dt + \bar{\Sigma} dW_t,
\end{equation}
where $\bar{\Sigma}$ is a $n$-by-$m$ matrix, not necessarily full rank, $x=(x_1,x_2,\ldots,x_n)\in \mathbb{R}^n$, $A\in\mathbb{R}^{n,n}$ and $W_t$ is an $m$-dimensional Brownian motion (with independent components).
A common way to verify the quality of a numerical solution is to consider statistical moments of the solution; in particular, we focus on the expectation and the variance, i.e., $\mathbb{E}(x(t))$ and $\mathbb{D}(x(t))=\mathbb{E}(x(t))^2-(\mathbb{E}(x(t)))^2$.

On the analytical side, by Ito's formula (see e.g.~\cite{Ok00}) we have with $B(t)=\exp(At)$
\begin{subequations}\label{eq:expvar}
\begin{align}
    \mathbb{E}(x(t)) &= B(t)x(0) \label{eq:exp}\\
    \mathbb{D}(x(t)) &= \int_0^t B(\tau) \bar{\Sigma} \bar{\Sigma}^T B(\tau)^T d\tau. \label{eq:var}
\end{align}
\end{subequations}
The expectation and the variance can always be computed if $A$ and $\bar{\Sigma}$ are given. For big systems, however, such a mundane computation is quite complex.
On the numerical side, we run an ensemble of simulations (of total number $M$), all starting from the same initial condition but for each simulation an independent set of noise (i.e., different $\xi_k$) is used. The ensemble are indicated by $x^1(t),x^2(t),\ldots, x^M(t)$ where for any $j$, $x^j(t)=(x^j_1(t), x^j_2(t), \ldots, x^j_n(t))$ is a vector. We compute the empirical moments by
\begin{subequations}\label{eq:expvar_d}
\begin{align}
    \bar{\mathbb{E}}(x(t)) &\approx \frac{1}{M} \sum_{j=1}^M x^j(t)  \label{eq:exp_d}\\
    \bar{\mathbb{D}}(x(t)) &\approx \frac{1}{M} \sum_{j=1}^M (x^j(t))^2 - \frac{1}{M^2} \left(\sum_{j=1}^M x^j(t)\right)^2. 
    \label{statMoments}
\end{align}
\end{subequations}
The numerical method is validated if for large enough $M$ the empirical moments \eqref{eq:expvar_d} are close to the analytical ones \eqref{eq:expvar}.

In our setting, we can rewrite the reduced stochastic Euler-Lagrange equations \eqref{ELc_K_noise} in the form of \eqref{eq:sde} with $x=(\tilde{q},\tilde{p}) \in \mathbb{R}^{2(n-m)}$, $\bar{\Sigma} = \left( \begin{array}{cc} 0 & 0\\ 0& K_2^T \Sigma \end{array} \right)\in\mathbb{R}^{2(n-m),2n}$, and the obvious definition of $A\in \mathbb{R}^{2(n-m),2(n-m)}$ with $\Sigma \in \mathbb{R}^{n,n}$ and $K_2\in \mathbb{R}^{n,n-m}$.
The analytical variance matrix $\mathbb{D}((\tilde{q}(t),\tilde{p}(t))\in \mathbb{R}^{2(n-m),2(n-m)}$ for the reduced system can now be calculated using equation \eqref{eq:var}.
The corresponding variance matrix for the full system can then be calculated as
\[
\mathbb{D}(q(t),p(t)) = \left( \begin{array}{cc} K_2 & 0 \\ 0& K_2 \end{array}\right) \mathbb{D}(\tilde{q}(t),\tilde{p}(t))  \left( \begin{array}{cc} K_2^T & 0 \\ 0& K_2^T \end{array}\right) \in \mathbb{R}^{2n,2n}.
\]
As a demonstration, we calculate the empirical and analytical moments for the circuit introduced in Section \ref{subsec:osLC}. For the experiments throughout this section, we defined $\Sigma$ as $4$-by-$4$ diagonal matrix with diagonal entries $\Sigma_{jj}=0.01,\, j=1,\ldots 4$. The step size is $h=0.1$, the integration time for each simulation is $T=30$, and we start with the initial conditions $\tilde{q}_0=(1,0)$, $\tilde{v}_0=(0,0)$, and $\tilde{p}_0=(0,0)$. The empirical averages are calculated over an ensemble of $M=100000$ independent simulations.

\begin{figure}[ht]
 \centering
\begin{tabular}{cc}
\includegraphics[width=0.45\textwidth]{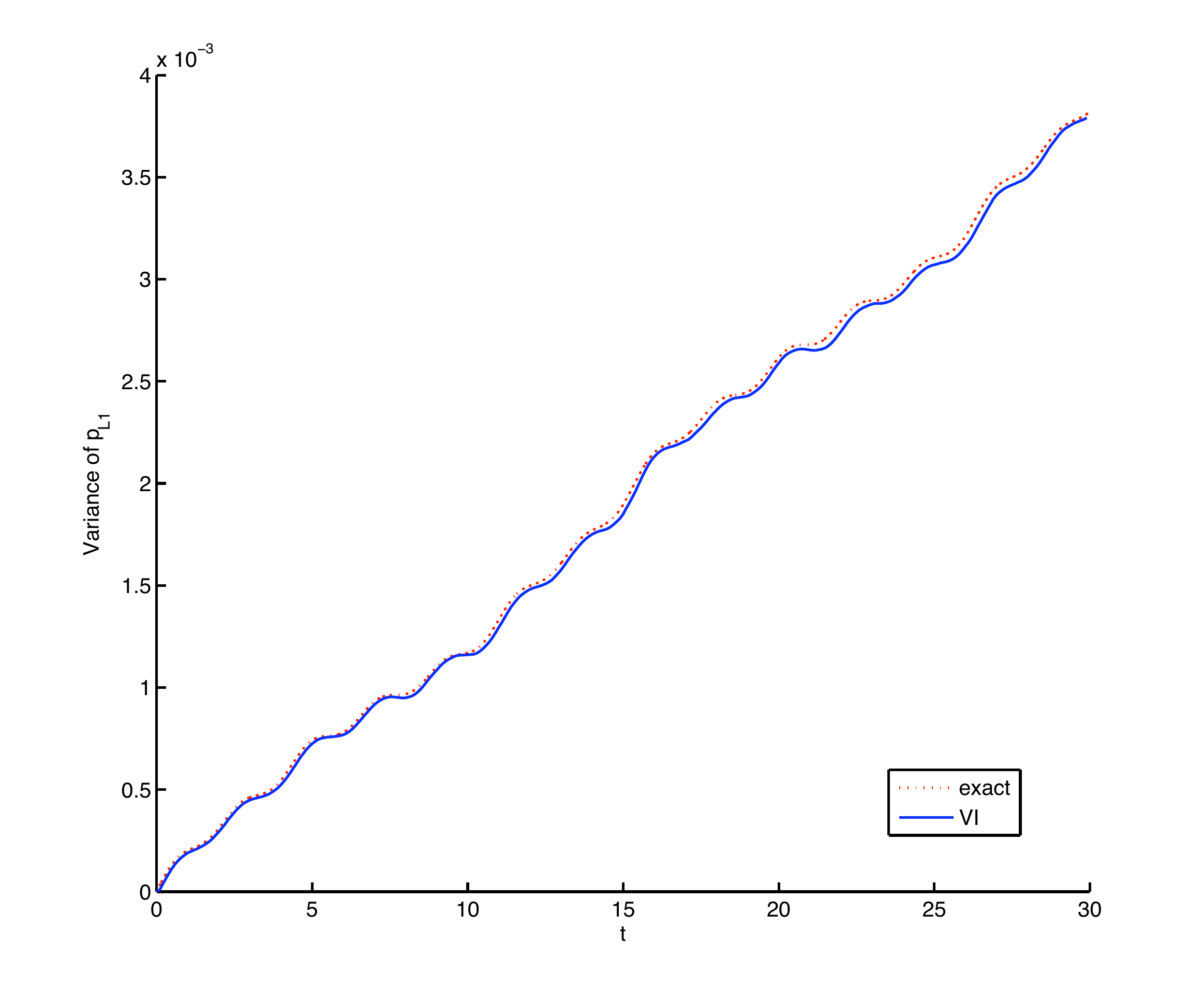} &\includegraphics[width=0.45\textwidth]{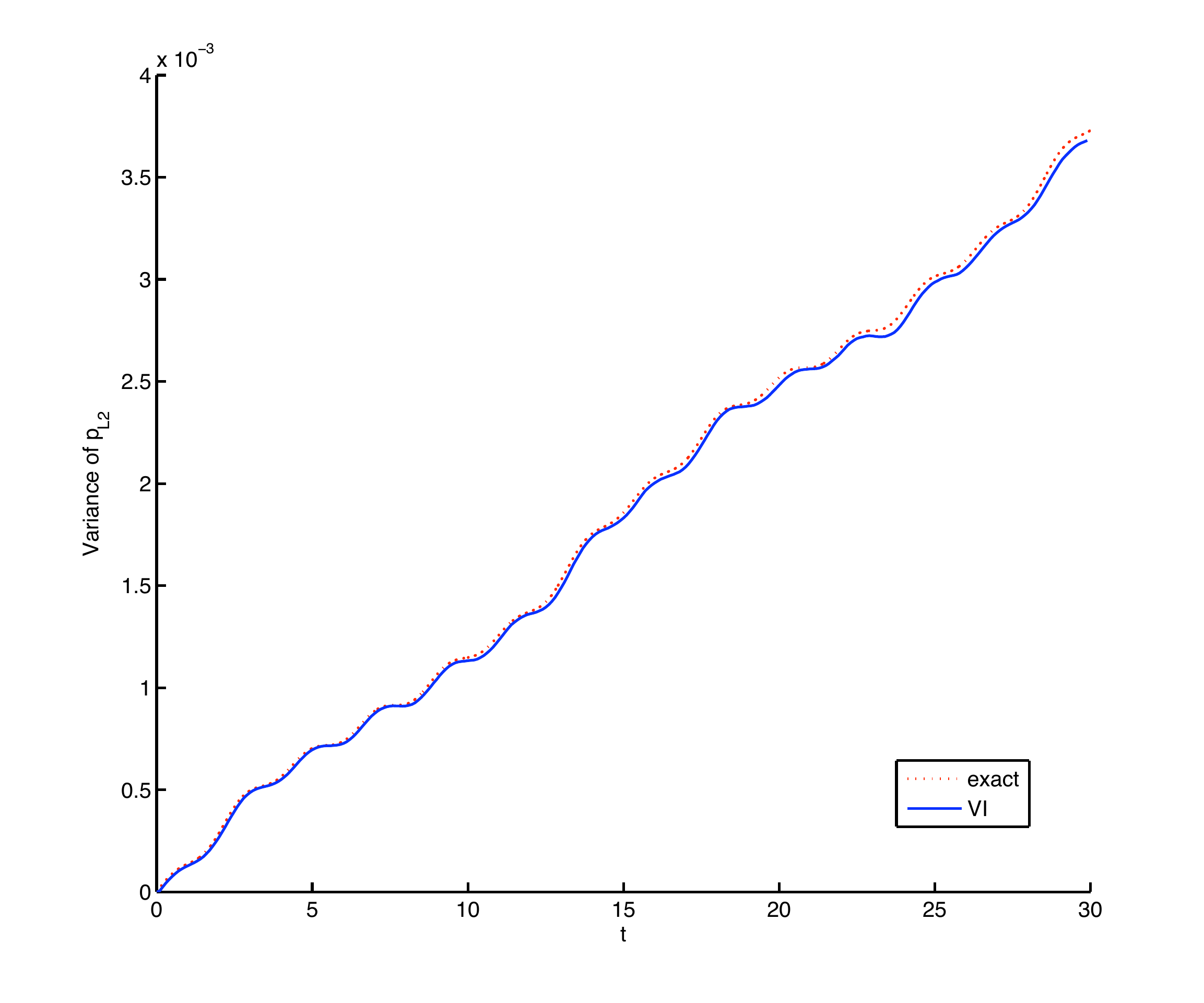} \\
\footnotesize{(a)} & \footnotesize{(b)}
\end{tabular}
 \caption{Benchmark of variances as functions of time according to \eqref{eq:var} (red) and variances as functions of time computed numerically by averaging over an ensemble according to \eqref{statMoments} (blue). a) $\mathbb{D}p_1$ b) $\mathbb{D}p_2$.}
 \label{TrueVariance}
\end{figure}

The analytical variance of $p_{L_1}$ and $p_{L_2}$, i.e., the fifth and sixth diagonal elements of the variance matrix in the full system, are plotted as functions of time (see Figure~\ref{TrueVariance}, red dotted line). Notice, that $p_{L_1}$ and $p_{L_2}$ in our case are just the currents through inductor branch $1$ and $2$, the inductances are $L_1 = L_2 = 1$. 
The result using the stochastic variational integrator is also shown in Figure~\ref{TrueVariance} (blue solid line). Both function shapes and ranges agree very well. 
In particular, all the little bumps in the variance that are subtly different are approximated correctly.
This classical test serves as an evidence that the stochastic integration works fine.

\subsection{Multiscale integration with FLAVORS}

When the circuit exhibits behavior in two time scales, our integrators can be FLAVORized \cite{TaOwMa2010} to capture the slow time scale without resolving the fast time scale to greatly reduce integration time. We first give a brief description of FLAVORS (\emph{FLow AVeraging integratORS}). For more details, we refer to \cite{TaOwMa2010}.

Consider an ordinary differential system on $\mathbb{R}^d$
\begin{equation}\label{eq:fastslow}
\dot{u}^\epsilon =G(u^\epsilon) +\frac{1}{\epsilon} F(u^\epsilon)
\end{equation}
with $\epsilon \ll 1$.
In the context of Lagrangian systems, we consider a multiscale Lagrangian as
\begin{equation}\label{eq:Lfastslow}
\mathcal{L}(q,v) = \frac{1}{2} v^T L v - V(q) - \frac{1}{\epsilon}U(q)
\end{equation}
where $V(q)$ is denoted as ``slow'' and $ \frac{1}{\epsilon}U(q)$ denoted as ``fast'' potential.
In the case of a linear circuit, the slow potential corresponds to the charge potential of a capacitor with high capacitance, whereas the fast potential corresponds to a capacitor with very small capacitance. For instance, consider the oscillating LC circuit in Section~\ref{subsec:osLC} with $C_1=1$ and $C_2=\epsilon$ ($\epsilon=1$ in Section~\ref{subsec:osLC}).
The corresponding potential in \eqref{eq:Lfastslow} can be written as the sum of slow and fast potential as
\[V(q) = \frac{1}{2}q_{C_1}^2, \quad U(q)= \frac{1}{2}q_{C_2}^2.\]
The smaller $\epsilon$ (i.e. $C_2$) is, the wider the two time scales will be separated.

FLAVORS are based on the averaging
of the instantaneous flow of the differential equation \eqref{eq:fastslow} with hidden slow and fast variables. The way to FLAVORize any of our integrators for circuits is to break each timestep into a composition of two substeps. The first one uses a timestep of length $\tau$, and the second one uses a timestep $\delta-\tau$. $\tau$ has to be small enough to resolve the stiffness, but $\delta-\tau$ does not. In fact, in the first substep one integrates the entire system with the original value of the stiffness $1/\epsilon$, whereas in the second substep one integrates the system with the stiffness turned off, i.e., $1/\epsilon$ is temporarily set to $0$. Of course, for the first substep of the next step, $1/\epsilon$ has to be restored to its original big value again. To be more precise, FLAVOR is implemented using an arbitrary legacy integrator $\Phi_h^{\frac{1}{\epsilon}}$ for \eqref{eq:fastslow} in which the parameter $\frac{1}{\epsilon}$ can be controlled.
By switching on and off the stiff parameter, FLAVOR approximates the flow of $\eqref{eq:fastslow}$ over a
coarse time step $H$ (resolving the slow time scale) by the flow
\[
\Phi_H := \left( \Phi^0_{\frac{H}{M}-\tau} \circ \Phi_\tau^{\frac{1}{\epsilon}}   \right)^M,
\]
where $\tau$ is a fine time step resolving the ``fast'' time scale ($\tau \ll \epsilon$) and $M$ is a positive integer corresponding to the number of ``samples'' used to average the flow ($\delta=H/M$).
Since FLAVORS are obtained by flow composition, they inherit the structure-preserving properties (for instance, symplecticity and symmetries under a group action) of the legacy integrator for Hamiltonian systems.

Theorem 1.4 in \cite{TaOwMa2010} guarantees the accuracy of FLAVORS for $\delta \ll h_0$, $\tau \ll \epsilon$ and $\left(\frac{\tau}{\epsilon}\right)^2 \ll \delta \ll \frac{\tau}{\epsilon}$, where $h_0$ is the stability limit of step length for the legacy integrator. Furthermore, if the hidden fast and slow variables are affine functions of the original variables (such as in our case of linear circuit), the condition relaxes to $\delta \ll h_0$, $\tau \ll \epsilon$ and $\delta \ll \frac{\tau}{\epsilon}$. A intuitive interpretation of that theorem is, the numerically integrated slow variable will convergence strongly (as a function of time) to the solution to an averaged effective equation, and the fast variable will converge weakly to its local ergodic measure.

We use FLAVORS to simulate the LC circuit with $\epsilon=10^{-3}$, $\tau=0.1\epsilon=10^{-4}$, $H = 0.1$ and $M=100$.
The charges and currents as functions of time are plotted in Figure~\ref{multiscale}. Notice that the slow components in the solution are captured strongly, but the fast components may have altered wave shapes: for instance, Figure~\ref{multiscale2} shows a zoomed-in investigation of the current through the second branch, which is a superposition of a slow global oscillation and a fast local oscillation; the slow one is obviously well-captured in the usual sense, and the fast one is captured in the less-commonly-used sense of averaging.

\begin{figure}[ht]
 \centering
\includegraphics[width=\textwidth]{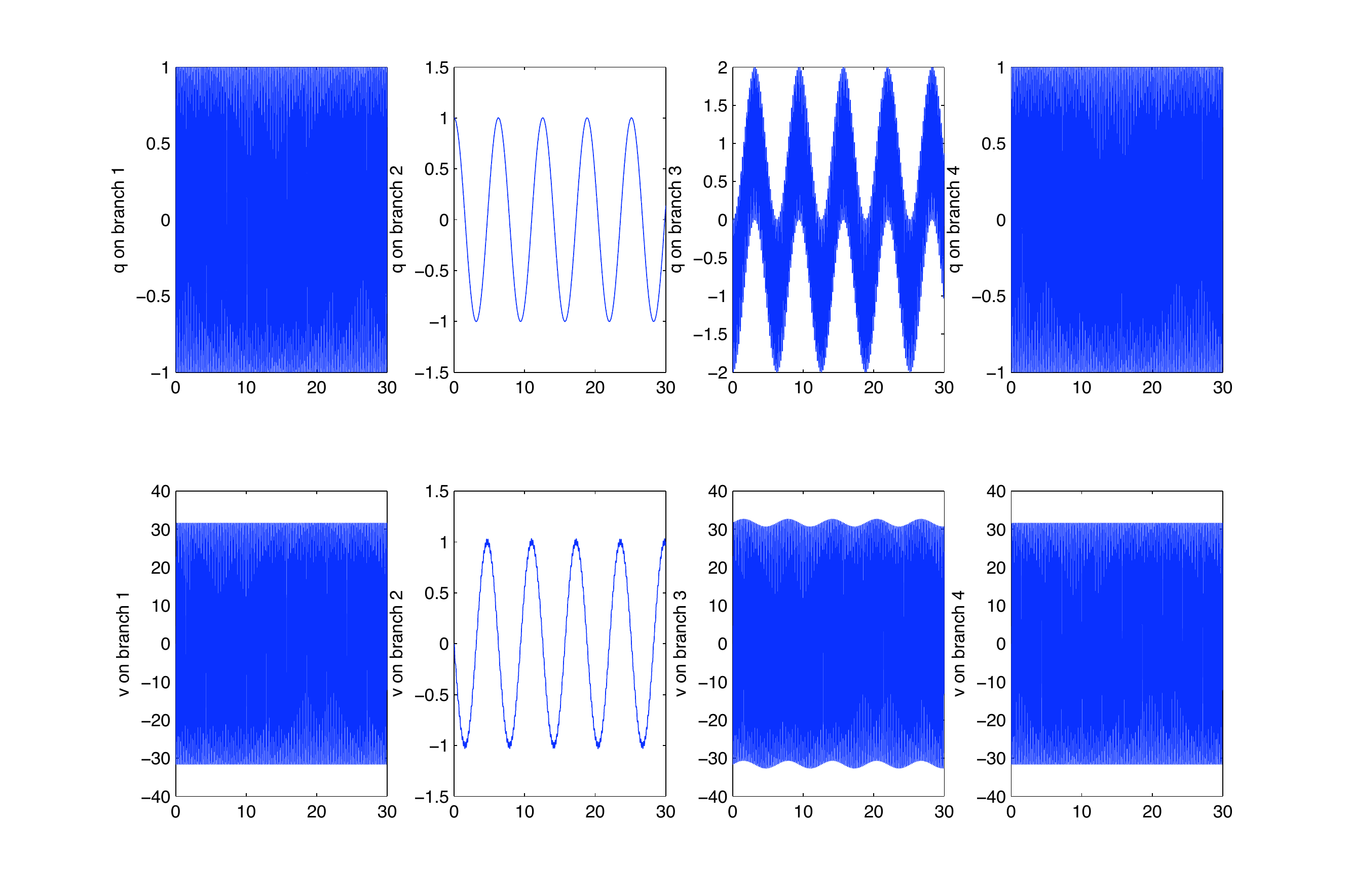} \\
\footnotesize{(a)} \\
\includegraphics[width=\textwidth]{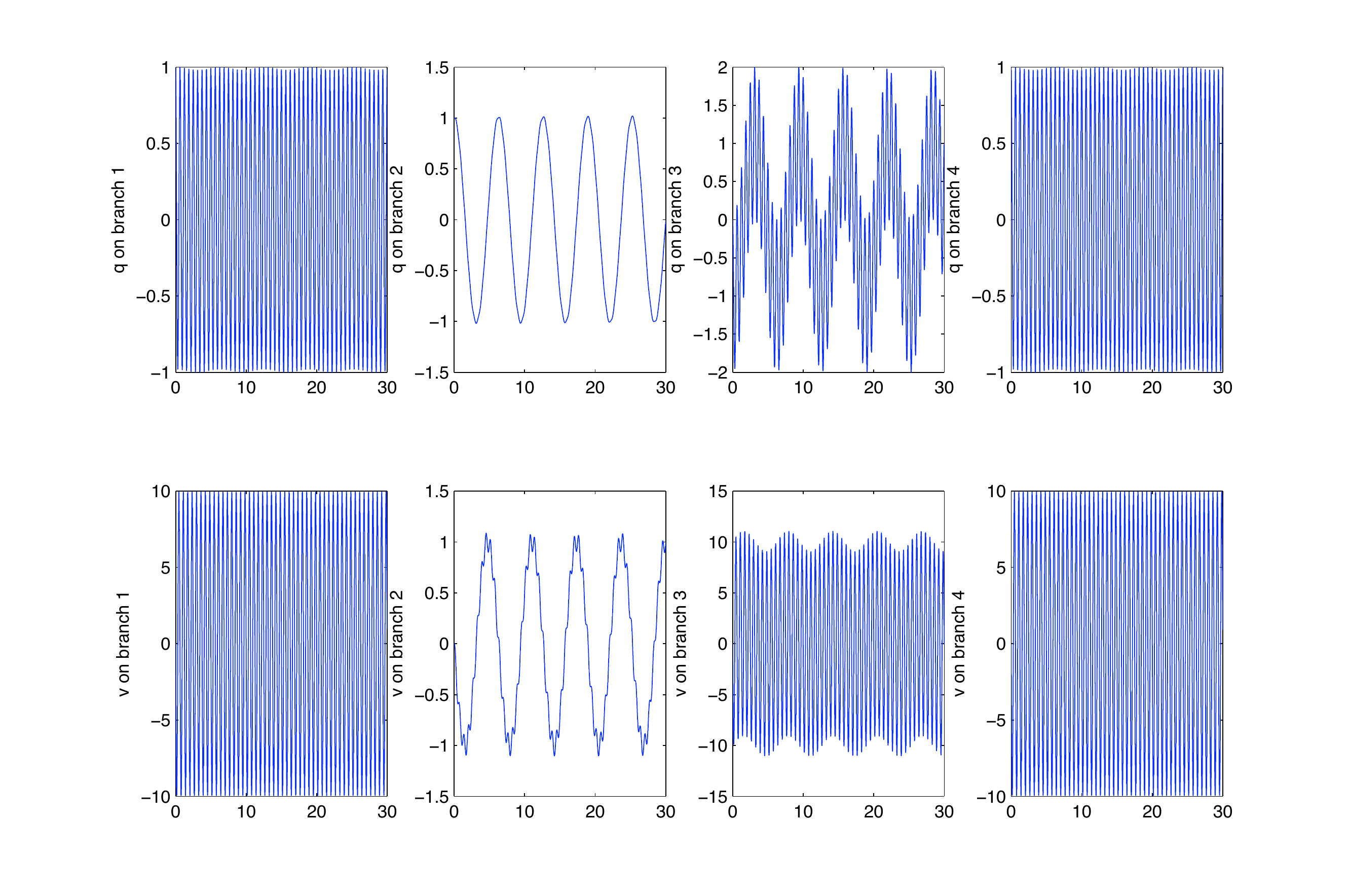} \\
\footnotesize{(b)}
 \caption{Simulations of a multiscale system: a) Benchmark solution computed with a variational integrator ($h=10^{-4}$) b) FLAVOR with $\tau =10^{-4}, \delta =10^{-3}$ and $\epsilon=10^{-3}$.}
 \label{multiscale}
\end{figure}

\begin{figure}[ht]
 \centering
 \begin{tabular}{cc}
\includegraphics[width=0.5\textwidth]{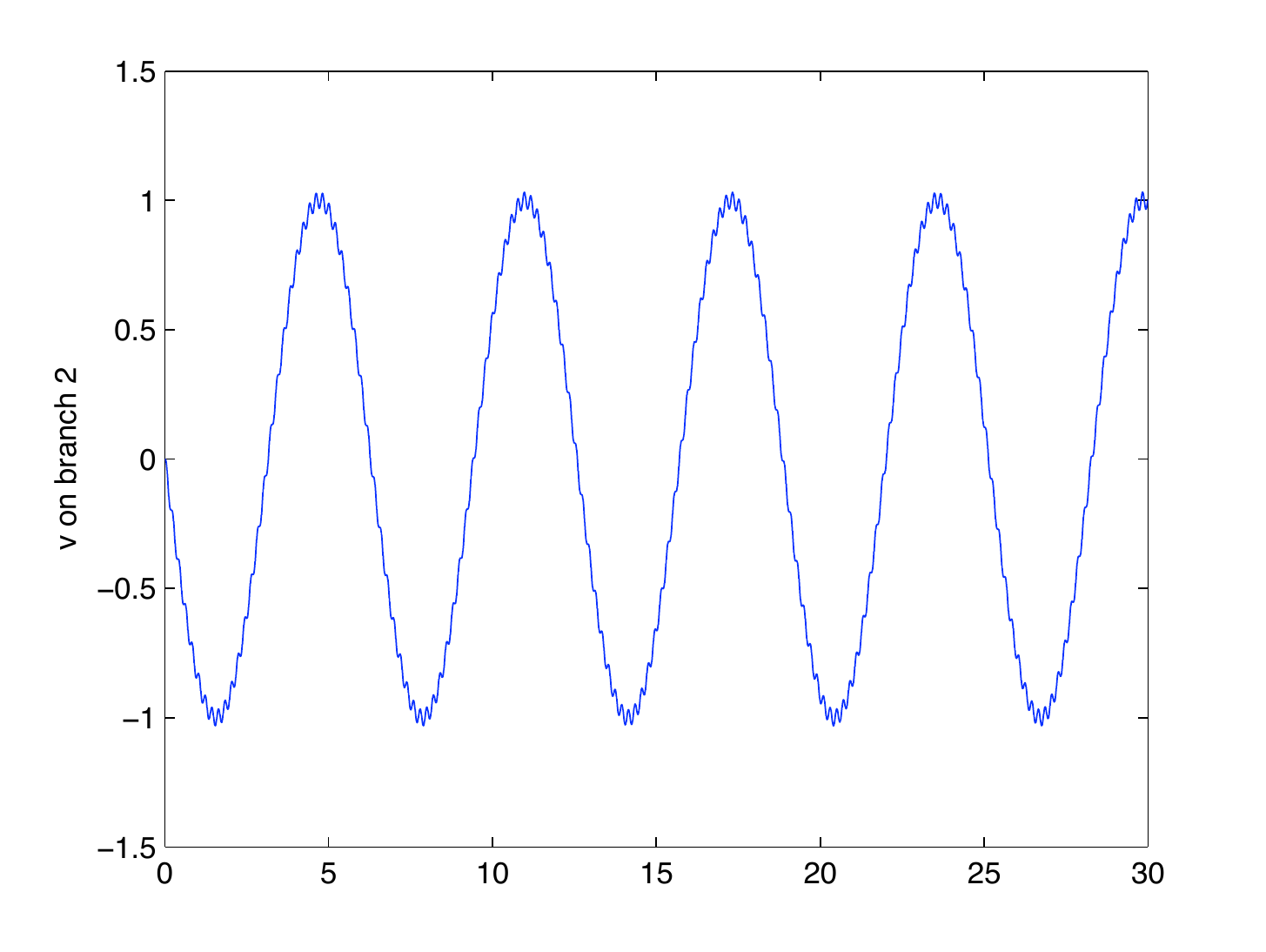} &\includegraphics[width=0.5\textwidth]{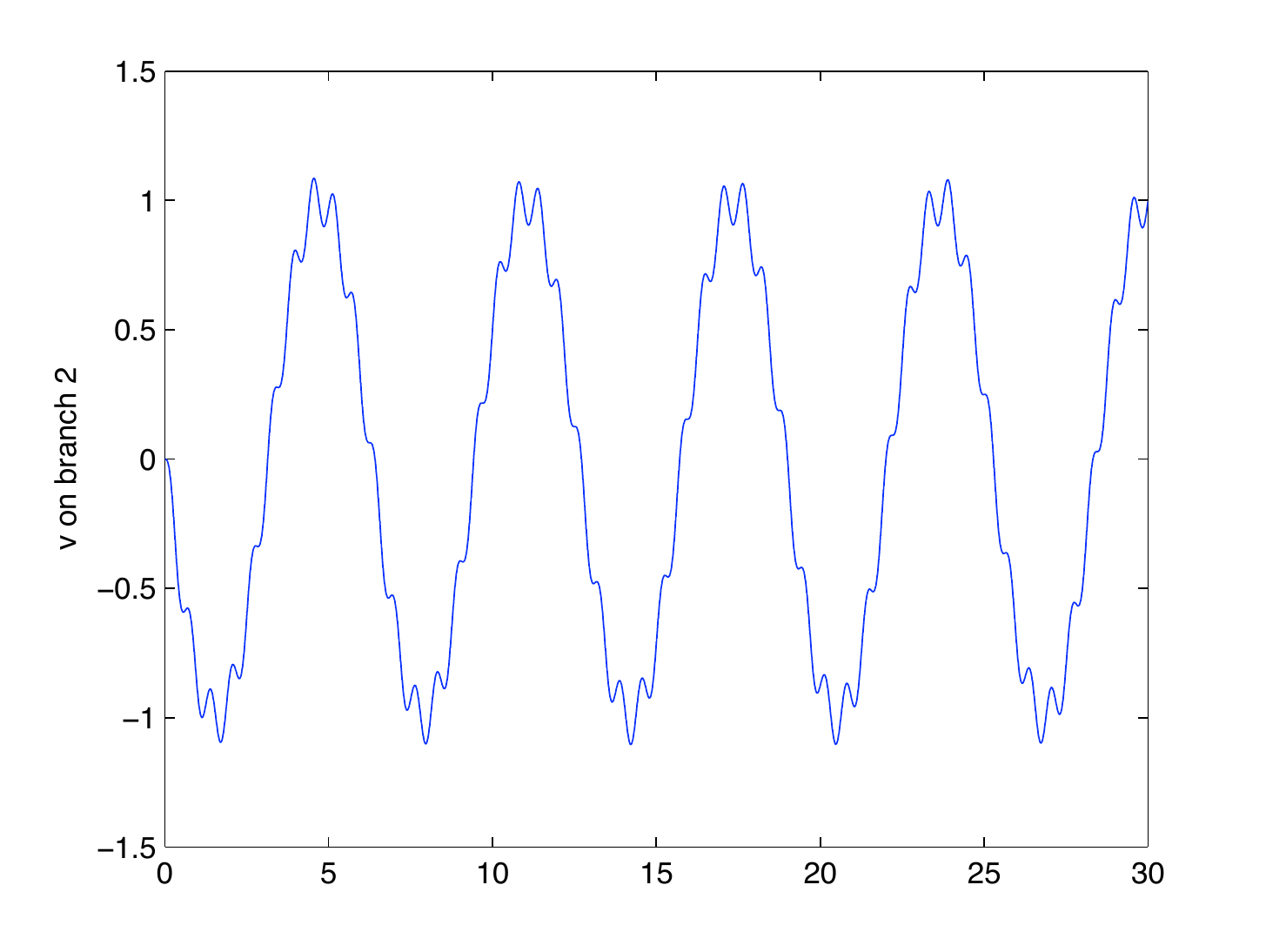} \\
\footnotesize{(a)} & \footnotesize{(b)}
\end{tabular}
 \caption{Simulations of a multiscale system: a) Benchmark solution computed with a variational integrator ($h=10^{-4}$) b) FLAVOR with $\tau =10^{-4}, \delta =10^{-3}$ and $\epsilon=10^{-3}$.}
 \label{multiscale2}
\end{figure}

\section{Conclusions}

In this contribution, we presented a unified framework for the modeling and simulation of electric circuits.
Starting with a geometric setting, we formulate a unified variational formulation for the modeling of electric circuits. 
Analogous to the formulation of mechanical systems, we define a degenerate Lagrangian on the space of branches consisting of electric and magnetic energy, dissipative and external forces that describe the influence of resistors and voltage sources as well as (non-)holonomic constraints given by the KCL of the circuit. The Lagrange-d'Alembert-Pontryagin principle is used to derive in a variational way the implicit Euler-Lagrange equations being differential-algebraic equations to describe the system's dynamics. A reduced version on the space of meshes is presented that is shown to be equivalent to the original system and for which under some topology assumptions the degeneracy of the Lagrangian is canceled.   

Based on the reduced version, a discrete variational approach is presented that provides different variational integrators for the simulation of circuits. In particular, the generated integrators are symplectic, preserve momentum maps in presence of symmetries and have good long-time energy bahvior. Furthermore, we observe that the spectrum of high frequencies is especially better preserved compared to simulations using Runge-Kutta or BDF methods. Having the variational framework for the model and the simulation, extensions of the approach using already-existing types of different variational integrators can be easily accomplished. As an example, we presented the extension for the simulation of noisy circuits using stochastic variational integrator approaches as well as multiscale methods (in particular FLAVORS) for an efficient treatment of circuits with multiple time scales. 

In the future, we will extend the approach to the simulation and analysis of more complicated nonlinear and magnetic circuits that might include nonlinear inductors, capacitors, resistors and transistors. Since a variational formulation in terms of energies, forces and constraints is still valid for the nonlinear case, the presented integrators will be derived and applied in straight forward way. Furthermore, the inclusion of controlled sources allows for the consideration of optimal control problems for circuits for which techniques also based on a variational formulation can be easily applied (see e.g.~\cite{DMOC}).
The variational simulation of combined mechanical and electric (electro-mechanical) systems is the natural next step towards the development of a unified variational modeling and simulation method for mechatronic systems. 
Furthermore, at nanoscales, thermal noise and electromagnetic interactions become an essential component of the dynamic of electric circuits. We plan to investigate the coupling of variational integrators for circuits with multi-symplectic variational integrators for EM fields and continuum mechanics (see e.g.~\cite{Lew03phd,MPSW01,StToDeMa2009}) to produce a robust structure-preserving numerical integrator for Microelectromechanical and Nanoelectromechanical systems.
Recently, Mike Giles has developed a Multilevel Monte Carlo method for differential equations with stochastic forcings \cite{Giles08} that shows huge computation accelerations, ~100 times in some cases. The extension of the current method to multilevel stochastic variational integrators is straight forward and may further accelerate computation dramatically, especially for multiscale problems,  while preserving certain properties of the circuit network.

\section{Acknowledgement}

This contribution was partly developed and published in the course of the Collaborative Research Centre 614 ``Self-Op\-ti\-mi\-zing Concepts and Structures in Mechanical Engineering" funded by the German Research Foundation (DFG) under grant number SFB 614. The authors acknowledge partial support from NSF grant CMMI-092600.
The authors gratefully acknowledge Henry Jacobs, Melvin Leok, and Hiroaki Yoshimura for delightful discussions about variational mechanics for degenerate systems. Furthermore, the authors thank Stefan Klus, Sujit Nair, Olivier Verdier, and Hua Wang for helpful discussions regarding circuit theory. Finally, we thank Sydney Garstang for proofreading the document.





\begin{thebibliography}{10}

\bibitem{Baec07}
S.~B{\"a}chle.
\newblock {\em Numerical solution of differential-algebraic systems arising in
  circuit simulation}.
\newblock PhD thesis, Technische Universit\"at Berlin, 2007.

\bibitem{BL89}
G.~M.\ Bernstein and M.~A.\ Liebermann.
\newblock A method for obtaining a canonical {H}amiltonian for nonlinear {LC}
  circuits.
\newblock {\em IEEE Transactions on Circuits and Systems}, 36(3):411--420,
  1998.

\bibitem{Bl00}
G.\ Blankenstein.
\newblock {\em Implicit {H}amiltonian Systems: Symmetry and Interconnection}.
\newblock PhD thesis, University of Twente, 2000.

\bibitem{BC97}
A.~M. Bloch and P.~R. Crouch.
\newblock Representations of {D}irac structures on vector spaces and nonlinear
  {LC}-circuits.
\newblock In {\em Differential Geometry and Control (Boulder, CO, 1997)},
  volume~64 of {\em Proceedings of Symposia in Pure Mathematics}, pages
  103--117, Providence, RI, 1997. American Mathematical Society.

\bibitem{Bou07}
N.\ Bou-Rabee.
\newblock {\em Hamilton-{P}ontryagin Integrators on {L}ie Groups}.
\newblock PhD thesis, California Institute of Technology, 2007.

\bibitem{BRO08}
N.\ Bou-Rabee and H.\ Owhadi.
\newblock Stochastic variational integrators.
\newblock {\em IMA Journal of Numerical Analysis}, 29:421--443, 2008.

\bibitem{BRO10}
N.\ Bou-Rabee and H.\ Owhadi.
\newblock Long-run accuracy of variational integrators in the stochastic
  context.
\newblock {\em SIAM Journal of Numerical Analysis}, 48(1):278--297, 2010.

\bibitem{CvS07}
J.~Cervera, A.~J. van~der Schaft, and A.~Ba\~{n}os.
\newblock Interconnection of port-{H}amiltonian systems and composition of
  {D}irac structures.
\newblock {\em Automatica}, 43:212--225, February 2007.

\bibitem{Chua74}
L.~O.\ Chua and J.~D.\ McPherson.
\newblock Explicit topological formulations of {L}agrangian and {H}amiltonian
  equations for nonlinear networks.
\newblock {\em IEEE Transactions on Circuits and Systems}, 21(2):277--286,
  March 1974.

\bibitem{ClSch2003}
J.~Clemente-Gallardo and J.~M.~A.\ Scherpen.
\newblock Relating {L}agrangian and {H}amiltonian formalisms of {LC} circuits.
\newblock {\em IEEE Transactions on Circuits and Systems - I: Fundamental
  Theory and Applications}, 50(10):1359--1363, October 2003.

\bibitem{FMOW03}
R.~C.\ Fetecau, J.~E.\ Marsden, M.\ Ortiz, and M.\ West.
\newblock Nonsmooth {L}agrangian mechanics and variational collision
  integrators.
\newblock {\em SIAM J.\ Applied Dynamical Systems}, 2(3):381--416, 2003.

\bibitem{Giles08}
M.~B. Giles.
\newblock Multilevel {M}onte {C}arlo path simulation.
\newblock {\em Operations Research}, 56(3):607--617, 2008.

\bibitem{GY04}
J.L. Gross and J.~Yellen.
\newblock {\em Handbook of Graph Theory}.
\newblock Discrete Mathematics and its Applications. CRC Press, 2004.

\bibitem{GFtM05}
M.~G{\"u}nther, U.~Feldmann, and J.~ter Maten.
\newblock Modelling and discretization of circuit problems.
\newblock Technical report, OAI Repository of the Technische Universiteit
  Eindhoven, 2005.

\bibitem{HaLuWa}
E.\ Hairer, C.\ Lubich, and G.\ Wanner.
\newblock {\em Geometric numerical integration}, volume~31 of {\em Springer
  Series in Computational Mathematics}.
\newblock Springer, 2002.

\bibitem{HW91}
E.~Hairer and G.~Wanner.
\newblock {\em Solving Ordinary Differential Equations II: Stiff and
  Differential-Algebraic Problems}.
\newblock Springer, 1991.

\bibitem{JYM10}
H.~Jacobs, Y.~Yoshimura, and J.~E.\ Marsden.
\newblock Interconnection of {L}agrange-{D}irac dynamical systems for electric
  circuits.
\newblock In {\em 8th International Conference of Numerical Analysis and
  Applied Mathematics}, volume 1281, pages 566--569, 2010.
\newblock DOI:10.1063/1.3498539.

\bibitem{Kane00}
C.\ Kane, J.~E.\ Marsden, M.\ Ortiz, and M.\ West.
\newblock Variational integrators and the {N}ewmark algorithm for conservative
  and dissipative mechanical systems.
\newblock {\em International Journal for Numerical Methods in Engineering},
  49(10):1295--1325, 2000.

\bibitem{KoMa2010}
M.\ Kobilarov, J.~E.\ Marsden, and G.~S.\ Sukhatme.
\newblock Geometric discretization of nonholonomic systems with symmetries.
\newblock {\em Discrete and Continuous Dynamical Systems - Series S},
  1(1):61--84, 2010.

\bibitem{KM06}
P.~Kunkel and V.~Mehrmann.
\newblock {\em Differential-Algebraic Equations}.
\newblock EMS Textbooks in Mathematics. European Mathematical Society, 2006.

\bibitem{KwMaBa82}
H.~G.\ Kwatny, F.~M.\ Massimo, and L.~Y.\ Bahar.
\newblock The generalized {L}agrange formulation for nonlinear {RLC} networks.
\newblock {\em IEEE Transactions on Circuits and Systems}, 29(4):220--233,
  April 1983.

\bibitem{LeOh2010}
M.~Leok and T.~Ohsawa.
\newblock Discrete {D}irac structures and implicit discrete {L}agrangian and
  {H}amiltonian systems.
\newblock In {\em GEOMETRY AND PHYSICS: XVIII International Fall Workshop on
  Geometry and Physics}, volume 1260, pages 91--102. AIP Conference
  Proceedings, 2010.
\newblock DOI:10.1063/1.3479325.

\bibitem{Lew03phd}
A.~Lew.
\newblock {\em Variational time integrators in computational solid mechanics}.
\newblock PhD thesis, California Institute of Technology, 2003.

\bibitem{LMOW04}
A.\ Lew, J.~E.\ Marsden, M.\ Ortiz, and M.\ West.
\newblock An overview of variational integrators.
\newblock In L.~P.\ Franca, T.~E.\ Tezduyar, and A.\ Masud, editors, {\em
  {F}inite {E}lement {M}ethods: 1970's and Beyond}, pages 98--115. CIMNE, 2004.

\bibitem{LMOWe04}
A.\ Lew, J.~E.\ Marsden, M.\ Ortiz, and M.\ West.
\newblock Variational time integrators.
\newblock {\em International Journal for Numerical Methods in Engineering},
  60(1):153--212, 2004.

\bibitem{leyendecker07-2}
S.\ Leyendecker, J.~E.\ Marsden, and M.\ Ortiz.
\newblock Variational integrators for constrained dynamical systems.
\newblock {\em Journal of Applied Mathematics and Mechanics}, 88(9):677--708,
  2008.

\bibitem{DMOCC}
S.\ Leyendecker, S.\ Ober-Bl\"obaum, J.~E.\ Marsden, and M.\ Ortiz.
\newblock Discrete mechanics and optimal control for constrained systems.
\newblock {\em Optimal Control, Applications and Methods}, 2009.
\newblock DOI: 10.1002/oca.912.

\bibitem{MPSW01}
J.~E.\ Marsden, S.\ Pekarsky, S.\ Shkoller, and M.\ West.
\newblock Variational methods, multisymplectic geometry and continuum
  mechanics.
\newblock {\em Journal of Geometry and Physics}, 38(3-4):253--284, 2001.

\bibitem{MarRat94}
J.~E.\ Marsden and T.\ Ratiu.
\newblock {\em Introduction to Mechanics and Symmetry}, volume~17 of {\em Texts
  in Applied Mathematics}.
\newblock Springer, 1994.

\bibitem{MaWe01}
J.~E.\ Marsden and M.\ West.
\newblock Discrete mechanics and variational integrators.
\newblock {\em Acta Numerica}, 10:357--514, 2001.

\bibitem{MvSB95}
B.~M. Maschke, A.~J. van~der Schaft, and P.~C. Breeveld.
\newblock An intrinsic {H}amiltonian formulation of the dynamics of
  {LC}-circuits.
\newblock {\em IEEE Transactions on Circuits and Systems I: Fundamental Theory
  and Applications}, 42(2):73--82, 1995.

\bibitem{MvdS02}
L.\ Moreau and A.~J. van~der Schaft.
\newblock Implicit {L}agrangian equations and mathematical modeling of physical
  systems.
\newblock In {\em 41st IEEE Conference on Decision and Control}, Las Vegas,
  Nevada, USA, 2002.

\bibitem{Nils2005}
J.~W.\ Nilsson and S.~A.\ Riedel.
\newblock {\em Electric Circuits}.
\newblock seventh edition, Prentice Hall, 2005.

\bibitem{DMOC}
S.\ Ober-Bl\"obaum, O.\ Junge, and J.~E.\ Marsden.
\newblock Discrete mechanics and optimal control: an analysis.
\newblock {\em Control, Optimisation and Calculus of Variations}, 2010.
\newblock DOI: 10.1051/cocv/2010012.

\bibitem{Ok00}
B.~{\O}ksendal.
\newblock {\em Stochastic Differential Equations. An Introduction with
  Applications}.
\newblock Springer, 5th edition, corrected 2nd printing edition, 2000.
\newblock ISBN 3-540-63720-6.

\bibitem{RoMa02}
C.~W.\ Rowley and J.~E.\ Marsden.
\newblock Variational integrators for degenerate {L}agrangians, with
  application to point vortices.
\newblock In {\em 41st IEEE Conference on Decision and Control}, volume~2,
  pages 1521--1527, 2002.

\bibitem{StToDeMa2009}
A.\ Stern, Y.\ Tong, M.\ Desbrun, and J.~E.\ Marsden.
\newblock Geometric computational electrodynamics with variational integrators
  and discrete differential forms.
\newblock Preprint, 2009.

\bibitem{Sz79}
A.~Szatkowski.
\newblock Remark on ``explicit topological formulation of {L}agrangian and
  {H}amiltonian equations for nonlinear networks''.
\newblock {\em IEEE Transactions on Circuits and Systems}, 26(5):358--360,
  1979.

\bibitem{TaOwMa2010}
M.\ Tao, H.\ Owhadi, and J.~E.\ Marsden.
\newblock Nonintrusive and structure preserving multiscale integration of stiff
  odes, sdes, and {H}amiltonian systems with hidden slow dynamics via flow
  averaging.
\newblock {\em Multiscale Modeling and Simululation}, 8(4):1269--1324, 2010.

\bibitem{Voigtmann06}
S.\ Voigtmann.
\newblock {\em General linear methods for integrated circuit design}.
\newblock PhD thesis, Humboldt-Universit\"at zu Berlin, 2006.

\bibitem{YoMa2006c}
H.~Yoshimura and J.~E.\ Marsden.
\newblock Dirac structures and implicit {L}agrangian systems in electric
  networks.
\newblock In {\em 17th International Symposium on Mathematical Theory of
  Networks and Systems}, pages 1444--1449, 2006.

\bibitem{YoMa2006a}
H.\ Yoshimura and J.~E.\ Marsden.
\newblock Dirac structures and {L}agrangian mechanics. {P}art {I}: Implicit
  {L}agrangian systems.
\newblock {\em Journal of Geometry and Physics}, 57:133--156, 2006.

\bibitem{YoMa2006b}
H.\ Yoshimura and J.~E.\ Marsden.
\newblock Dirac structures and {L}agrangian mechanics. {P}art {II}:
  {V}ariational structures.
\newblock {\em Journal of Geometry and Physics}, 57:209--250, 2006.

\end{thebibliography}







\end{document}